\documentclass[12pt,twoside]{article} 
\usepackage[utf8]{inputenc}
\usepackage{amsmath, amsthm, amssymb}
\usepackage[margin=1in]{geometry}
\usepackage{enumerate}
\usepackage{mathrsfs}
\usepackage{mathtools}
\usepackage{tikz-cd}
\usepackage{xcolor}
\usepackage{graphicx}
\usepackage[shortlabels]{enumitem}
\usepackage{wrapfig}

\numberwithin{equation}{section}
\numberwithin{figure}{section}

\newtheorem{theorem}{Theorem}[section]

\newtheorem{proposition}[theorem]{Proposition}
\newtheorem{corollary}[theorem]{Corollary}
\newtheorem{lemma}[theorem]{Lemma}

\newtheorem*{problem*}{Problem}
\newtheorem*{theorem*}{Theorem}
\newtheorem*{proposition*}{Proposition}
\newtheorem*{corollary*}{Corollary}
\newtheorem*{lemma*}{Lemma}
\newtheorem*{claim*}{Claim}
\newtheorem*{solution*}{Solution}

\theoremstyle{definition}
\newtheorem{notation}[theorem]{Notation}
\newtheorem{definition}[theorem]{Definition}

\newtheorem{remark}[theorem]{Remark}
\newtheorem*{remark*}{Remark}
\newtheorem*{summary*}{Summary}
\newtheorem*{example*}{Example}

\newcommand{\R}{{\mathbb R}}
\newcommand{\Q}{{\mathbb Q}}
\newcommand{\N}{{\mathbb N}}
\newcommand{\C}{{\mathbb C}}
\newcommand{\Z}{{\mathbb Z}}
\renewcommand{\P}{{\mathbb P}}
\renewcommand{\arraystretch}{1.2}

\DeclarePairedDelimiter\ceil{\lceil}{\rceil}
\DeclarePairedDelimiter\floor{\lfloor}{\rfloor}

\usepackage{amsfonts}
\usepackage{makeidx}
\usepackage{graphicx}
\usepackage{physics}
\usepackage{array}
\usepackage{float}

\usepackage{amsmath}

\newcommand{\Oc}{{\mathbb O}}
\newcommand{\T}{{\mathbb T}}
\newcommand{\I}{{\mathbb I}}
\newcommand{\D}{{\mathbb D}}

\newcommand{\SU}{\text{SU}}
\newcommand{\SO}{\text{SO}}

\newcommand{\Sp}{\text{Sp}}
\newcommand{\FS}{\text{FS}}
\newcommand{\CZ}{\operatorname{CZ}}
\newcommand{\fp}{\mathfrak{p}}
\newcommand{\fP}{\mathfrak{P}}

\newcommand{\tb}{\textcolor{blue}}

\newcommand{\wind}{\operatorname{wind}}
\newcommand{\ind}{\operatorname{ind}}
\newcommand{\M}{\mathcal{M}}

\definecolor{pink}{RGB}{255,105,180}

\usepackage{hyperref}
\usepackage{subcaption}
\hypersetup{colorlinks}

\definecolor{indigo}{RGB}{51,0,102}
\definecolor{brightpurple}{RGB}{102,0,153}
\definecolor{fuchsia}{RGB}{180,51,180}
\definecolor{jolightpurple}{RGB}{188,171,240}

\hypersetup{colorlinks,
linkcolor=brightpurple,
filecolor=brightpurple,
urlcolor=indigo,
citecolor=fuchsia}

\title{A contact McKay correspondence for links of simple singularities}
\author{Leo Digiosia and Jo Nelson}
\date{}

\begin{document}

\maketitle

\begin{abstract}
We compute the cylindrical contact homology of the links of the simple singularities. These manifolds are contactomorphic to $S^3/G$ for finite subgroups $G\subset\SU(2)$. We perturb the degenerate contact form on $S^3/G$ with a Morse function, which is invariant under the corresponding $H\subset\SO(3)$ action on $S^2$, to achieve nondegeneracy up to an action threshold. The cylindrical contact homology is recovered by taking a direct limit of the action filtered homology groups. The ranks of this homology are given in terms of $|\text{Conj}(G)|$, demonstrating a Floer theoretic McKay correspondence.
\end{abstract}
\tableofcontents

\section{Introduction}
A simple singularity is modeled by the isolated singular point of the variety $\C^2/G$, for a finite nontrivial subgroup $G\subset\SU(2)$. The action of $G$ on $\C[u,v]$ admits an invariant subring, generated by three monomials, $m_i(u,v)$ for $i=1,2,3$, that satisfy a minimal polynomial relation,
 \[f_G(m_1(u,v),m_2(u,v),m_3(u,v))=0,\] 
for some nonzero $f_G\in\C[z_1,z_2,z_3]$. These weighted polynomials $f_G$ provide an alternative perspective of the simple singularities as hypersurface singularities in $\C^3$. Specifically, the map 
\[\C^2/G\to V_G:=f^{-1}_G(0),\,\,\,\,\,[(u,v)]\mapsto(m_1(u,v),m_2(u,v),m_3(u,v))\] 
defines an isomorphism of complex varieties, $\C^2/G\simeq V_G$, and produces a hypersurface singularity given any finite nontrivial $G\subset \SU(2)$.  The following table summarizing the relationship of $G$ to $f_G$. The integer triple $(p,q,r)$  corresponds to the lengths of the 3 branches of the associated Dynkin diagram denoted by $\Gamma(G)$.  In the $A_n$ case, $(k,l)$  is an arbitrary pair of positive integers satisfying $k + l = n +1$. 

\begin{table}[h!]
\centering
\begin{tabular}{ | c | c | c | c | c | c | } 
\hline Group $G$ & Graph $\Gamma(G)$ & \ \ \ \ \ $f_G(z_1,z_2,z_3)$ \ \ \ \ \ & branches $(p,q,r)$ \\ 
\hline $\Z_{n+1}$ & $A_n$ & $z_1^{n+1} + z_2^2+z_3^2$ & $(1, k, l)$ \\ 
\hline $\D^*_{2n-4}$ & $D_n$ & $z_1^2z_2 + z_1^{n-1} + z_3^2$ & $(2, 2, n-2)$ \\ 
\hline $\T^*$ &$E_6$ & $z_1^4 + z_2^3 + z_3^2$ & $(2, 3, 3)$ \\ 
\hline $\mathbb{O}^*$ &$E_7$ & $z_1^3z_1 + z_2^3 + z_3^2$ & $(2, 3, 4)$ \\
\hline $\I^*$ &$E_8$ & $z_1^5 + z_2^3 + z_3^2$ & $(2, 3, 5)$ \\
\hline 
\end{tabular}
\caption{Polynomial relation $f_G$ for finite subgroups $G\subset\SU(2)$. }
\label{table: diagrams of graphs}
\end{table}

   One can recover the conjugacy class of $G$ from $V_G$ by studying the Dynkin diagram associated to the minimal resolution $\widetilde{X}_G$ of a simple singularity $\mathbf{0}$ of $V_G$, using the McKay correspondence \cite{mckay} summarized below.   The \emph{Dynkin diagram} associated to (the minimal resolution of) $(V_G, \mathbf{0} )$ is the finite graph whose vertex $v_i$ is labeled by the exceptional holomorphic sphere $Z_i$ of self-intersection -2, and $v_i$ is adjacent to $v_j$ if and only if $Z_i$ transversely intersects with $Z_j$. In this way, we associate to any simple singularity $(V_G,\mathbf{0})$ the graph $\Gamma(V_G,\mathbf{0})$. It is a classical fact that $\Gamma(V_G,\mathbf{0})$ is isomorphic to one of the $A_n$, $D_n$, or the $E_6$, $E_7$, or $E_8$ graphs (see \cite[\S 6]{Sl}), depicted in Figure \ref{figure: dynkin diagrams}. 

 \begin{figure}[h]
    \centering
    \captionsetup{justification=centering}
    \includegraphics[width=0.5\textwidth]{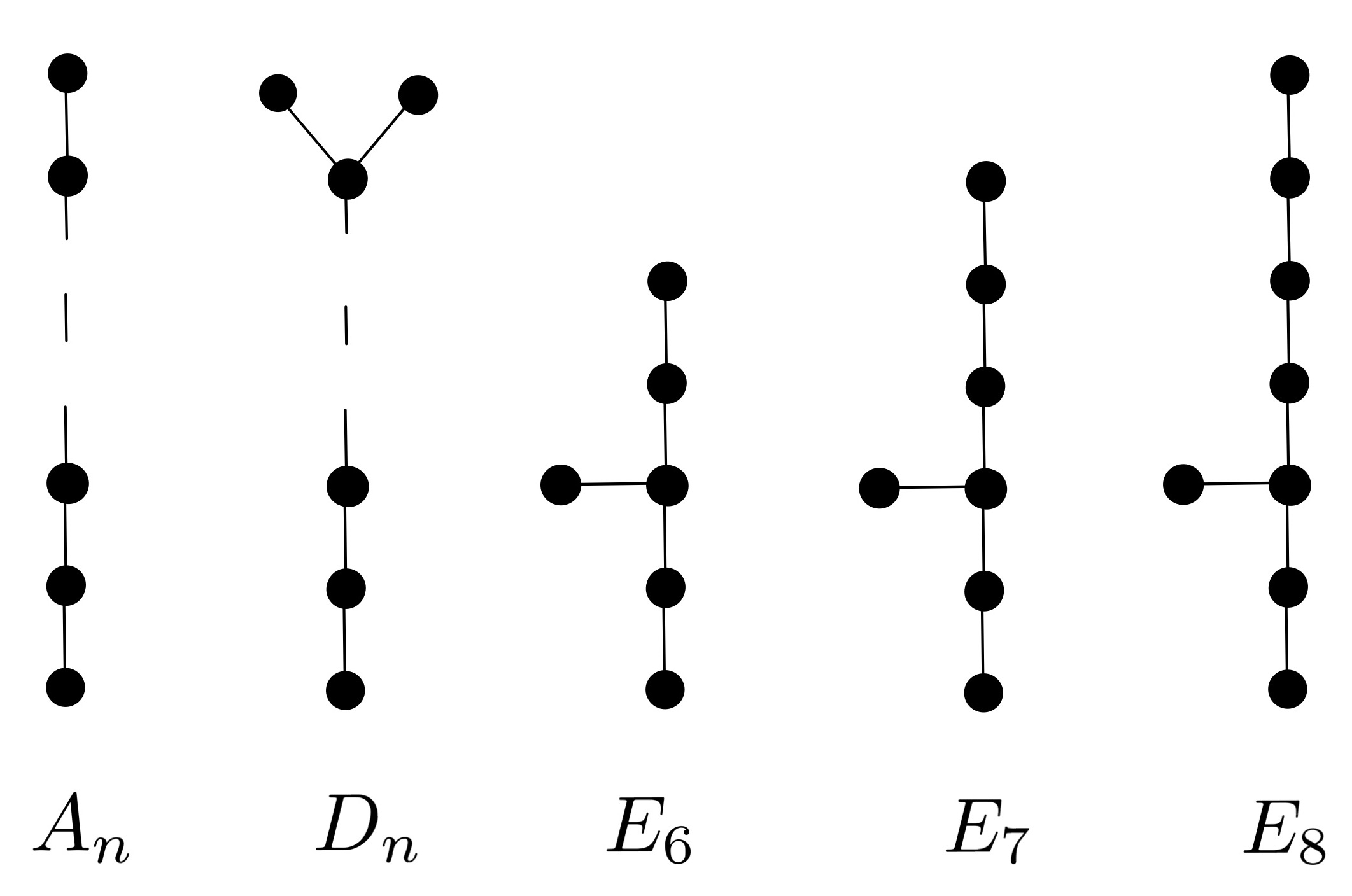}
    \caption{The Dykin diagrams; $A_n$ and $D_n$ feature $n$ nodes.}
    \label{figure: dynkin diagrams}
\end{figure}

The Dynkin diagrams also simultaneously classify the types of conjugacy classes of finite subgroups $G$ of $\SU(2)$.  Any finite subgroup $G\subset\SU(2)$ must be either cyclic, conjugate to $\D_{2n}^*$, or is a binary polyhedral group, cf. \cite[\S 1.6]{Z}. Associated to each type of finite subgroup $G\subset\SU(2)$ is a finite graph, $\Gamma(G)$. The vertices of $\Gamma(G)$ are in correspondence with the nontrivial irreducible representations $V_i$ of $G$, of which there are $|\text{Conj}(G)|-1$, where $\text{Conj}(G)$ denotes the set of conjugacy classes 
of a group $G$.   The McKay correspondence states that $\Gamma(G)$ is isomorphic to one of the $A_n$, $D_n$, or the $E_6$, $E_7$, or $E_8$ graphs, as enumerated in Table \ref{table: diagrams of graphs}.  The adjacency matrix $A_{ij}$ of the Dynkin diagram determines the tensor products $\C^2 \otimes V_i \cong \oplus A_{ij} V_j$ with the canonical representation, cf. \cite{St}.  We also note that the dimension of the cohomology of the minimal resolution is precisely the number of irreducible representations.

\begin{table}[h!]
\centering
 \begin{tabular}{||m{2cm} | m{2.5cm} | m{1.5cm} ||} 
 \hline
$G\subset\SU(2)$  & $|\text{Conj}(G)|-1$ & $\Gamma(G)$  \\ [0.5ex] 
 \hline\hline
 $\Z_n$ & $n-1$ & $A_{n-1}$  \\ 
 \hline
 $\D_{2n}^*$ & $n+2$ & $D_{n+2}$\\
 \hline
  $\T^*$ & $6$ & $E_6$\\
 \hline
   $\Oc^*$ & $7$ & $E_7$\\
 \hline
    $\I^*$ & $8$ & $E_8$\\
 \hline
\end{tabular}
\caption{Dykin diagrams associated to finite subgroups $G\subset\SU(2)$. }
\label{table: diagrams of graphs}
\end{table}



We adapt a method of computing the cylindrical contact homology of $(S^3/G,\xi_G)$ as a direct limit of action filtered homology groups, described by Nelson in \cite{N2}. This process uses a (lift of a) Morse function, which is invariant under the corresponding symmetry group in $\SO(3)$, to perturb the standard degenerate contact form.  In order to define the exact symplectic cobordism maps necessary to take direct limits, a detailed analysis of the homotopy classes of Reeb orbits is needed due to the presence of contractible and torsion Reeb orbits. 

Our computation realizes a contact Floer theoretic McKay correspondence result, namely that the ranks of the cylindrical contact homology of the links\footnote{ Recall that the \emph{link} of a hypersurface singularity in $\C^3$ is the 3-dimensional contact manifold $L:=S^5_{\epsilon}(0)\cap \{f^{-1}(0)\}$, with contact structure $\xi_L:=TL\cap J_{\C^3}(TL)$, where $J_{\C^3}$ is the standard integrable complex structure on $\C^3$, and $\epsilon>0$ is small. There is a contactomorphism $(S^3/G,\xi_G)\simeq(L,\xi_L)$, where $\xi_G$ on $S^3/G$ is the descent of the standard contact structure $\xi$ on $S^3$ to the quotient by the $G$-action. 
} of simple singularities are given in terms of the number of conjugacy classes of the group $G$.  It additionally recovers the presentation of the manifold as a Seifert fiber space and, in this sense, provides a natural basis for the cylindrical contact homology in terms of the Reeb orbits realizing the different conjugacy classes of $G$, cf. Remark \ref{rem:SFS}.

 We expect that our explicit description of the cylindrical chain complexes will enable computations of embedded contact homology and its associated spectral invariants after an appropriate adaption of arguments from \cite{NW, NW2}.  Such results will be of interest in the context of gauge theory as well as have applications to the study of symplectic embeddings and fillings.   Our computations realize McLean and Ritter's work, which computes the positive $S^1$-equivariant symplectic cohomology of the crepant resolution $Y$ of $\C^n/G$ in terms of the number of conjugacy classes of the finite $G\subset\SU(n)$, \cite[Theorem 1.10, Corollary 2.13]{MR}, without needing to know the cohomology of the minimal resolution.   




\subsection{Definitions and overview of cylindrical contact homology}
First we recall some basic definitions.  Let $(Y,\xi)$ be a closed contact three manifold with defining contact form $\lambda$. This contact form determines a smooth vector field, $R_{\lambda}$, called the \emph{Reeb vector field}, which uniquely satisfies $\lambda(R_{\lambda})=1$ and $d\lambda(R_{\lambda},\cdot)=0$. A $\emph{Reeb orbit}$ $\gamma$ is a map  $\R/T\Z\to M$, considered up to reparametrization, with $\dot{\gamma}(t)=R_{\lambda}(\gamma(t))$. Let $\mathcal{P}(\lambda)$ denote the set of Reeb orbits of $\lambda$. If $\gamma\in\mathcal{P}(\lambda)$ and $k\in\N$, then the $k$-fold iterate of $\gamma$, denoted $\gamma^k$, is the precomposition of $\gamma$ with $\R/kT\Z\to\R/T\Z$. The orbit $\gamma$ is \emph{embedded} when $\R/T\Z\to Y$ is injective. If $\gamma$ is the $m$-fold iterate of an embedded Reeb orbit, then $m(\gamma):=m$ is the \emph{multiplicity} of $\gamma$.

For a Reeb orbit $\gamma$ as above, the time $T$ linearized Reeb flow defines a symplectic linear map
\[
P_{\gamma}: \left( \xi_{\gamma(0)},d\lambda \right) \to \left( \xi_{\gamma(0)},d\lambda \right), 
\]
after making a choice of trivialization, which we also denote by $P_\gamma$.  We say $\gamma$ is \emph{nondegenerate} if $P_\gamma$ does not have 1 as an eigenvalue.  The contact form  $\lambda$ is called \emph{nondegenerate} if all $\gamma\in\mathcal{P}(\lambda)$ are nondegenerate.   A nondegenerate Reeb orbit is said to be \emph{elliptic} if $P_\gamma$ has its eigenvalues on the unit circle and hyperbolic if $P_\gamma$ has real eigenvalues.   (If both real eigenvalues are positive then $\gamma$ is a \emph{positive hyperbolic orbit} and if both real eigenvalues are negative then $\gamma$ is a \emph{negative hyperbolic orbit}.)

If $\tau$ is a homotopy class of trivializations of $\xi|_\gamma$, then the \emph{Conley Zehnder index}, $\mu_{\CZ}^{\tau}(\gamma)\in\Z$ is defined and related to the rotation of the Reeb flow along $\gamma$.  The parity of the Conley-Zehnder index does not depend on the choice of trivialization and is even when $\gamma$ is positive hyperbolic and odd when $\gamma$ is elliptic.  If $\gamma$ is an embedded negative hyperbolic orbit then the parity of the Conley-Zehnder index is odd for all odd iterates and even for all even iterates, with respect to any homotopy class of trivializations.
 An orbit $\gamma\in\mathcal{P}(\lambda)$ is said to be \emph{bad} if it is an even iterate of a negative hyperbolic orbit, otherwise, $\gamma$ is said to be \emph{good}. 
 Let $\mathcal{P}_{\text{good}}(\lambda)\subset\mathcal{P}(\lambda)$ denote the set of good Reeb orbits. 

If $\langle c_1(\xi), \pi_2(Y)\rangle=0$ and if $\mu_{\CZ}^{\tau}(\gamma)\geq3$ for all contractible  $\gamma\in\mathcal{P}(\lambda)$ with any $\tau$ extendable over a disc, we say the nondegenerate contact form $\lambda$ is \emph{dynamically convex}. The symplectic vector bundle $(\xi, d\lambda)$ admits a global trivialization if $c_1(\xi)=0$, which is unique up to homotopy if $\mbox{rank }H_1(Y)=0$. In this case, the integral \emph{grading} $|\gamma|$ of the generator $\gamma$ is defined to be $\mu_{\CZ}^{\tau}(\gamma)-1$ for any $\tau$ induced by a global trivialization of $\xi$. 
\begin{definition}
We say that an almost complex structure $J$ on $\R\times Y$ is {\em $\lambda$-compatible\/} if 
\begin{itemize}
\itemsep-.35em
\item $J(\xi)=\xi$; 
\item $d\lambda(v,Jv)>0$ for nonzero $v\in\xi$; 
\item $J$ is invariant under translation of the $\R$ factor; 
\item $J(\partial_s)=R_\lambda$, where $s$ denotes the $\R$ coordinate.   
\end{itemize}
We denote the set of all $\lambda$-compatible $J$ by $\mathcal{J}(Y,\lambda)$.  \end{definition}

Fix such a $\lambda$-compatible $J$. If $\gamma_+$ and $\gamma_-$ are Reeb orbits, we consider \emph{$J$-holomorphic cylinders} interpolating between them, which are smooth maps  $u: \R \times S^1 \to \R \times Y$ such that the nonlinear Cauchy-Riemann equation holds
\[
\partial_s u + J\partial_t u =0,
\]
 $\lim_{s \to \pm \infty} \pi_{\R}\circ u(s,t) = \pm \infty $, and $\lim_{s \to \pm \infty}\pi_Y u(s,\cdot)$ is a parametrization of $\gamma_\pm$.  Here $\pi_\R$ and $\pi_Y$ are the respective projections from $\R \times Y$ to $\R$ and $Y$.  We say that $u$ is positively asymptotic to $\gamma_+$ and negatively asymptotic to $\gamma_-$.  We declare two maps to be equivalent if they differ by translation and rotation of the domain $\R \times S^1$, and denote the set of equivalence classes by $\mathcal{M}^J(\gamma_+,\gamma_-)$.  There is an additional $\R$ action $\mathcal{M}^J(\gamma_+,\gamma_-)$ by translation of the $\R$ factor on the target $\R \times Y$.
 
We define the \emph{Fredholm index} of a cylinder $u\in \mathcal{M}^J(\gamma_+,\gamma_-)$ by
 \[\text{ind}(u)=\mu_{\CZ}^{\tau}(\gamma_+)-\mu_{\CZ}^{\tau}(\gamma_-)+2c_1(u^*\xi,\tau),\]
after fixing a trivialization $\tau$ of $\xi$ over $\gamma_+$ and $\gamma_-$.  The relative first Chern class  $c_1(u^*\xi,\tau)$ vanishes when $\tau$ extends to a trivialization of $u^*\xi$.   For  $k\in\Z$, $\mathcal{M}_k^J(\gamma_+,\gamma_-)$ denotes those cylinders with $\text{ind}(u)=k$. The significance of the Fredholm index is that  if $J$ is generic and $u\in \mathcal{M}_k^J(\gamma_+,\gamma_-)$ is somewhere injective, then  $\mathcal{M}^J_{k}(\gamma_+,\gamma_-)$ is naturally a manifold near $u$ of dimension $k$.


For a nondegenerate contact form $\lambda$, and under favorable transversality conditions, we define the cylindrical contact homology chain complex $CC_*(Y,\lambda,J)$ over $\Q$ as follows.  (The original definition is due to Eliashberg-Givental-Hofer \cite{EGH} and we are using notation from \cite{HN2}, but suppressing some decorations as we only consider one cylindrical flavor of contact homology in this paper.)  As a module, $CC_*(Y,\lambda,J)$ is noncanonically isomorphic to the vector space over $\Q$ generateed by good Reeb orbits; an isomorphism is fixed after a choice of coherent orientations, which is used to define a $\Z$-module $\mathcal{O}_\gamma$ that is noncanonically isomorphic to $\Z$, cf. \cite[A.3]{HN2}.  We then define
\[
CC_*(Y,\lambda, J) = \bigoplus_{\gamma \in \mathcal{P}_{\text{good}}(\lambda)} \mathcal{O}_\gamma \otimes_\Z \Q.
\]   
The choice of a generator of $\mathcal{O}_\gamma$ for each good Reeb orbit specifies an isomorphism
\[
CC_*(Y,\lambda, J) \simeq \Q\langle\mathcal{P}_{\text{good}}(\lambda)\rangle.
\]
This chain complex admits a canonical $\Z/2$-grading determined by the mod 2 Conley-Zehnder index, which can be upgraded to a relative or absolute $\Z$ grading in certain circumstances.  In the setting of this paper, we have an absolute $\Z$ grading given by
\[
|\gamma| = \mu_{\CZ}^\tau(\gamma) - 1,
\]
where $\tau$ is any homotopy class of the global unitary trivialization constructed in \eqref{equation: global trivialization} and Remark \ref{rem:triv}.

To define the differential, we first define the following operator assuming that all moduli spaces $\mathcal{M}_k^J(\alpha,\beta)$ with Fredholm index $k \leq 1$ are cut out transversely:
\[
\delta: CC_*(Y,\lambda, J) \to  CC_{*-1}(Y,\lambda, J),
\]
given by
\[
\delta \alpha = \sum_{\beta \in \mathcal{P}_{\text{good}}(\lambda)} \sum_{u \in \mathcal{M}_1^J(\alpha,\beta)/\R} \frac{\epsilon(u)}{d(u)} \beta.
\]
Here $\epsilon(u)$ is an element of $\{ \pm 1 \}$ after generators of $\mathcal{O}_\alpha$ and  $\mathcal{O}_\beta$ have been chosen, cf. \cite[Def. A.26]{HN2}, and $d(u) \in \Z_{>0}$ is the covering multiplicity of $u$, which is 1 if and only if $u$ is somewhere injective.  

Next we define an operator
\[
\kappa : CC_*(Y,\lambda, J) \to  CC_{*}(Y,\lambda, J)
\]
by
\[
\kappa(\alpha) = d(\alpha)\alpha.
\]
Under suitable transversality assumptions for $\mathcal{M}_2^J(\alpha,\beta)$, then counting their ends yields
\begin{equation}\label{dkd}
\delta \kappa \delta =0.
\end{equation}
This was proven in the dynamically convex case in \cite{HN} and recovered in arbitrary odd dimensions in the absence of contractible Reeb orbits in \cite{HN2}.  As a result of \eqref{dkd}, we obtain that
\[
\partial := \delta \kappa
\]
is a differential on $CC_{*}(Y,\lambda, J)$.  The differential preserves the free homotopy class of Reeb orbits because they count cylinders which project to homotopies in $Y$ between Reeb orbits.

 Under additional hypotheses, this homology is independent of contact form $\lambda$ defining $\xi$ and generic $J$ (for example, if $\lambda$ admits no contractible Reeb orbits, \cite[Corollary 1.10]{HN2}), and is denoted $CH_*(Y,\xi)$. This is the \emph{cylindrical contact homology} of $(Y,\xi)$. Upcoming work of Hutchings and Nelson will show that $CH_*(Y,\xi)=CH_*(Y,\lambda,J)$ is independent of dynamically convex $\lambda$ and generic $J$.

\subsection{Main result and connections to other work}
The link of the $A_n$ singularity is shown to be contactomorphic to the lens space $L(n+1,n)$ in \cite[Theorem 1.8]{AHNS}. More generally, 
  the links of simple singularities $(L,\xi_L)$ are shown to be contactomorphic to quotients $(S^3/G,\xi_G)$ in \cite[Theorem 5.3]{N3}.  Theorem \ref{theorem: main} computes the cylindrical contact homology of $(S^3/G,\xi_G)$ as a direct limit of  filtered homology groups.

\begin{theorem} \label{theorem: main}
    Let $G\subset\mbox{\em SU}(2)$ be a finite nontrivial group, and let $m=|\mbox{\em Conj}(G)|\in\N$ be the number of conjugacy classes of $G$. The cylindrical contact homology of $(S^3/G, \xi_G)$ is \[CH_*(S^3/G,\lambda_G, J):=\varinjlim_N CH_*^{L_N}(S^3/G,\lambda_N, J_N)\cong\bigoplus_{i\geq0}\Q^{m-2}[2i]\oplus\bigoplus_{i\geq0} H_*(S^2;\Q)[2i].\]
    \end{theorem}
     The directed system of filtered cylindrical contact homology groups $CH_*^{L_N}(S^3/G,\lambda_N,J_N)$ is described in Section \ref{subsection: structure of proof of main theorem}. Upcoming work of Hutchings and Nelson will show that this direct limit is an invariant of $(S^3/G,\xi_G)$, in the sense that it is isomorphic to $CH_*(S^3/G,\lambda,J)$  where $\lambda$ is any dynamically convex contact form on $S^3/G$ with kernel $\xi_G$, and $J\in\mathcal{J}(\lambda)$ is generic. 
     
     \medskip
     
     The brackets in Theorem \ref{theorem: main} describe the degree of the grading.\footnote{For example, $\Q^8[5]\oplus H_*(S^2;\Q)[3]$ is a ten dimensional space with nine dimensions in degree 5, and one dimension in degree 3.}  By the classification of finite subgroups $G$ of $\SU(2)$, the following enumerates the possible values of $m=|\text{Conj}(G)|$: 
\begin{enumerate}[(i)]
        \itemsep-.35em
       \item If $G$ is cyclic of order $n$, then $m=n$.
       \item If $G$ is binary dihedral,  $G\cong\D^*_{2n}$ for some $n$, then $m=n+3$.
       \item If $G$ is binary polyhedral, $G\cong\T^*, \Oc^*$, or $\I^*$, then  $m=7, 8$, or $9$, respectively.
\end{enumerate}

\begin{remark}\label{rem:SFS}
The cylindrical contact homology in Theorem \ref{theorem: main} recovers the presentation of the manifold $S^3/G$ as a Seifert fiber space, whose $S^1$-action agrees with the Reeb flow of a contact form defining $\xi_G$. Viewing the manifold $S^3/G$ as an $S^1$-bundle over an orbifold surface\footnote{Namely $S^2/H$ where $H = P(G) \subset \SO(3)$ and $P:\SU(2) \to \SO(3)$, cf. Section \ref{subsection: structure of proof of main theorem}.} homeomorphic to $S^2$, the copies of $H_*(S^2;\Q)$ appearing in Theorem \ref{theorem: main} may be understood as the \emph{orbifold Morse homology} of this base. Each orbifold point $p$ with isotropy order $n_p$ corresponds to an exceptional fiber, $\gamma_p$, in $S^3/G$, which may be realized as an embedded Reeb orbit. The generators of the $\Q^{m-2}[0]$ term are the iterates $\gamma_p^k$ for $k=1,2,\dots, n_p-1$ so that the dimension $m-2=\sum_p(n_p-1)$ of this summand can be regarded as a kind of total isotropy of the base. 
\end{remark}

\begin{remark}
Theorem \ref{theorem: main} can alternatively be expressed as
\[CH_*(S^3/G,\lambda_G, J)\cong\begin{cases} \Q^{m-1} & *=0,\\ \Q^{m} & *\geq2\,\, \mbox{and even} \\ 0 &\mbox{else.}\end{cases}\]
In this form, we realize the expected isomorphism \cite{BO}  between cylindrical contact homology and the positive $S^1$-equivariant symplectic cohomology with coefficients in $\Q$ of the crepant resolutions $Y$ of the singularities $\C^2/G$, as computed by McLean and Ritter. Their work shows that these groups with $\Q$-coefficients are free $\Q[u]$-modules of rank equal to $m=|\text{Conj}(G)|$, where $G\subset\SU(n)$ and $u$ has degree 2 \cite[Corollary 2.13]{MR}.  \end{remark}

\begin{remark}
Recent work of Haney and Mark  computes the cylindrical contact homology in \cite{HM} of a family of hyperbolic Brieskorn manifolds $\Sigma(p,q,r)$, for $p$, $q$, $r$ relatively prime positive integers satisfying $\frac{1}{p}+\frac{1}{q}+\frac{1}{r}<1$, using methods from \cite{N2}. Their work uses a family of \emph{hypertight} contact forms, whose Reeb orbits are non-contractible. These manifolds are also Seifert fiber spaces, whose cylindrical contact homology features summands arising from copies of the homology of the orbit space, as well as summands from the total isotropy of the orbifold. 
\end{remark}

\subsection{Structure of proof of main theorem}\label{subsection: structure of proof of main theorem}

We now outline the proof of Theorem \ref{theorem: main}.  Section \ref{section: geometric setup}  explains the process of perturbing a degenerate contact form $\lambda_G$ on $S^3/G$ using an orbifold Morse function. Given a finite, nontrivial subgroup $G\subset\SU(2)$, $H$ denotes the image of $G$ under the double cover of Lie groups $P:\SU(2) \cong \mbox{Spin}(3)\to\SO(3)$. By Lemma \ref{lemma: commutes}, the quotient by the $S^1$-action on the Seifert fiber space $S^3/G$ may be identified with a map $\fp:S^3/G\to S^2/H$. This $\fp$ fits into a commuting square of topological spaces \eqref{diagram: commuting square} involving the Hopf fibration $\fP:S^3\to S^2$.

An $H$-invariant Morse-Smale function on $(S^2, \omega_{\FS}(\cdot,j\cdot))$, constructed in Section \ref{appendix: constructing morse functions}, descends to an \emph{orbifold Morse function}, $f_H$, on $S^2/H$. Here, $\omega_{\FS}$ is the Fubini-Study form on $S^2\cong\C P^1$, and $j$ is the standard integrable complex structure. By Lemma \ref{lemma: reeb1}, the Reeb vector field of the perturbed contact form  on $S^3/G$ 
\[
\lambda_{G,\varepsilon}:=(1+\varepsilon\fp^*f_H)\lambda_G
\]
is the descent of the vector field
\[
R_{\lambda_\varepsilon}:=\frac{R_{\lambda}}{1+\varepsilon\fP^*f}-\varepsilon\frac{\widetilde{X_f}}{(1+\varepsilon\fP^*f)^2}
\]
 to $S^3/G$. Here, $\widetilde{X_f}$ is a horizontal lift to $S^3$ of the Hamiltonian vector field $X_f$ of $f$ on $S^2$, computed with respect to $\omega_{\FS}$, and we use the convention that $\iota_{X_f}\omega_{\FS}=-df$. Thus, the $\widetilde{X_f}$ term vanishes along exceptional fibers $\gamma_p$ of $S^3/G$ projecting to orbifold critical points $p\in S^2/H$ of $f_H$, implying that these parametrized circles and their iterates $\gamma_p^k$ are Reeb orbits of $\lambda_{G,\varepsilon}$. Lemma \ref{lemma: ActionThresholdLink} computes the Conley-Zehnder index $\mu_{\CZ}^\tau(\gamma_p^k)$ in terms of $k$ and the Morse index of $f_H$ at $p$ with respect to a global unitary trivialization. 
 
 Next we outline our procedure of taking direct limits in Section \ref{section: direct limits of filtered homology} of action filtered cylindrical contact homology in Section \ref{section: computation of filtered contact homology}.  Given a contact manifold $(Y,\lambda)$, the \emph{action} of a Reeb orbit $\gamma:\R/T\Z\to Y$ is the positive quantity 
 \[
 \mathcal{A}(\gamma):=\int_{\gamma}\lambda=T.
 \]
  For $L>0$, we let $\mathcal{P}^L(\lambda)\subset\mathcal{P}(\lambda)$ denote the set of orbits $\gamma$ with $\mathcal{A}(\gamma)<L$. A contact form $\lambda$ is $L$-\emph{nondegenerate} when all $\gamma\in\mathcal{P}^L(\lambda)$ are nondegenerate. If $\langle c_1(\xi),\pi_2(Y)\rangle=0$ and $\mu_{\CZ}(\gamma)\geq 3$ for all contractible $\gamma\in\mathcal{P}^L(\lambda)$, we say that the $L$-nondegenerate contact form $\lambda$ is $L$-\emph{dynamically convex}. By Lemma \ref{lemma: ActionThresholdLink}, given $L>0$, all  $\gamma\in\mathcal{P}^L(\lambda_{G,\varepsilon})$ are nondegenerate and project to critical points of $f_H$ under $\fp$, when $\varepsilon$ is sufficiently small.

This lemma allows for the computation in Section \ref{section: computation of filtered contact homology} of the  \emph{action filtered} cylindrical contact homology.  After fixing $L>0$, $\partial$ restricts to a differential, $\partial^L$, on the subcomplex generated by  $\gamma\in\mathcal{P}_{\text{good}}^L(\lambda)$, denoted $CC_*^L(Y,\lambda,J)$,  whose homology is denoted $CH_*^L(Y,\lambda,J)$.
This is because the differential decreases action: if $\mathcal{A}(\gamma_+)<\mathcal{A}(\gamma_-)$ then $\mathcal{M}^J(\gamma_+,\gamma_-)$ is empty because by Stokes' theorem, action decreases along holomorphic cylinders in a symplectization.  

In Section \ref{section: computation of filtered contact homology}, we use Lemma \ref{lemma: ActionThresholdLink} to produce a  sequence $(L_N, \lambda_N, J_N)_{N=1}^{\infty}$, where $L_N\nearrow\infty$ in $\R$, $\lambda_N$ is an $L_N$-dynamically convex contact form for $\xi_G$, and $J_N\in\mathcal{J}(\lambda_N)$ is generic. By Lemmas \ref{lemma: CZdihedral} and \ref{lemma: CZpolyhedral}, every orbit $\gamma\in\mathcal{P}^{L_N}_{\text{good}}(\lambda_N)$ is of even degree, and so $\partial^{L_N}=0$, providing
\begin{equation}\label{equation: Identify}
    CH_*^{L_N}(S^3/G,\lambda_N,J_N)\cong\Q\langle\,\mathcal{P}_{\text{good}}^{L_N}(\lambda_N)\,\rangle\cong\bigoplus_{i=0}^{2N-1}\Q^{m-2}[2i]\oplus\bigoplus_{i=0}^{2N-2} H_*(S^2;\Q)[2i].
\end{equation}

Finally, we prove Theorem \ref{theorem: cobordisms induce inclusions} in Section \ref{section: direct limits of filtered homology}, which states  that a completed symplectic  cobordism $(X,\lambda,J)$  from $(\lambda_N,J_N)$ to $(\lambda_M,J_M)$, for $N\leq M$, induces a  homomorphism,  \[\Psi:CH_*^{L_N}(S^3/G,\lambda_N,J_N)\to  CH_*^{L_M}(S^3/G,\lambda_M,J_M)\] which takes the form of the standard inclusion when making the identification \eqref{equation: Identify}. The proof of Theorem \ref{theorem: cobordisms induce inclusions} comes in two steps. First, the moduli spaces $\mathcal{M}_0^J(\gamma_+,\gamma_-)$ are finite by Proposition \ref{proposition: buildings in cobordisms} and Corollary \ref{corollary: moduli spaces are finitie}, implying that the map $\Psi$ is well defined. Second, the identification of $\Psi$ with a standard inclusion is made precise in the following manner. Given $\gamma_+\in\mathcal{P}^{L_N}_{\text{good}}(\lambda_N)$, there is a unique $\gamma_-\in\mathcal{P}^{L_M}_{\text{good}}(\lambda_M)$ which
\begin{enumerate}[(i)]
        \itemsep-.3em
\item projects to the same critical point of $f_H$ as $\gamma_+$ under $\fp$,
\item satisfies $m(\gamma_+)=m(\gamma_-)$.
\end{enumerate} 
When (i) and (ii) hold, we write $\gamma_+\sim\gamma_-$. We argue in Section \ref{section: direct limits of filtered homology} that $\Psi$ takes the form $\Psi([\gamma_+])=[\gamma_-]$, when $\gamma_+\sim\gamma_-$.

Theorem \ref{theorem: cobordisms induce inclusions} now implies that the system of filtered contact homology groups is identified with a sequence of inclusions of vector spaces, providing isomorphic direct limits: \[\varinjlim_N CH_*^{L_N}(S^3/G,\lambda_N, J_N)\cong\bigoplus_{i\geq0}\Q^{m-2}[2i]\oplus\bigoplus_{i\geq0}H_*(S^2;\Q).\]

\subsection{Connections to orbifold Morse homology}\label{subsection: cylindrical contact homology as an analogue of orbifold Morse homology}
Using the construction of the \emph{orbifold Morse-Smale-Witten complex} as in Cho and Hong in \cite{CH} we can draw the following parallels between orbifold Morse homology and cylindrical contact homology. Given an {orbifold Morse function} $f$ on an orbifold $X$, the chain group $CM_*(X,f)$ is generated by the {orientable} critical points of $f$. The differential, $\partial^M$, is given as a weighted count of the negative gradient flow lines between orientable critical points in $X$.  There are two notable similarities between the chain complexes $(CC_*,\partial)$ and $(CM_*,\partial^M)$, exemplified by our computations.  

\medskip

(1) \emph{ Bad Reeb orbits are analogous to non-orientable critical points.} \\
Bad Reeb orbits are excluded as generators of $CC_*$ for the same reasons that {non-orientable critical points} are excluded as generators of $CM_*$. A critical point $p$ of $f$ on an orbifold is \emph{non-orientable} if the action of its isotropy group $\Gamma_p$ on a choice of  unstable manifold is not orientation preserving. Analogously, a Reeb orbit $\gamma$ is \emph{bad} if the action of its cyclic deck group $\Delta_{\gamma}$ on an asymptotic operator is not orientation preserving.  \tb{}

Our Seifert projections $\fp:S^3/G\to S^2/H$ geometrically realize this analogy: if $\gamma$ is a bad Reeb orbit associated to $(1+\varepsilon\fp^*f_H)\lambda_G$ in $S^3/G$ that projects to orbifold critical point $p$ of $f_H$, then $p$ is non-orientable. Conversely, if $p\in S^2/H$ is a non-orientable critical point of $f_H$, then there is a bad Reeb orbit $\gamma$ associated to $(1+\varepsilon\fp^*f_H)\lambda_G$ in $S^3/G$ projecting to $p$.  This interplay can be realized through the pairs $(\gamma,p)=(h^2, p_h)$ for the binary dihedral group in Sections \ref{subsection: dihedral} and $(\gamma,p)=(\mathcal{E}^2, \mathfrak{e})$ for the binary polyhedral group in Section \ref{subsection: polyhedral}.

\medskip
    
(2) \emph{The differentials are structurally identical}. \\ Take good Reeb orbits $\alpha$ and $\beta$ with $\mu_{\CZ}(\alpha)-\mu_{\CZ}(\beta)=1$, and take orientable critical points $p$ and $q$ of orbifold Morse $f$ with $\text{ind}_f(p)-\text{ind}_f(q)=1$. Now compare 
\[
\langle \partial \alpha, \beta\rangle=\sum_{u\in\mathcal{M}_1^J(\alpha,\beta)/\R}\dfrac{\epsilon(u)d(\alpha)}{d(u)},\,\,\,\,\,\,\,\,\, \langle \partial^M p, q\rangle=\sum_{x\in\mathcal{M}(p, q)/\R}\dfrac{\epsilon(x)|\Gamma_p|}{|\Gamma_x|}.\]

 Here, $\mathcal{M}(p, q)$ is the space of negative gradient paths $x$ from $p$ to $q$, and $\Gamma_x$ is the local isotropy group at any point on the path $x$, whose order divides $|\Gamma_p|$. Both $\epsilon\in\{\pm1\}$ quantities come from choices of orientations in each setting and are well-defined because $\alpha$ and $\beta$ are good, and because $p$ and $q$ are orientable. 
 
The similarities of both boundary operators as  \emph{weighted} counts of moduli spaces reflect the parallels between  the breaking and gluing of the two theories: a single broken gradient path or building may serve as the limit of \emph{multiple} ends of a 1-dimensional moduli space in either setting.  For a thorough treatment of why these signed counts generally produce a differential that squares to zero, see \cite[Theroem 5.1]{CH} (in the orbifold case) and \cite[\S 4.3]{HN} (in the contact case).  \\

In Sections \ref{subsection: cylinders over orbifold Morse trajectories} and \ref{subsection: visualizing holomorphic cylinders: an example}, we explain the analogies between the contact data of $S^3/G$ and the orbifold Morse data of $S^2/H$ in further detail.

\vspace{-.25cm}

\subsection*{Acknowledgements}
Leo Digiosia thanks his advisor, Jo Nelson, for her exceptional guidance and discussions. Leo Digiosia was supported by NSF grants \href{https://www.nsf.gov/awardsearch/showAward?AWD_ID=1745670&HistoricalAwards=false}{DMS-1745670}, \href{https://www.nsf.gov/awardsearch/showAward?AWD_ID=1840723&HistoricalAwards=false}{DMS-1840723}, and \href{https://www.nsf.gov/awardsearch/showAward?AWD_ID=2104411&HistoricalAwards=false}{DMS-2104411}.  Jo Nelson is supported by NSF grants \href{https://www.nsf.gov/awardsearch/showAward?AWD_ID=2104411&HistoricalAwards=false}{DMS-2104411} and \href{https://www.nsf.gov/awardsearch/showAward?AWD_ID=2142694&HistoricalAwards=false}{CAREER DMS-2142694}.  We thank the referee for their thoughtful reading and helpful comments on this paper, especially with respect to the cobordism maps in Section 4.  (Section 4 also benefited from a comment from Chris Wendl and his book \cite{wint}.)

\section{Geometric setup and dynamics}\label{section: geometric setup}
In this section we first review the process of perturbing degenerate contact forms on $S^3$ and $S^3/G$ using a Morse function to achieve nondegeneracy up to an action threshold, following \cite[\S 1.5]{N2}. We then identify the associated Reeb orbits of $S^3/G$ and compute their Conley Zehnder indices in Lemma \ref{lemma: ActionThresholdLink}.  In Section \ref{appendix: constructing morse functions} we construct the $H$-invariant Morse functions we use to perturb the contact forms on $S^3/G$.  In Sections  \ref{subsection: cylinders over orbifold Morse trajectories} and \ref{subsection: visualizing holomorphic cylinders: an example} we elucidate how the Reeb dynamics realize the Morse orbifold data associated to $S^2/H$.

\subsection{Spherical geometry and associated Reeb dynamics}\label{subsection: spherical geometry and associated Reeb dynamics}
The diffeomorphism between $S^3\subset\C^2$ and  $\SU(2)$ provides $S^3$ with the structure of a Lie group:
\begin{equation} \label{equation: S^3 Lie}
    (\alpha,\beta)\in S^3\mapsto\begin{pmatrix}\alpha &-\overline{\beta}\,\,\\ \beta & \overline{\alpha}\end{pmatrix}\in\SU(2)
\end{equation}
and we see that $e=(1,0)\in S^3$ is the identity element. The round contact form on $S^3$, denoted $\lambda$, is defined as the restriction of the 1-form $\iota_{V}\omega_0\in\Omega^1(\C^2)$ to $S^3$, where
\[\omega_0:=\frac{i}{2}\sum_{k=1}^2dz_k\wedge d\overline{z_k}\,\,\,\,\,\text{and}\,\,\,\,\,V:=\frac{1}{2}\sum_{k=1}^2z_k\partial_{z_k}+\overline{z_k}\partial_{\overline{z_k}},\,\,\implies\,\,\,\iota_{V}\omega_0=\frac{i}{4}\sum_{k=1}^2z_k\wedge d\overline{z_k}-\overline{z_k}dz_k.\] The $\SU(2)$-action on $\C^2$ preserves  $\omega_0$ and $V$, and so the $\SU(2)$-action on $S^3$ preserves $\lambda$.

There is a natural Lie algebra isomorphism between the tangent space of the identity element of a Lie group and its collection of left-invariant vector fields.  The contact plane $\xi_e=\text{ker}(\lambda_e)$ at the identity element $e=(1,0)\in S^3$ is spanned by the tangent vectors $\partial_{x_2}|_e$ and $\partial_{y_2}|_e$, where we are viewing \[\xi_e\subset T_e\C^2\cong T_e\R^4=\text{Span}_{\R}\{\partial_{x_i}|_e,\partial_{y_i}|_e\}.\] Let $V_1$ and $V_2$ be the  left-invariant vector fields corresponding to $\partial_{x_2}|_e=\langle 0,0,1,0\rangle$ and  $\partial_{y_2}|_e=\langle 0,0,0,1\rangle$, respectively. Because $S^3$ acts on itself by \emph{contactomorphisms},  $V_1$ and $V_2$ are sections of $\xi$ and provide a global unitary trivialization of $(\xi, d\lambda|_{\xi},J_{\C^2})$, denoted $\tau$:
\begin{equation}\label{equation: global trivialization}
    \tau:S^3\times\R^2\to\xi, \,\,\, (p,\eta_1,\eta_2)\mapsto \eta_1V_1(p)+\eta_2V_2(p)\in\xi_p.
\end{equation}
Here, $J_{\C^2}$ is the standard integrable complex structure on $\C^2$. Note that $J_{\C^2}
(V_1)=V_2$ everywhere. Given a Reeb orbit $\gamma$ of any contact form on $S^3$, let the symbol $\mu_{\CZ}(\gamma)$ denote the Conley Zehnder index of $\gamma$ with respect to this $\tau$. If $(\alpha,\beta)\in S^3$, write $\alpha=a+ib$ and $\beta=c+id$. Then, with respect to the ordered basis $(\partial_{x_1}, \partial_{y_1}, \partial_{x_2}, \partial_{y_2})$ of $T_{(\alpha,\beta)}\R^4$, we have the following expressions
\begin{equation}\label{equation: vector field coordinates}
    V_1(\alpha,\beta)=\langle -c,d, a, -b\rangle,\,\,\,\,
    V_2(\alpha, \beta)=\langle -d,-c,b,a\rangle.
\end{equation}

Consider the double cover of Lie groups, $P:\SU(2)\to\SO(3)$. The kernel of $P$ has order 2 and is generated by $-\text{Id}\in\SU(2)$, the only element of $\SU(2)$ of order 2.

\begin{equation}\label{equation: P in coordinates}
    \begin{pmatrix}\alpha & -\overline{\beta}\,\, \\ \beta & \overline{\alpha}\end{pmatrix}\in\SU(2)\xmapsto{P}\begin{pmatrix}|\alpha|^2-|\beta|^2 & 2\text{Im}(\alpha\beta) & 2\text{Re}(\alpha\beta) \\ -2\text{Im}(\overline{\alpha}\beta) & \text{Re}(\alpha^2+\beta^2) & -\text{Im}(\alpha^2+\beta^2)\\ -2\text{Re}(\overline{\alpha}\beta) & \text{Im}(\alpha^2-\beta^2) & \text{Re}(\alpha^2-\beta^2)\end{pmatrix}\in\SO(3)
\end{equation}
A diffeomorphism $\C P^1\to S^2\subset\R^3$ is given in homogeneous coordinates ($|\alpha|^2+|\beta|^2=1$) by \begin{equation}\label{equation: sphere cp1 diffeomorphism}
    (\alpha:\beta)\in \C P^1\mapsto(|\alpha|^2-|\beta|^2,-2\text{Im}(\bar{\alpha}\beta),-2\text{Re}(\bar{\alpha}\beta))\in S^2.
\end{equation}
We have an $\SO(3)$-action on $\C P^1$, pulled back from the $\SO(3)$-action on $S^2$ by \eqref{equation: sphere cp1 diffeomorphism}. Lemma \ref{lemma: commutes} illustrates how the action of $\SU(2)$ on $S^3$ is related to the action of $\SO(3)$ on $\C P^1\cong S^2$ via $P:\SU(2)\to\SO(3).$

\begin{lemma} \label{lemma: commutes}
For a point $z$ in $S^3$, let $[z]\in\C P^1$ denote the corresponding point under the quotient of the $S^1$-action on $S^3$. Then for all $z\in S^3$, and all matrices $A\in\mbox{\em SU}(2)$, we have \[[A\cdot z]=P(A)\cdot[z]\in\C P^1\cong S^2.\]
\end{lemma}

\begin{proof}
First, note that the result holds for the case $z=e=(1,0)\in S^3$. This is because $[e]\in\C P^1$ corresponds to $(1,0,0)\in S^2$ under \eqref{equation: sphere cp1 diffeomorphism}, and so for any $A$, $P(A)\cdot[e]$ is the first column of the $3\times 3$ matrix $P(A)$ appearing in \eqref{equation: P in coordinates}. That is, \begin{equation}\label{equation: sphere points}
    P(A)\cdot[e]=(|\alpha|^2-|\beta|^2,-2\text{Im}(\bar{\alpha}\beta),-2\text{Re}(\bar{\alpha}\beta)),
\end{equation} (where $(\alpha,\beta)\in S^3$ is the unique element corresponding to $A\in \SU(2)$, using \eqref{equation: S^3 Lie}). By \eqref{equation: sphere cp1 diffeomorphism}, the point \eqref{equation: sphere points} equals $[(\alpha,\beta)]=[(\alpha,\beta)\cdot (1,0)]=[A\cdot e]$, and so the result holds when $z=e$.

For the general case, note that any $z\in S^3$ equals $B\cdot e$ for some $B\in\SU(2)$, and use the fact that $P$ is a group homomorphism.
\end{proof}

The Reeb flow of $\lambda$ is given by the $S^1\subset\C^*$ (Hopf) action, $p\mapsto e^{it}\cdot p$. Thus, all $\gamma\in\mathcal{P}(\lambda)$ have period $2k\pi$, with linearized return maps equal to $\text{Id}:\xi_{\gamma(0)}\to\xi_{\gamma(0)}$, and are degenerate.

\begin{notation}\label{notation: morse data}
Following the general recipe of perturbing the degenerate contact form on a prequantization bundle, outlined in \cite[\S 1.5]{N2}, we establish the following notation:
\begin{itemize}
     \itemsep-.35em
    \item $\fP:S^3\to S^2$  is the Hopf fibration.
    \item $f$ is a  Morse-Smale function on $(S^2,\omega_{\FS}(\cdot, j\cdot))$ and $\text{Crit}(f)$ is its set of critical points,
    \item For $\varepsilon>0$,
\begin{itemize}    
\item[]      $f_{\varepsilon}:=1+\varepsilon f:S^2\to \R$,
\item[] $F_{\varepsilon}:=f_{\varepsilon}\circ\fP:S^3\to \R$,
\item[]  $\lambda_{\varepsilon}:=F_{\varepsilon}\lambda\in\Omega^1(S^3)$.
\end{itemize}
    \item $\widetilde{X_f}\in\xi$ is the horizontal lift of $X_f\in TS^2$ using the fiberwise linear symplectomorphism $d\fP|_{\xi}:(\xi,d\lambda|_{\xi})\to (TS^2, \omega_{\text{FS}})$, where $X_f$ denotes the Hamiltonian vector field of $f$ on $S^2$ with respect to $\omega_{\text{FS}}$.
\end{itemize}
\end{notation}
\begin{remark}\label{remark: ham}
Our convention is that for a smooth, real valued function $f$ on symplectic manifold $(M,\omega)$, the Hamiltonian vector field $X_f$ uniquely satisfies $\iota_{X_f}\omega=-df$. 
\end{remark}

For small $\varepsilon$,  ker $\lambda_{\varepsilon} = \mbox{ker} \lambda$. We refer to $\lambda_{\varepsilon}$ as the \emph{perturbed contact form} on $S^3$. Although $\lambda$ and $\lambda_{\varepsilon}$ define the same contact structure, their Reeb dynamics differ.

\begin{lemma} \label{lemma: reeb1}
The following relationship between vector fields on $S^3$ holds:
\[R_{\lambda_{\varepsilon}}=\frac{R_{\lambda}}{F_{\varepsilon}}-\varepsilon\frac{\widetilde{X_f}}{F_{\varepsilon}^2}.\]
\end{lemma}

\begin{proof}
This is \cite[Prop. 4.10]{N2}. Note that the sign discrepancy is a result of our convention regarding  Hamiltonian vector fields, see Remark \ref{remark: ham}.
\end{proof}

We now explore how the relationship between vector fields from Lemma \ref{lemma: reeb1} provides a relationship between Reeb and Hamiltonian flows. 
\begin{notation}
(Reeb and Hamiltonian flows). For any $t\in\R$,
\begin{itemize}
    \itemsep-.35em
    \item $\phi_0^t:S^3\to S^3$ denotes the time $t$ flow of the unperturbed Reeb vector field $R_{\lambda}$,
    \item $\phi^t:S^3\to S^3$ denotes the time $t$ flow of the perturbed Reeb vector field $R_{\lambda_{\varepsilon}}$,
    \item $\varphi^t:S^2\to S^2$ denotes the time $t$ flow of the vector field $V:=-\frac{\varepsilon X_{f}}{f_{\varepsilon}^2}.$
    
\end{itemize}
\end{notation}

\begin{lemma}\label{lemma: flow}
For all $t$ values, we have $\fP\circ\phi^t=\varphi^t\circ\fP$ as smooth maps $S^3\to S^2$.
\end{lemma}

\begin{proof}
 Pick $z\in S^3$ and let $\widetilde{\gamma}:\R\to S^3$ denote the unique integral curve for $R_{\lambda_{\varepsilon}}$ which passes through $z$ at time $t=0$, i.e., $\widetilde{\gamma}(t)=\phi^t(z)$.  By Lemma \ref{lemma: reeb1}, $d\fP$ carries the derivative of $\widetilde{\gamma}$ precisely to the vector $V\in TS^2$. Thus, $\fP\circ\widetilde{\gamma}:\R\to S^2$ is the unique integral curve, $\gamma$, of $V$ passing through $p:=\fP(z)$ at time $t=0$, i.e., $\gamma(t)=\varphi^t(p)$. Combining these facts provides \[\fP(\widetilde{\gamma}(t))=\gamma(t)\,\,\implies\,\,\fP(\phi^t(z))=\varphi^t(p)\,\,\implies\fP(\phi^t(z))=\varphi^t(\fP(z)).\]
\end{proof}

Lemma \ref{lemma: orbits} describes the orbits $\gamma\in\mathcal{P}(\lambda_{\varepsilon})$ projecting to critical points of $f$ under $\fP$.
\begin{lemma}\label{lemma: orbits}
Let $p\in\mbox{\em Crit}(f)$ and take $z\in\fP^{-1}(p)$. Then the map \[\gamma_p:[0,2\pi f_{\varepsilon}(p)]\to S^3,\,\,\,\,t\mapsto e^{\frac{it}{f_{\varepsilon}(p)}}\cdot z\]
descends to a closed, embedded Reeb orbit $\gamma_p:\R/2\pi f_{\varepsilon}(p)\Z\to S^3$ of $\lambda_{\varepsilon}$, passing through point $z$ in $S^3$, whose image under $\fP$ is $\{p\}\subset S^2$, where $\cdot$ denotes the $S^1\subset\C^*$ action on $S^3$.
\end{lemma}

\begin{proof}
The map $\R/2\pi\Z\to S^3$, $t\mapsto e^{it}\cdot z$ is a closed, embedded integral curve for the degenerate Reeb field $R_{\lambda}$, and so by the chain rule we have that $\dot{\gamma_p}(t)=R_{\lambda}(\gamma_p(t))/f_{\varepsilon}(p)$. Note that $\fP(\gamma(t))=\fP(e^{\frac{it}{f_{\varepsilon}(p)}}\cdot z)=\fP(z)=p$ and, because $\widetilde{X_f}(\gamma(t))$ is a lift of $X_f(p)=0$, we have $\widetilde{X_f}(\gamma(t))=0$. By the description of $R_{\lambda_{\varepsilon}}$ in Lemma \ref{lemma: reeb1}, we have $\dot{\gamma_p}(t)=R_{\lambda_{\varepsilon}}(\gamma_p(t))$.
\end{proof}

 Next we set notation to be used in describing the local models for the linearized Reeb flow along the orbits $\gamma_p$ from Lemma \ref{lemma: orbits}. For $s\in\R$, $\mathcal{R}(s)$ denotes the $2\times2$ rotation matrix: \[\mathcal{R}(s):=\begin{pmatrix}\cos{(s)} & -\sin{(s)} \\ \sin{(s)} & \cos{(s)}\end{pmatrix}\in\SO(2).\]
Note that $J_0=\mathcal{R}(\pi/2)$. For $p\in\text{Crit}(f)\subset S^2$, pick coordinates $\psi:\R^2\to S^2$, so that $\psi(0,0)=p$. Then we let $H(f,\psi)$ denote the Hessian of $f$ in these coordinates at $p$.

\begin{notation}\label{notation: stereographic}
The term \emph{stereographic coordinates at $p\in S^2$} describes a smooth $\psi:\R^2\to S^2$ with $\psi(0,0)=p$, which has a factorization $\psi=\psi_1\circ\psi_0$, where $\psi_0:\R^2\to S^2$ is the map 
\[
(x,y)\mapsto\frac{1}{1+x^2+y^2}(2x,2y,1-x^2-y^2),
\]
 taking $(0,0)$ to $(0,0,1)$, and $\psi_1:S^2\to S^2$ is given by the action of some element of $\SO(3)$ taking $(0,0,1)$ to $p$. If $\psi$ and $\psi'$ are both stereographic coordinates at $p$, then they differ by a precomposition with some $\mathcal{R}(s)$ in $\SO(2)$. Note that $\psi^*\omega_{\text{FS}}=\frac{dx\wedge dy}{(1+x^2+y^2)^2}$ (\cite[Ex. 4.3.4]{MS}). 
\end{notation}

Lemma \ref{lemma: rotate} describes the linearized Reeb flow of the unperturbed $\lambda$ with respect to $\tau$.

\begin{lemma} \label{lemma: rotate}
For any $z\in S^3$, the linearization $d\phi_0^t|_{\xi_z}:\xi_z\to\xi_{e^{it}\cdot z}$ is represented by $\mathcal{R}(2t)$, with respect to ordered bases $(V_1(z),V_2(z))$ of $\xi_z$ and $(V_1(e^{it}\cdot z), V_2(e^{it}\cdot z))$ of $\xi_{e^{it}\cdot z}$.
\end{lemma}

\begin{proof}
Since the  $V_i$ are $\SU(2)$-invariant and the $\SU(2)$-action commutes with $\phi_0^t$, we may reduce to the case $z=e=(1,0)\in S^3$. That is, we must show
for all $t$ values that
\begin{align}
    (d\phi_0^t)_{(1,0)}V_1(1,0)&=\cos{(2t)}V_1(e^{it},0)+\sin{(2t)}V_2(e^{it},0) \label{equation: eq1}\\
    (d\phi_0^t)_{(1,0)}V_2(1,0)&=-\sin{(2t)}V_1(e^{it},0)+\cos{(2t)}V_2(e^{it},0). \label{equation: eq2}
\end{align}
Note that \eqref{equation: eq2}  follows from  \eqref{equation: eq1} by applying the endomorphism $J_{\C^2}$ to both sides of \eqref{equation: eq1}, and noting both that $d\phi_0^t$ commutes with $J_{\C^2}$, and that $J_{\C^2}(V_1)=V_2$. We now prove \eqref{equation: eq1}.

The coordinate descriptions \eqref{equation: vector field coordinates} tell us that $V_1(e^{it},0)$ and $V_2(e^{it},0)$ can be respectively written as $\langle 0,0,\cos{t},-\sin{t}\rangle$ and $\langle 0,0,\sin{t},\cos{t}\rangle$. Angle sum formulas now imply \[\cos{(2t)}V_1(e^{it},0)+\sin{(2t)}V_2(e^{it},0)=\langle 0, 0, \cos{t},\sin{t}\rangle.\]
The vector on the right is precisely $(d\phi_0^t)_{(1,0)}V_1(1,0)$, and so we have proven \eqref{equation: eq1}.
\end{proof}

Proposition \ref{proposition: flow} and Corollary \ref{corollary: CZ sphere} conclude our discussion of dynamics on $S^3$.

\begin{proposition}\label{proposition: flow}
Fix a critical point $p$ of $f$ in $S^2$ and stereographic coordinates $\psi:\R^2\to S^2$ at $p$, and suppose $\gamma\in\mathcal{P}(\lambda_{\varepsilon})$ projects to $p$ under $\fP$. Let $M_t\in\mbox{\em Sp}(2)$ denote $d\phi^t|_{\xi_{\gamma(0)}}:\xi_{\gamma(0)}\to\xi_{\gamma(t)}$ with respect to the trivialization $\tau$. Then $M_t$ is a conjugate of the matrix  \[\mathcal{R}\bigg(\frac{2t}{f_{\varepsilon}(p)}\bigg)\cdot\exp\bigg(\frac{-t\varepsilon}{f_{\varepsilon}(p)^2} J_0\cdot H(f,\psi)\bigg)\] by some  element of $\mbox{\em SO}(2)$, which is independent of $t$.
\end{proposition}

\begin{proof}
Let $z:=\gamma(0)$. We linearize the identity $\fP\circ\phi^t=\varphi^t\circ\fP$ from Lemma \ref{lemma: flow}, restrict to $\xi_z$, and rearrange to recover the equality $d\phi^t|_{\xi_z}=a\circ b\circ c:\xi_z\to\xi_{\phi^t(z)}$, where \[a=\left(d\fP|_{\xi_{\phi^t(z)}}\right)^{-1}:T_pS^2\to\xi_{\phi^t(z)},\,\, b=d\varphi^t_p:T_pS^2\to T_pS^2,\,\,\text{and}\,\, c=d\fP|_{\xi_z}:\xi_z\to T_pS^2.\]

Let $v_i:=d\fP (V_i(z))\in T_pS^2$, then $(v_1,v_2)$ and $(V_1,V_2)$ provide an oriented basis of each of the three vector spaces appearing in the above composition of linear maps. Let $A$, $B$, and $C$ denote the matrix representations of $a$, $b$, and $c$ with respect to these ordered bases. We have $M_t=A\cdot B\cdot C$. Note that $C=\text{Id}$. We compute $A$ and $B$:

To compute $A$, recall that $\phi^t_0:S^3\to S^3$  denotes the time $t$ flow of the unperturbed Reeb field (alternatively, the Hopf action). Linearize the equality $\fP\circ\phi^t_0=\fP$, then use $\phi_0^t(z)=\phi^{t/f_{\varepsilon}(p)}(z)$ from Lemma \ref{lemma: orbits} and  Lemma \ref{lemma: rotate} to find that
\begin{equation}\label{equation A_t}A=\mathcal{R}(2t/f_{\varepsilon}(p)).\end{equation}

To compute $B$, note that $\varphi^t$ is the Hamiltonian flow of the function $1/f_{\varepsilon}$ with respect to $\omega_{\FS}$. It is advantageous to study this flow using our stereographic coordinates: recall that $\psi$ defines a symplectomorphism $\left(\R^2,\frac{dx\wedge dy}{(1+x^2+y^2)^2}\right)\to(S^2\setminus\{p'\},\omega_{\FS})$, where $p'\in S^2$ is antipodal to $p$ in $S^2$ (see Notation \ref{notation: stereographic}). Because symplectomorphisms preserve Hamiltonian data, we have that $\psi^{-1}\circ\varphi^t\circ\psi:\R^2\to\R^2$ is  given near the origin as the time $t$ flow of the Hamiltonian vector field for $\psi^*(1/f_{\varepsilon})$ with respect to $\frac{dx\wedge dy}{(1+x^2+y^2)^2}$. That is, \[\psi^{-1}\circ\varphi^t\circ\psi\,\,\,\, \text{is the time $t$ flow of} \,\,\,\,\frac{\varepsilon (1+x^2+y^2)^2f_y}{f_{\varepsilon}^2}\partial_x-\frac{\varepsilon (1+x^2+y^2)^2f_x}{f_{\varepsilon}^2}\partial_y.\]
Recall that if $P\partial_x+Q\partial_y$ is a smooth vector field on $\R^2$, vanishing at $(0,0)$, then the linearization of its time $t$ flow evaluated at the origin is represented by the $2\times2$ matrix $\exp(t X)$, with respect to the standard ordered basis $(\partial_x,\partial_y)$ of $T_{(0,0)}\R^2$, where \[X=\begin{pmatrix}P_x & P_y \\ Q_x & Q_y \end{pmatrix}.\]
Here, the partial derivatives of $P$ and $Q$ are implicitly assumed to be evaluated at the origin. In this spirit, set $P=\frac{\varepsilon (1+x^2+y^2)^2f_y}{f_{\varepsilon}^2}$ and $Q=\frac{-\varepsilon (1+x^2+y^2)^2f_x}{f_{\varepsilon}^2}$, and we compute \[X=\begin{pmatrix}P_x & P_y \\ Q_x & Q_y \end{pmatrix}=\frac{\varepsilon}{f_{\epsilon}(p)^2}\begin{pmatrix}f_{yx}(0,0) & f_{yy}(0,0) \\ -f_{xx}(0,0) & -f_{xy}(0,0)\end{pmatrix}=\frac{-\varepsilon}{f_{\varepsilon}(p)^2}J_0\cdot H(f,\psi).\]
This implies that $B=D^{-1}\exp\big(\frac{-t\varepsilon}{f_{\varepsilon}(p)^2}J_0 H(f,\psi)\big)D$, where $D$ is a change of basis matrix relating $(v_1,v_2)$ and the pushforward of $(\partial_x,\partial_y)$ by $\psi$ in $T_pS^2$. Because $\psi$ is holomorphic, $D$ must equal $r\cdot\mathcal{R}(s)$ for some $r>0$ and some $s\in\R$. This provides that \begin{equation}\label{equation: B_t}B=\mathcal{R}(-s)\exp\bigg(\frac{-t\varepsilon}{f_{\varepsilon}(p)^2}J_0 H(f,\psi)\bigg)\mathcal{R}(s).\end{equation}

Finally, we combine \eqref{equation A_t} and \eqref{equation: B_t} to conclude
\begin{align*}
    M_t=A\cdot B&=\mathcal{R}\bigg(\frac{2t}{f_{\varepsilon}(p)}\bigg)\cdot \mathcal{R}(-s)\exp\bigg(\frac{-t\varepsilon}{f_{\varepsilon}(p)^2}J_0 H(f,\psi)\bigg) \mathcal{R}(s) \\
    &=\mathcal{R}(-s)\mathcal{R}\bigg(\frac{2t}{f_{\varepsilon}(p)}\bigg)\exp\bigg(\frac{-t\varepsilon}{f_{\varepsilon}(p)^2}J_0 H(f,\psi)\bigg)\mathcal{R}(s).
\end{align*}
\end{proof}

\begin{corollary}\label{corollary: CZ sphere}
Fix $L>0$. Then there exists some $\varepsilon_0>0$ such that for $\varepsilon\in(0,\varepsilon_0]$, all Reeb orbits $\gamma\in\mathcal{P}^L(\lambda_{\varepsilon})$ are nondegenerate, take the form $\gamma_p^k$, a $k$-fold cover of an embedded Reeb orbit $\gamma_p$ as in in Lemma \ref{lemma: orbits}, and $\mu_{\CZ}(\gamma)=4k+\mbox{\em ind}_{f}(p)-1$
for $p\in\mbox{\em Crit}(f)$.
\end{corollary}

\begin{proof}
That $\gamma$ is nondegenerate and projects to a critical point $p$ of $f$ is proven in \cite[Lemma 4.11]{N2}. To compute $\mu_{\CZ}(\gamma)$, we apply the naturality, loop, and signature properties of the Conley Zehnder index (see \cite[\S 2.4]{S}) to our path $\{M_t\}\subset\Sp(2)$. By Proposition \ref{proposition: flow}, this family of matrices has a factorization, up to $\SO(2)$-conjugation,  $M_t=\Phi_t\Psi_t$, where $\Phi$ is the loop of symplectic matrices $\R/2\pi kf_{\varepsilon}(p)\Z\to\Sp(2)$,  $t\mapsto\mathcal{R}(2t/f_{\varepsilon}(p))$, and $\Psi$ is the path of matrix exponentials $t\mapsto \exp(\frac{-t\varepsilon}{f_{\varepsilon}(p)^2}J_0 H(f,\psi))$, where $\psi$ denotes a choice of stereographic coordinates at $p$. In total,
\begin{align*}
    \mu_{\CZ}(\gamma_p^k)&=2\mu(\Phi)+\mu_{\CZ}(\Psi)\\
    &=2\cdot 2k+\frac{1}{2}\text{sign}(-H(f,\psi))\\
    &=4k+\text{ind}_f(p)-1.
\end{align*}
Here, $\mu$ denotes the Maslov index of a loop of symplectic matrices (see \cite[\S 2]{MS}).
\end{proof}

\subsection{Geometry of $S^3/G$ and associated Reeb dynamics} \label{subsection: sfs geometry}
The previous process of perturbing a degenerate contact form on prequantization bundles, is often used to compute Floer theories, for example, their cylindrical contact homology \cite{N2} and embedded contact homology \cite{NW}. Although the quotients $S^3/G$ are not prequantization bundles, they do admit an $S^1$-action (with fixed points), and are examples of {Seifert fiber spaces} which are realizable as principal $S^1$-orbibundles over integral symplectic orbifolds.

Let $G\subset\SU(2)$ be a finite nontrivial group. Since $G$ acts on $S^3$ without fixed points, $S^3/G$ inherits smooth structure. The quotient $\pi_G:S^3\to S^3/G$ is a universal cover, thus $\pi_1(S^3/G)\cong G$ is completely torsion, and  $\mbox{rank } H_1(S^3/G)=0$. Because the $G$-action preserves $\lambda\in\Omega^1(S^3)$, we have a descent of $\lambda$ to a contact form on $S^3/G$, denoted $\lambda_G\in\Omega^1(S^3/G)$, with $\xi_G:=\text{ker}(\lambda_G)$. As the actions of $S^1$ and $G$ on $S^3$ commute, we obtain an $S^1$-action on $S^3/G$, which realizes the Reeb flow of $\lambda_G$. Hence, $\lambda_G$ is degenerate.

Let $H\subset\SO(3)$ denote $P(G)$, the image of $G$ under $P:\SU(2)\to \SO(3)$. The $H$-action on $S^2$ has fixed points, and so the quotient $S^2/H$ inherits \emph{orbifold} structure. Lemma \ref{lemma: commutes} provides a unique map $\fp:S^3/G\to S^2/H$, making the following diagram commute

\begin{equation}\label{diagram: commuting square}\begin{tikzcd}
S^3 \arrow{r}{\pi_G} \arrow{d}{\fP}
& S^3/G \arrow{d}{\fp} \\
S^2 \arrow{r}{\pi_H}
& S^2/H
\end{tikzcd}
\end{equation}
where $\pi_G$ is a finite cover, $\pi_H$ is an orbifold cover, $\fP$ is a projection of a prequantization bundle, and $\fp$ is identified with the Seifert fibration.

\begin{remark}\label{rem:triv}
(Global trivialization of $\xi_G$). Recall the $\SU(2)$-invariant vector fields $V_i$ spanning $\xi$ on $S^3$ \eqref{equation: global trivialization}. Because these $V_i$ are $G$-invariant, they descend to smooth sections of $\xi_G$, providing a global unitary trivialization, $\tau_G$, of $\xi_G$, hence $c_1(\xi_G)=0$. Given a Reeb orbit $\gamma$ of some contact form on $S^3/G$, we denote by $\mu_{\CZ}(\gamma)$ the Conley Zehnder index of $\gamma$ with respect to this global trivialization.  
\end{remark}

We assume that the Morse function $f:S^2\to\R$ is $H$-invariant and descends to an orbifold Morse function, $f_H:S^2/H\to\R$, in the language of \cite{CH}. The $H$-invariance of $f$ provides that the smooth $F=f\circ\fP$ is $G$-invariant, and descends to a smooth function, $F_G:S^3/G\to\R$. We define, analogously to Notation \ref{notation: morse data}, \[f_{H,\varepsilon}:=1+\varepsilon f_H,\hspace{1.5cm} F_{G,\varepsilon}:=1+\varepsilon F_G,\hspace{1.5cm}\lambda_{G,\varepsilon}:=F_{G,\varepsilon}\lambda_G.\] For sufficiently small $\varepsilon$, $\lambda_{G,\varepsilon}$ is a contact form on $S^3/G$ with kernel $\xi_G$. The condition $\pi_G^*\lambda_{G,\varepsilon}=\lambda_{\varepsilon}$ implies that $\gamma:[0,T]\to S^3$ is an integral curve of $R_{\lambda_{\varepsilon}}$ if and only if $\pi_G\circ\gamma:[0,T]\to S^3/G$ is an integral curve of $R_{\lambda_{G,\varepsilon}}$.

\begin{remark}\label{remark: same local models} 
(Local models on $S^3$ and $S^3/G$ agree). Suppose $\gamma:[0,T]\to S^3$ is a Reeb trajectory of  $\lambda_{\varepsilon}$, so that $\pi_G\circ\gamma:[0,T]\to S^3/G$ is a Reeb trajectory of  $\lambda_{G,  \varepsilon}$. For $t\in[0,T]$, let $M_t\in\Sp(2)$ denote the time $t$ linearized Reeb flow of $\lambda_{\varepsilon}$ along $\gamma$ with respect to $\tau$, and let $N_t\in\Sp(2)$ denote that of $\lambda_{G,\varepsilon}$ along $\pi_G\circ\gamma$ with respect to $\tau_G$. Then $M_t=N_t$, because the local contactomorphism $\pi_G$ preserves the trivializations in addition to the contact forms.
\end{remark}

Let $\mathcal{O}(p):=\{h\cdot p \ |\,h\in H\}$ to be the \emph{orbit} of $p$, and denote the \emph{isotropy subgroup} of $p$ by
\[
H_p : = \{h\in H\ |\, h\cdot p=p\}\subset H.
\]
 Recall that $|\mathcal{O}(p)||H_p|=|H|$ for any $p \in S^2$. A point $p\in S^2$ is a \emph{fixed point} if $|H_p|>1$. The set of fixed points of $H$ is $\text{Fix}(H)$. The point  $q\in S^2/H$ is an \emph{orbifold point} if $q=\pi_H(p)$ for some $p\in\text{Fix}(H)$. We now additionally assume that $f$ satisfies $\text{Crit}(f)=\text{Fix}(H)$; this will be the case in  Section \ref{section: computation of filtered contact homology}. The Reeb orbit $\gamma_p\in\mathcal{P}(\lambda_{\varepsilon})$ from Lemma \ref{lemma: orbits} projects to $p\in\text{Crit}(f)$ under $\fP$, and thus $\pi_G\circ\gamma_p\in\mathcal{P}(\lambda_{G,\varepsilon})$ projects to the orbifold point $\pi_H(p)$ under $\fp$. Lemma \ref{lemma: covering multiplicity} computes the Reeb orbit multiplicity $d(\pi_G\circ\gamma_p)$.

\begin{lemma}\label{lemma: covering multiplicity}
Let $\gamma_p\in\mathcal{P}(\lambda_{\varepsilon})$ be the embedded Reeb orbit in $S^3$ from Lemma \ref{lemma: orbits}. Then the multiplicity of $\pi_G\circ\gamma_p\in\mathcal{P}(\lambda_{G,\varepsilon})$  is $2|H_p|$ if $|G|$ is even, and is $|H_p|$ if $|G|$ is odd.
\end{lemma}

\begin{proof}
Recall that $|G|$ is even if and only if $|G|=2|H|$, and that $|G|$ is odd if and only if $|G|=|H|$ (the only element of $\SU(2)$ of order 2 is $-\text{Id}$, the generator of $\text{ker}(P)$). By the classification of finite subgroups of $\SU(2)$, if $|G|$ is odd, then $G$ is cyclic. 

Let $q:=\pi_H(p)\in S^2/H$, let $d:=|\mathcal{O}(p)|$ so that $d|H_p|=|H|$ and $|G|=rd|H_p|$, where $r=2$ when $|G|$ is even and $r=1$ when odd. Label the points of $\mathcal{O}(p)$ by $p_1=p, p_2, \dots, p_d$. Now $\fP^{-1}(\mathcal{O}(p))$ is a disjoint union of $d$ Hopf fibers, $C_i$, where $C_i=\fP^{-1}(p_i)$. Let $C$  denote the embedded circle $\fp^{-1}(q)\subset S^3/G$. By commutativity of   \eqref{diagram: commuting square}, we have that $\pi_G^{-1}(C)=\sqcup_{i}C_i=\fP^{-1}(\mathcal{O}(p))$. We have the following commutative diagram of circles and points:
\[\begin{tikzcd}
\sqcup_{i=1}^dC_i=\pi_G^{-1}(C) \arrow{r}{\pi_G} \arrow{d}{\fP}
& C=\fp^{-1}(q) \arrow{d}{\fp} \\
\{p_1,\dots, p_d\}=\mathcal{O}(p) \arrow{r}{\pi_H}
& \{q\}
\end{tikzcd}\]
We must have that $\pi_G:\sqcup_{i=1}^dC_i\to C$ is a $|G|=rd|H_p|$-fold cover from the disjoint union of $d$ Hopf fibers to one embedded circle. The restriction of $\pi_G$ to any one of these circles $C_i$ provides a smooth covering map, $C_i\to C$; let $n_i$ denote the degree of this cover.  Because $G$ acts transitively on these circles, all of the degrees $n_i$ are equal to some $n$. Thus, $\sum_{i}^dn_i=dn=|G|=rd|H_p|$ which implies that $n=r|H_p|$ is the covering multiplicity of $\pi_G\circ\gamma_p$.
\end{proof}

We conclude this section with an analogue of Corollary $\ref{corollary: CZ sphere}$ for the Reeb orbits of $\lambda_{G,\varepsilon}$.

\begin{lemma}\label{lemma: ActionThresholdLink}
Fix $L>0$. Then there exists some $\varepsilon_0>0$ such that, for  $\varepsilon\in(0,\varepsilon_0]$, all Reeb orbits  $\gamma\in\mathcal{P}^L(\lambda_{G,\varepsilon})$ are nondegenerate, project to an orbifold critical point of $f_{H}$ under $\fp:S^3/G\to S^2/H$, and $\mu_{\CZ}(\gamma)=4k+\mbox{\em ind}_f(p)-1$ whenever $\gamma$ is contractible with a lift to some orbit $\gamma_p^k$ in $S^3$ as in Lemma \ref{lemma: orbits}, where $p\in\mbox{\em Crit}(f)$.
\end{lemma}

\begin{proof}
Let $L':=|G|L$ and take the corresponding $\varepsilon_0$ as appearing in Corollary \ref{corollary: CZ sphere}, applied to $L'$. Now, for $\varepsilon\in(0,\varepsilon_0]$, elements of $\mathcal{P}^{L'}(\lambda_{\varepsilon})$ are nondegenerate and project to critical points of $f$. Let $\gamma\in\mathcal{P}^L(\lambda_{G,\varepsilon})$ and let $n\in\N$ denote the order of $[\gamma]$ in $\pi_1(S^3/G) \cong G$. Now we see that $\gamma^n$ is contractible and lifts to an orbit $\widetilde{\gamma}\in\mathcal{P}^{L'}(\lambda_{\varepsilon})$, which must be  nondegenerate and must project to some critical point $p$ of $f$ under $\fP$. 

If the orbit $\gamma$ is degenerate, then $\gamma^n$ is degenerate and, by the discussion in Remark \ref{remark: same local models}, $\widetilde{\gamma}$ would be degenerate. Commutativity of  \eqref{diagram: commuting square} implies that $\gamma$ projects to the orbifold critical point $\pi_H(p)$ of $f_H$ under $\fp$. Finally, if $n=1$ then again by Remark \ref{remark: same local models}, the local model $\{N_t\}_{t\in[0,T]}$ of the Reeb flow along $\gamma$ matches that of $\widetilde{\gamma}$, $\{M_t\}_{t\in[0,T]}$, and thus $\mu_{\CZ}(\gamma)=\mu_{\CZ}(\widetilde{\gamma})$. The latter index is computed in Corollary \ref{corollary: CZ sphere}.

\end{proof}

\subsection{Construction of $H$-invariant Morse-Smale functions}\label{appendix: constructing morse functions}

We now produce the $H$-invariant, Morse-Smale functions on $S^2$ for the dihedral $H=\D_{2n}$ and polyhedral $H=\P$ subgroups of $\SO(3)$. Table \ref{table: morse points} describes three finite subsets, $X_0$, $X_1$, and $X_2$, of $S^2$ which depend on $H\subset\SO(3)$. We construct an $H$-invariant, Morse-Smale function $f$ on $(S^2, \omega_{\FS}(\cdot, j\cdot))$, whose set of critical points of index $i$ is $X_i$, so that $\text{Crit}(f)=X:=X_0\cup X_1\cup X_2$. Additionally, $X=\text{Fix}(H)$, the fixed point set of the $H$-action on $S^2$.  This constructed $f$ is \emph{perfect} in the sense that it features the minimal number of required critical points, because $\text{Fix}(H)\subset\text{Crit}(f)$ must always hold. In the case that $H$ is a polyhedral group, $X_0$ is the set of vertex points, $X_1$ is the set of edge midpoints, and $X_2$ is the set of face barycenters.  

We illustrate the constructions of our perfect $H$-invariant, Morse-Smale functions on $S^2$ in Figures \ref{figure: dihedral}, \ref{figure: tetrahedral}, \ref{figure: octahedral}, and \ref{figure: icosahedral}, wherein the blue, violet, and red critical points are of Morse index 2, 1, and 0 respectively.  Further details are also given in the following proof of Lemma \ref{lemma: morse}


\begin{table}[h!]
\centering
 \begin{tabular}{|| c | c | c | c ||} 
 \hline
 $H$ & $X_0$ & $X_1$ & $X_2$\\ [0.5ex] 
 \hline\hline
 $\D_{2n}$ & $\{p_{-,k}\,|\,1\leq k\leq n\}$ & $\{p_{h,k}\,|\,1\leq k\leq n\}$ & $\{p_{+,k}\,|\, 1\leq k\leq 2\}$\\
 \hline
 $\P$ & $\{\mathfrak{v}_k\,|\,1\leq k\leq A\}$ & $\{\mathfrak{e}_k\,|\,1\leq k\leq B\}$   & $\{\mathfrak{f}_k\,|\,1\leq k\leq C\}$ \\
 \hline
\end{tabular}
 \caption{Fixed points of $H$ sorted by Morse index}
 \label{table: morse points}
\end{table}
\begin{lemma}\label{lemma: morse}
Let $H\subset\mbox{\em SO}(3)$ be either $\D_{2n}$ or $\P$. Then there exists an $H$-invariant, Morse function $f$ on $S^2$, with $\mbox{\em Crit}(f)=X$, such that $\mbox{\em ind}_f(p)=i$ if $p\in X_i$. Furthermore, there are stereographic coordinates at $p\in X$ in which $f$ takes the form  
\begin{enumerate}[(i)]
 \itemsep-.35em
    \item $q_0(x,y):=(x^2+y^2)/2-1$,\,\,\, if $p\in X_0$
    \item $q_1(x,y):=(y^2-x^2)/2$,\,\,\,\,\,\,\,\,\,\,\,\,\, if $p\in X_1$
    \item $q_2(x,y):=1-(x^2+y^2)/2$,\,\,\,\,\,if $p\in X_2$
\end{enumerate}
\end{lemma}
\begin{proof}
We first produce an auxiliary Morse function $\widetilde{f}$, which might not be $H$-invariant, then $f$ is taken to be the $H$-average of $\widetilde{f}$. Fix $\delta>0$;  for $p\in X$, let $D_p\subset S^2$ be the open geodesic disc centered at $p$ with radius $\delta$ with respect to the metric $\omega_{\FS}(\cdot, j\cdot)$. Define $\widetilde{f}$ on $D_p$ to be the pullback of $q_0$, $q_1$, or $q_2$, for $p\in X_0$, $X_1$, or $X_2$ respectively, by stereographic coordinates at $p$. Set $D:=\cup_{p\in X}D_p$. For $\delta$ small, $D$ is a \emph{disjoint} union and $\widetilde{f}:D\to\R$ is Morse. We can arrange for our selection of stereographic coordinates to satisfy \[(*)\,\,\text{for all}\,\, p\in X_1,\,\,\text{and}\,\, h\in H, \,\, \widetilde{f}|_{D_p}=\widetilde{f}\circ\phi_h|_{D_p},\] where $\phi_h:S^2\to S^2$ denotes $x\mapsto h\cdot x$. This ensures that the ``saddles are rotated in the same $H$-direction". Notice that $(*)$ automatically holds for $p\in X_0\cup X_2$, for any choice of coordinates, because of the rotational symmetry of the quadratics $q_0$ and $q_2$. Note that:
\begin{enumerate}[(a)]
    \itemsep-.35em
    \item The $\delta$-neighborhood $D$ of $X$ in $S^2$, is an $H$-invariant set, and for all $p\in X$, the $H_p$-action restricts to an action on $D_p$, where $H_p\subset H$ denotes the stabilizer subgroup.
        \item For all $p\in X$, $\widetilde{f}|_{D_p}:D_p\to\R$ is $H_p$-invariant.
    \item The function $\widetilde{f}:D\to\R$ is an $H$-invariant Morse function, with $\text{Crit}(\widetilde{f})=X$.
\end{enumerate}
 The $H$-invariance and $H_p$-invariance in (a) hold,  because $H\subset\SO(3)$ acts on $S^2$ by $\omega_{\FS}(\cdot,j\cdot)$-\emph{isometries} (rotations about axes through $p\in X$), and because $X$ is an $H$-invariant set. The $H_p$-invariance of (b) holds because, in stereographic coordinates, the $H_p$-action pulled back to $\R^2$ is always generated by some linear rotation about the origin, $\mathcal{R}(\theta)$. Both $q_0$ and $q_2$ are invariant with respect to any $\mathcal{R}(\theta)$, whereas $q_1$ is invariant with respect to the action generated by $\mathcal{R}(\pi)$, which is precisely the action by $H_p$ when $p\in X_1$, so (b) holds. Finally, the $H$-invariance in (c) holds directly by the $H_p$ invariance from (b), and by $(*)$. Now, extend the domain of $\widetilde{f}$ from $D$ to all of $S^2$ so that $\widetilde{f}$ is smooth and Morse, with $\text{Crit}(\widetilde{f})=X$. Figures \ref{figure: dihedral} and \ref{figure: tetrahedral} depict possible extensions $\widetilde{f}$ in the $H=\D_{2n}$ and $H=\T$ cases, for example.

For $h\in H$, let $\phi_h:S^2\to S^2$ denote the group action, $p\mapsto h\cdot p$. Define \[f:=\frac{1}{|H|}\sum_{h\in H}\phi_h^*\widetilde{f},\] where $|H|\in\N$ is the group order of $H$.  This $H$-invariant $f$ is smooth and agrees with $\widetilde{f}$ on $D$. If no critical points are created in the averaging process of $\widetilde{f}$, then we have that $\text{Crit}(f)=X$, implying that $f$ is Morse, and we are done.

We say that the extension $\widetilde{f}$ to $S^2$ from $D$ is \emph{roughly} $H$-invariant, if for any $p\in S^2\setminus X$ and $h\in H$, the angle between the nonzero gradient vectors \[\text{grad}(\widetilde{f})\,\,\, \text{and}\,\,\, \text{grad}(\phi_h^*\widetilde{f})\] in $T_pS^2$ is less than $\pi/2$. If $\widetilde{f}$ if roughly $H$-invariant, then  for $p\notin X$, $\text{grad}(f)(p)$ is an average of a collection of nonzero vectors in the same convex half space of $T_pS^2$ and must be nonzero, implying $p\notin\text{Crit}(f)$. That is, if $\widetilde{f}$ is roughly $H$-invariant, then $\text{Crit}(f)=X$, as desired. The extensions $\widetilde{f}$ in Figures \ref{figure: dihedral}, \ref{figure: tetrahedral}, \ref{figure: octahedral}, and \ref{figure: icosahedral} are all roughly $H$-invariant by inspection, and the proof is complete. 
\end{proof}

\begin{figure}[!htb]
    \centering
    \begin{minipage}{.5\textwidth}
        \centering
        \includegraphics[width=0.9\textwidth]{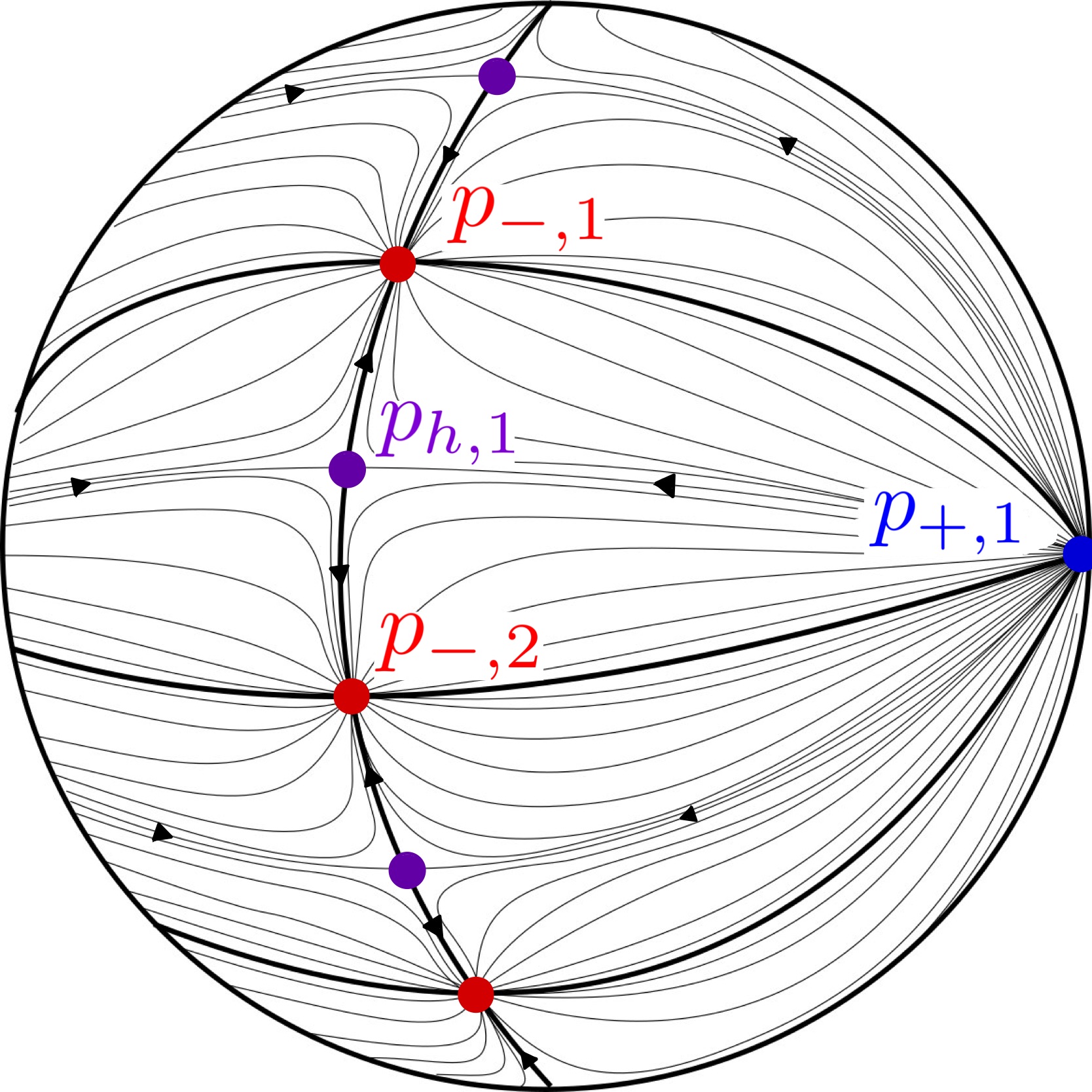}
        \caption{A dihedral $\widetilde{f}$}
        \label{figure: dihedral}
    \end{minipage}%
    \begin{minipage}{0.5\textwidth}
        \centering
        \includegraphics[width=0.9\textwidth]{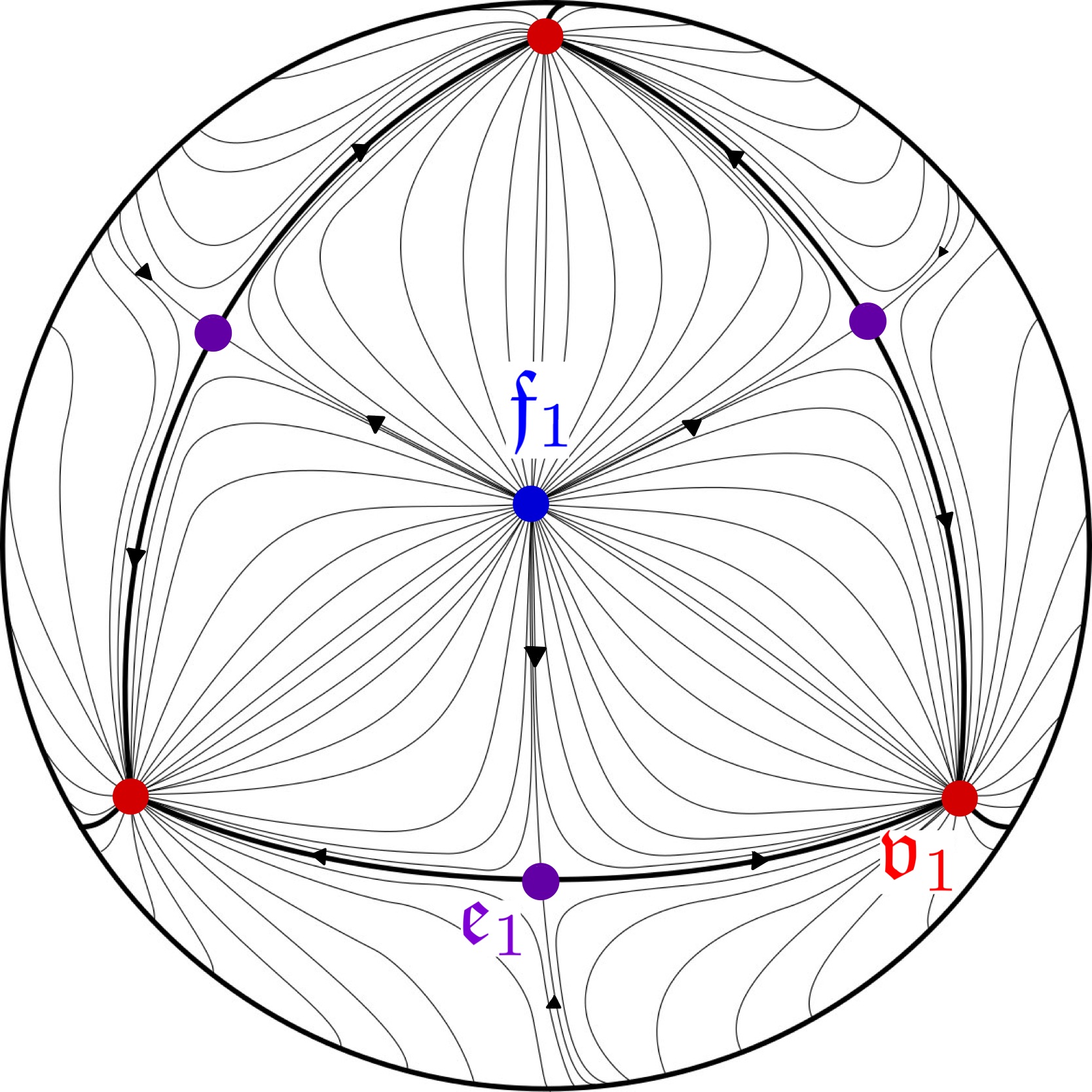}
        \caption{A tetrahedral $\widetilde{f}$}
        \label{figure: tetrahedral}
    \end{minipage}
\end{figure}

\begin{figure}[!htb]
    \centering
    \begin{minipage}{.5\textwidth}
        \centering
        \includegraphics[width=0.9\textwidth]{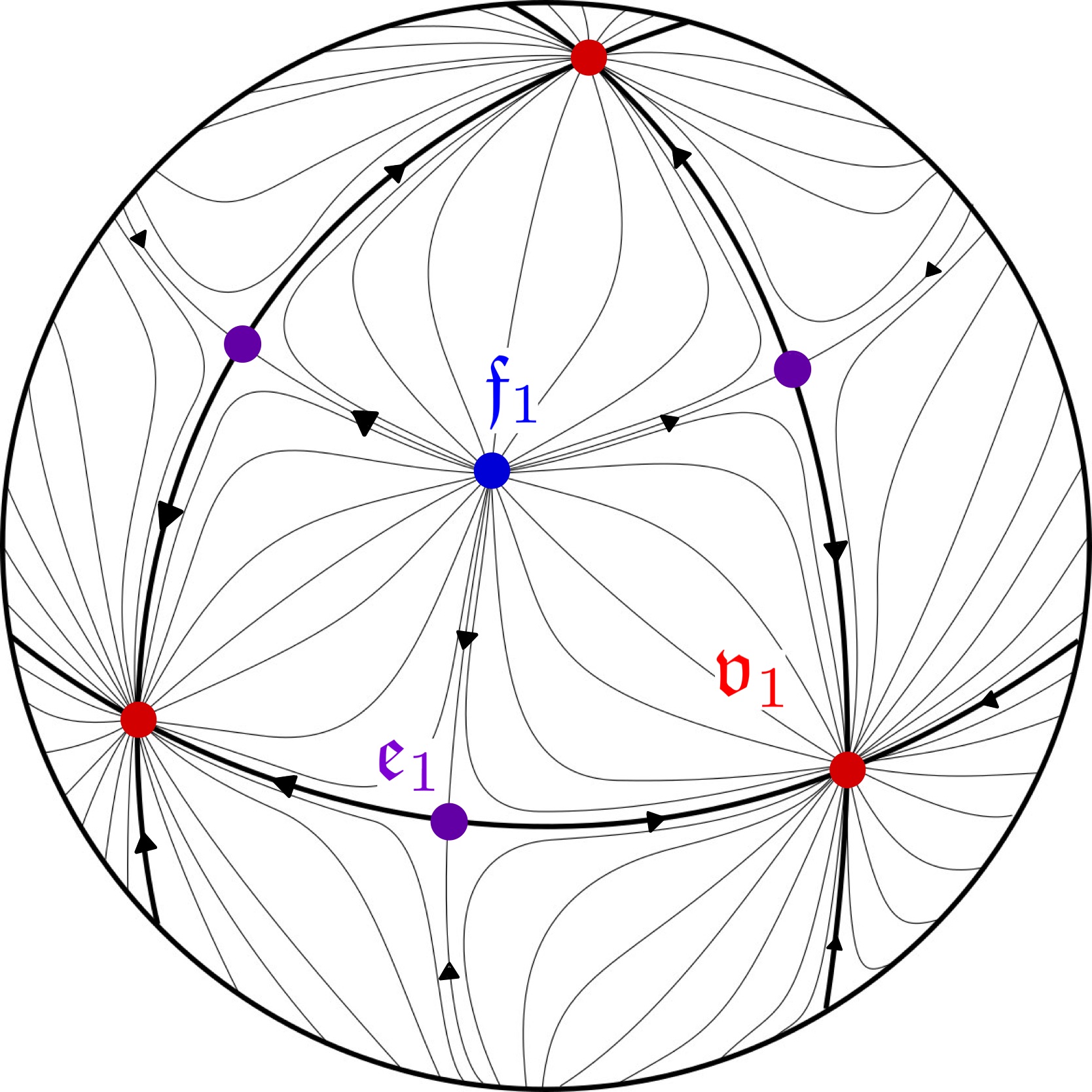}
        \caption{An octahedral $\widetilde{f}$}
        \label{figure: octahedral}
    \end{minipage}%
    \begin{minipage}{0.5\textwidth}
        \centering
        \includegraphics[width=0.9\textwidth]{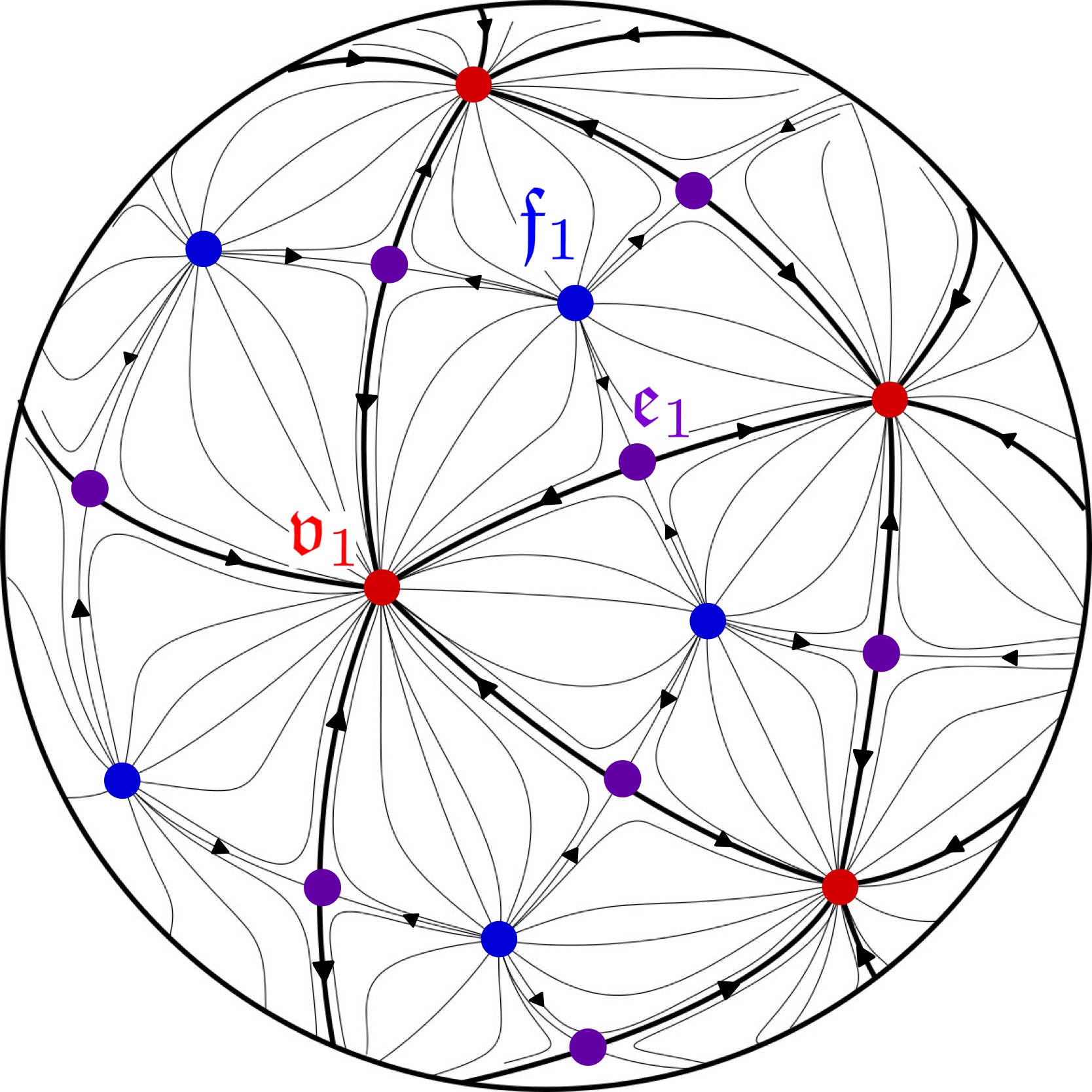}
        \caption{An icosahedral $\widetilde{f}$}
        \label{figure: icosahedral}
    \end{minipage}

\end{figure}

\begin{lemma}\label{lemma: smale}
If $f$ is a Morse function on a 2-dimensional manifold $S$ such that $f(p_1)=f(p_2)$ for all $p_1, p_2\in\mbox{\em Crit}(f)$ with Morse index 1, then $f$ is Smale, given any metric on $S$.
\end{lemma}
\begin{proof}
Given metric $g$ on $S$, $f$ fails to be Smale with respect to $g$ if and only if there are two distinct critical points of $f$ of Morse index 1 that are connected by a gradient flow line of $f$. Because all such critical points have the same $f$ value, no such flow line exists.
\end{proof}

\begin{remark}
By Lemma \ref{lemma: smale}, the Morse function $f$ provided in Lemma \ref{lemma: morse} is Smale for $\omega_{\FS}(\cdot, j\cdot)$.
\end{remark}

\subsection{Cylinders over orbifold Morse trajectories}\label{subsection: cylinders over orbifold Morse trajectories}
In Section \ref{section: computation of filtered contact homology} we will compute the action filtered cylindrical contact homology groups using the preceding set up.  In particular we will show that the grading of any generator of the filtered chain complex is even\footnote{Recall that the degree of a generator $\gamma$ is given by $|\gamma|=\mu_{\CZ}(\gamma)-1$.}, implying that the action filtered differential {vanishes}.

It is interesting to note however that not all moduli spaces of holomorphic cylinders are empty. In this section, we elucidate the correspondence between moduli spaces of certain $J$-holomorphic cylinders and the moduli spaces of {orbifold} Morse trajectories in the base; the latter is often nonempty. We establish an orbifold version of the correspondence between cylinders and flow lines, in particular constructing a holomorphic cylinder from an orbifold Morse trajectory.  We do not provide the full details as to why all holomorphic cylinders arise this way; this direction of the correspondence follows by way of the arguments as collected in \cite[\S 5]{N2} in the context of prequantization bundles, which may be invoked in the setting at hand as a result of \cite[\S 2, 4]{HN}, \cite[\S 10]{wendl-sft}, \cite{SZ}.  While nothing presented in this section is necessary to the proof of Theorem \ref{theorem: main}, the correspondence may be of value for computing other contact homology theories.

As discussed in Section \ref{subsection: cylindrical contact homology as an analogue of orbifold Morse homology}, the Seifert projection $\fp:S^3/G\to S^2/H$ highlights many of the interplays between orbifold Morse theory and cylindrical contact homology. In particular, the projection geometrically relates holomorphic cylinders in $\R\times S^3/G$ to the orbifold Morse trajectories in $S^2/H$. This necessitates a discussion about the complex structure on $\xi_{G}$ that we will use.

\begin{remark}\label{remark: the specific almost complex structure}(\emph{The canonical complex structure $J$ on $\xi_{G}$}) \\ Recall that the $G$-action on $S^3$ preserves the standard complex structure $J_{\C^2}$ on $\xi\subset TS^3$. Thus, $J_{\C^2}$ descends to a complex structure on $\xi_G$, which we will denote simply by $J$ in this section. Note that for any sufficiently small $\varepsilon >0$, and for any $f$ on $S^2$, $J_{\C^2}$ is $\lambda_{\varepsilon}$-compatible, thus $J$ is $\lambda_{G,\varepsilon}$-compatible as well.\footnote{This $J$ might not be one of the generic $J_N$ used to compute the filtered homology groups in the later Sections \ref{subsection: cyclic}, \ref{subsection: dihedral}, or \ref{subsection: polyhedral}.  A generic choice of $J_N$ is necessary to ensure transversality of the cylinders in symplectic cobordisms, which are used to define the chain maps later in Section \ref{section: direct limits of filtered homology}.}
\end{remark}

\begin{remark}\label{remark: cylinders over flow lines}(\emph{$J_{\C^2}$-holomorphic cylinders in $\R\times S^3$ over Morse flow lines in $S^2$}) \\ Fix critical points $p$ and $q$ of a Morse-Smale function $f$ on $(S^2,\omega_{\FS}(\cdot,j\cdot))$. Fix $\varepsilon>0$ sufficiently small. By \cite[Propositions 5.4, 5.5]{N2}, we have a bijective correspondence between $\mathcal{M}(p,q)$ and $\mathcal{M}^{J_{\C^2}}(\gamma_p^k,\gamma_q^k)/\R$, for any $k\in\N$, where $\gamma_p$ and $\gamma_q$ are the embedded Reeb orbits in $S^3$ projecting to $p$ and $q$ under $\fP$. Given a Morse trajectory $x\in\mathcal{M}(p,q)$, the components of the corresponding cylinder $u_x:\R\times S^1\to\R\times S^3$ are explicitly written down in \cite[\S 5]{N2} in terms of a parametrization  of $x$, the Hopf action on $S^3$, the Morse function $f$, and the horizontal lift of its gradient to $\xi$. The resulting $u_x\in\mathcal{M}^{J_{\C^2}}(\gamma_p,\gamma_q)/\R$ is $J_{\C^2}$-holomorphic.\footnote{We are abusing notation by conflating the parametrized cylindrical map $u_x$ with the equivalence class $[u_x]\in\mathcal{M}^{J_{\C^2}}(\gamma_p,\gamma_q)/\R$; we will continue to abuse notation in this way.}  Furthermore, the Fredholm index of $u_x$ agrees with that of $x$. The image of the composition  \[\R\times S^1\xrightarrow{u_x}\R\times S^3\xrightarrow{\pi_{S^3}}S^3\xrightarrow{\fP}S^2\] equals the image of $x$ in $S^2$. We call $u_x$ the \emph{cylinder over} $x$. \end{remark}

The following procedure uses Remark \ref{remark: cylinders over flow lines} and Diagram \ref{diagram: commuting square} to establish a similar correspondence between moduli spaces of {orbifold} flow lines of $S^2/H$ and moduli spaces of $J$-holomorphic cylinders in $\R\times S^3/G$, where $J$ is taken to be the $J_{\C^2}$-descended complex structure on $\xi_{G}$ from Remark \ref{remark: the specific almost complex structure}.
\begin{enumerate}
    \item Take $x\in\mathcal{M}(p,q)$, for orbifold Morse critical points $p, q\in S^2/H$ of $f_H$.
    \item Take a $\pi_H$-lift, $\widetilde{x}:\R\to S^2$ of $x$, from $\widetilde{p}$ to $\widetilde{q}$ in $S^2$, for some preimages $\widetilde{p}$ and $\widetilde{q}$ of $p$ and $q$. We have $\widetilde{x}\in\mathcal{M}(\widetilde{p},\widetilde{q})$.
    \item Let $u_{\widetilde{x}}$ be the $J_{\C^2}$-holomorphic cylinder in $\R\times S^3$ over $\widetilde{x}$ (see Remark \ref{remark: cylinders over flow lines}). We now have $u_{\widetilde{x}}\in\mathcal{M}^{J_{\C^2}}(\gamma_{\widetilde{p}},\gamma_{\widetilde{q}})/\R$.
    \item Let $u_x:\R\times S^1\to\R\times S^3/G$ denote the composition \[\R\times S^1\xrightarrow{u_{\widetilde{x}}}\R\times S^3\xrightarrow{\text{Id}\times\pi_G}\R\times S^3/G.\] Because $J$ is the $\pi_G$-descent of $J_{\C^2}$, we have that \[\text{Id}\times\pi_G:(\R\times S^3,J_{\C^2})\to(\R\times S^3/G,J)\] is a holomorphic map. This implies that $u_x$ is $J$-holomorphic; \begin{equation}\label{equation: induced cylinder}
        u_x\in\mathcal{M}^J(\pi_G\circ\gamma_{\widetilde{p}},\pi_G\circ\gamma_{\widetilde{q}})/\R.
    \end{equation}
\end{enumerate}
Note that $\pi_G\circ\gamma_{\widetilde{p}}$ and $\pi_G\circ\gamma_{\widetilde{q}}$ are contractible Reeb orbits of $\lambda_{G,\varepsilon}$ projecting to $p$ and $q$, respectively. Thus, if $\gamma_p$ and $\gamma_q$ are the embedded (potentially non-contractible) Reeb orbits of $\lambda_{G,\varepsilon}$ in $S^3/G$ over $p$ and $q$, then we have that \[\pi_G\circ\gamma_{\widetilde{p}}=\gamma_p^{m_p},\,\,\,\text{and}\,\,\,\pi_G\circ\gamma_{\widetilde{q}}=\gamma_q^{m_q},\]
where $m_p, m_q\in\N$ are the orders of $[\gamma_p]$ and $[\gamma_q]$ in $\pi_1(S^3/G)$. In particular, we can simplify equation \eqref{equation: induced cylinder}: \[u_x\in\mathcal{M}^J(\gamma_p^{m_p},\gamma_q^{m_q})/\R.\] This allows us to establish an orbifold version of the correspondence in \cite[\S 5]{N2}:
\begin{align}
    \mathcal{M}(p,q)&\cong\mathcal{M}^J(\gamma_p^{m_p},\gamma_q^{m_q})/\R \label{equation: induced cylinder 2}\\
    x&\sim u_x\nonumber.
\end{align}

\begin{remark}
In order to conclude that all holomorphic cylinders arise as lifts of orbifold Morse trajectories, one must make a straightforward modification of the arguments as explained in  \cite[\S 5]{N2}, which adapts \cite[Thm.~10.30, 10.32]{wendl-sft}, which in turn is a modification of the original arguments by Salamon and Zehnder \cite{SZ}.  The proof of \cite[Thm.~5.5]{N2} holds in the present setting as a result of the compactness results established in \cite[Prop.~2.8]{HN} and automatic transversality results \cite[Prop.~4.2(b)]{HN}, \cite{W}.
\end{remark}

\subsection{Orbifold and contact interplays: an example}\label{subsection: visualizing holomorphic cylinders: an example}
Before giving the proof of the main theorem, we continue with our digression establishing 
connections between the contact data of $S^3/G$ and the orbifold Morse data of $S^2/H$, as previously alluded to in Section \ref{subsection: cylindrical contact homology as an analogue of orbifold Morse homology}.  As before, for each $p\in\text{Crit}(f)$, select an orientation of the embedded disc $W^{-}(p)$.  The action of the stabilizer (equivalently, {isotropy}) subgroup of $p$, \[H_p:=\{h\in H\,|\, h\cdot p=p\}\subset H,\] on $S^2$ restricts to an action on $W^{-}(p)$ by diffeomorphisms. We say that the critical point $p$ is \emph{orientable} if this action is by {orientation preserving} diffeomorphisms. Let $\text{Crit}^+(f)\subset\text{Crit}(f)$ denote the set of orientable critical points.

Note that the $H$-action on $S^2$ permutes $\text{Crit}(f)$, and the action restricts to a permutation of $\text{Crit}^+(f)$. Furthermore, the index of a critical point is preserved by the action. Let $\text{Crit}(f_H)$, $\text{Crit}^+(f_H)$, and $\text{Crit}^+_k(f_H)$ denote the quotients $\text{Crit}(f)/H$, $\text{Crit}^+(f)/H$, and $\text{Crit}^+_k(f)/H$, respectively. As in the smooth case, we define the $k^{\text{th}}$-orbifold Morse chain group, denoted $CM_k^{\text{orb}}$, to be the free abelian group generated by $\text{Crit}_k^+(f_H)$. The differential will be defined by a signed {and weighted} count of negative gradient trajectories in $S^2/H$. The homology of this chain complex is, as in the smooth case, isomorphic to the singular homology of $S^2/H$ (\cite[Theorem 2.9]{CH}).

First, we demonstrate why it is necessary to discard the non-orientable critical points.

\begin{remark} \label{remark: nonorientable1}
(Discarding non-orientable critical points to recover singular homology) \\ Every index 1 critical point of $f:S^2\to\R$ depicted in Figures \ref{figure: dihedral},  \ref{figure: tetrahedral}, \ref{figure: octahedral}, and \ref{figure: icosahedral} is non-orientable. This is because the unstable submanifolds associated to each of these critical points is an open interval, and the action of the stabilizer of each such critical point is a 180-degree rotation of $S^2$ about an axis through the critical point. Thus, this action reverses the orientation of the embedded open intervals. If we were to include these index 1 critical points in the chain complex, then $CM^{\text{orb}}_*$ would have rank three, with \[CM^{\text{orb}}_0\cong CM^{\text{orb}}_1\cong CM^{\text{orb}}_2\cong\Z. \]
Note that it is not possible to define a differential on this purported chain complex with homology isomorphic to $H_*(S^2/H;\Z)\cong H_*(S^2;\Z)$. 
Indeed, the correct chain complex, obtained by discarding the non-orientable index 1 critical points, has rank two: \[CM_0^{\text{orb}}\cong CM_2^{\text{orb}}\cong\Z,\,\,\, CM_1^{\text{orb}}=0,\]  and has vanishing differential, producing isomorphic homology $H_*(S^2/H;\Z)\cong H_*(S^2;\Z)$.
\end{remark}

Next, we explain why it is necessary to discard the non-orientable critical points to orient the gradient trajectories.

\begin{remark}\label{remark: nonorientable2}
(Discarding non-orientable critical points to orient the gradient trajectories) \\ Let $p$ and $q$ be orientable critical points in $S^2$ with Morse index difference equal to 1. Let $x:\R\to S^2/H$ be a negative gradient trajectory of $f_H$ from $[p]$ to $[q]$. Because $p$ and $q$ are orientable, the value of $\epsilon(\widetilde{x})\in\{\pm1\}$ is independent of any choice of lift of $x$ to a negative gradient trajectory $\widetilde{x}:\R\to M$ of $f$ in $S^2$. We define $\epsilon(x)$ to be this value. Conversely, if one of the points $p$ or $q$ is non-orientable, then there exist two lifts of $x$ with opposite signs, making the choice $\epsilon(x)$ dependent on choice of lift.
\end{remark}
Recall that $\partial^{\text{orb}}:CM_k^{\text{orb}}\to CM_{k-1}^{\text{orb}}$ is defined as follows. Let $p\in\text{Crit}_k^+(f_H)$ and $q\in\text{Crit}_{k-1}^+(f_H)$ be orientable critical points then \[\langle \partial^{\text{orb}} p, q\rangle:=\sum_{x\in\mathcal{M}(p,q)}\epsilon(x)\frac{|H_p|}{|H_x|}\in\Z.\]  

As previously mentioned, for any finite $H\subset\SO(3)$, the quotient $S^2/H$ is an orbifold 2-sphere. When $H$ is cyclic, the orbifold 2-sphere $S^2/H$ resembles a lemon shape, featuring two orbifold points, and is immediately homeomorphic to $S^2$. If not cyclic, $H$ is dihedral, or polyhedral. A fundamental domain for the $H$-action on $S^2$ in these two latter cases can be taken to be an isosceles, closed, geodesic triangle, denoted $\Delta_{H}\subset S^2$. These geodesic triangles $\Delta_{H}$ are identified by the shaded regions of $S^2$ in Figure \ref{figure: dihedral fundamental domain} for $H=\D_{2n}$ and Figure \ref{figure: tetrahedral fundamental domain} for $H=\T$.
\begin{figure}[!htb]
    \centering
    \begin{minipage}{.5\textwidth}
        \centering
        \includegraphics[width=0.8\textwidth]{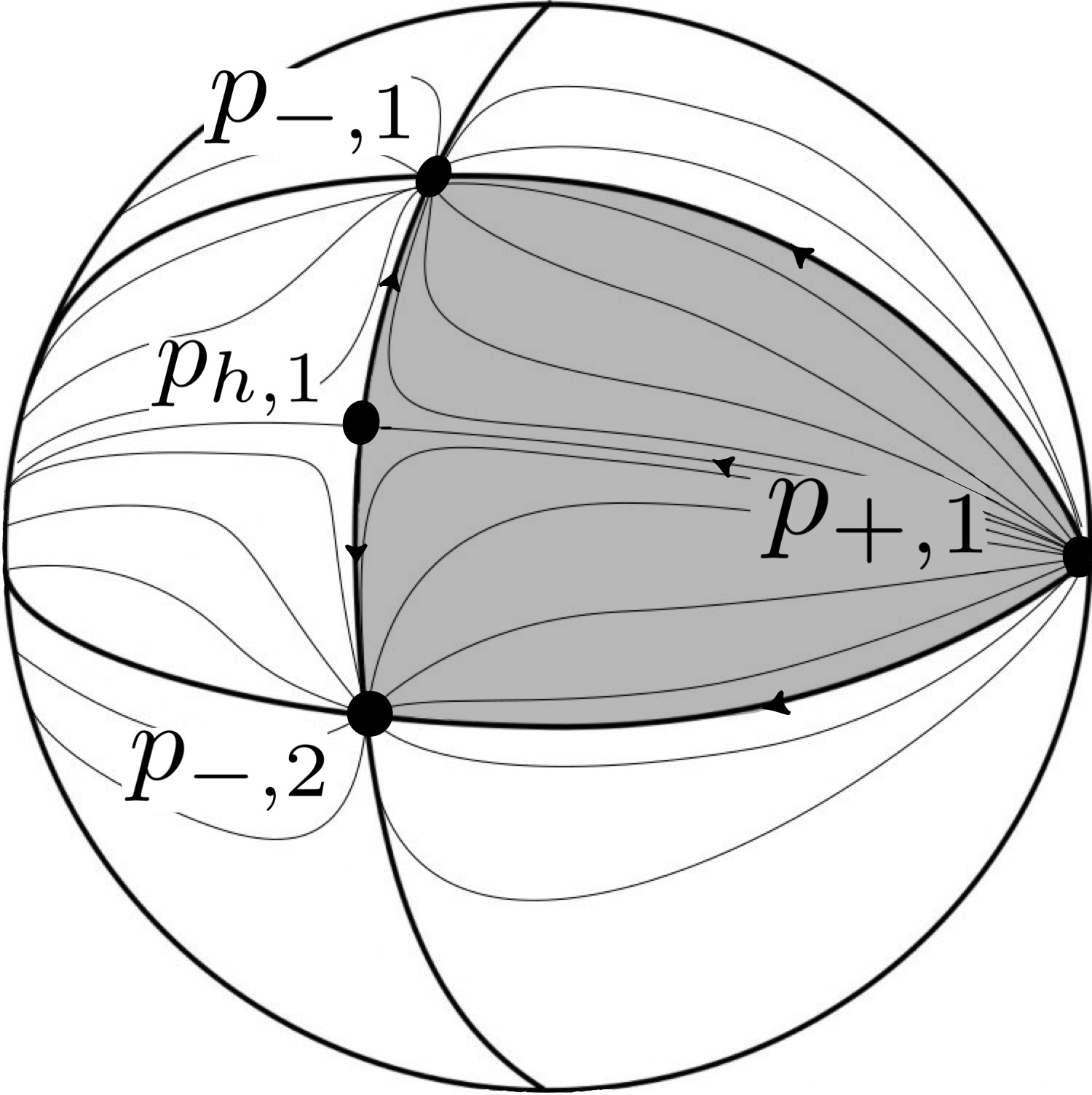}
        \caption{Fundamental domain $\Delta_{\D_{2n}}$}
        \label{figure: dihedral fundamental domain}
    \end{minipage}%
    \begin{minipage}{0.5\textwidth}
        \centering
        \includegraphics[width=0.8\textwidth]{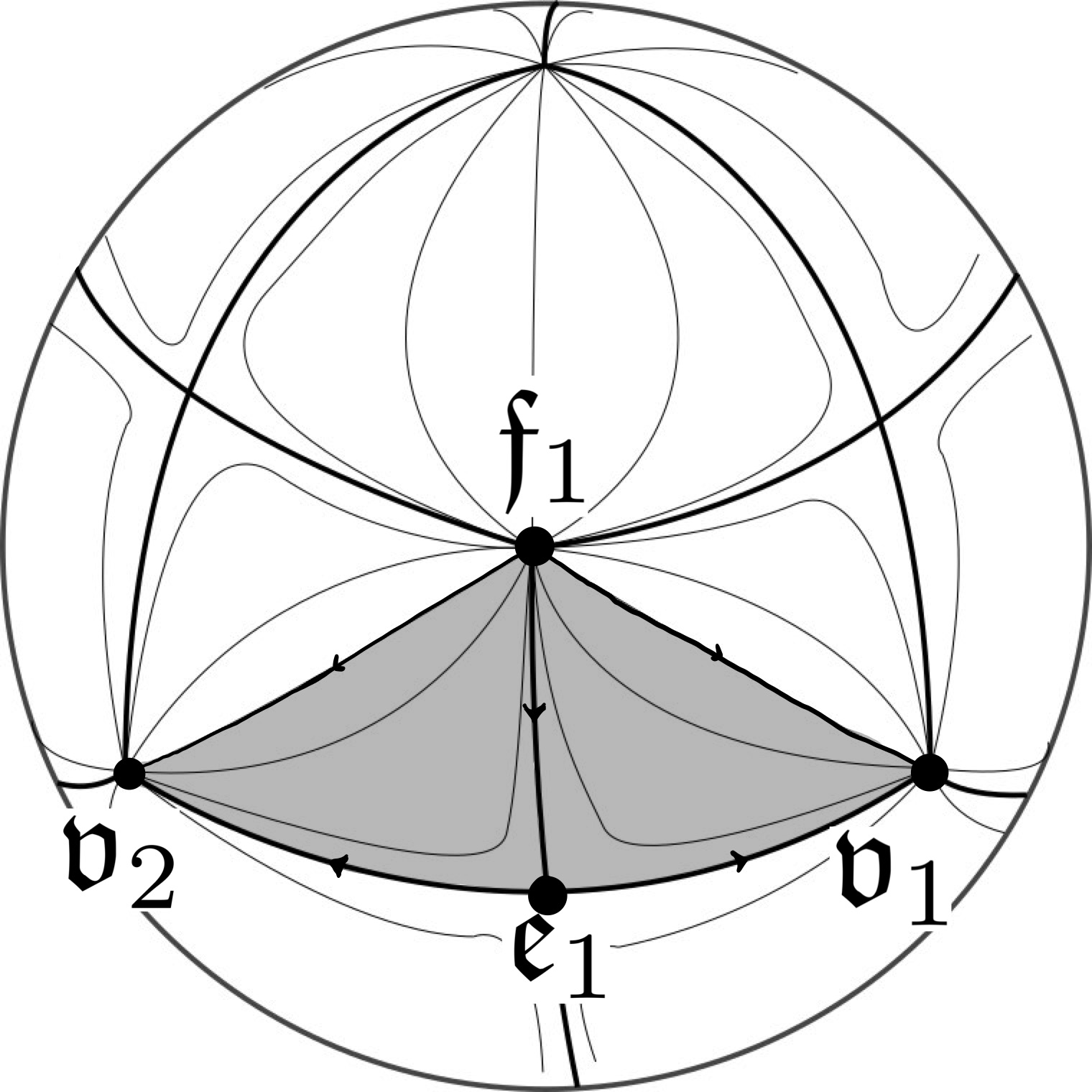}
        \caption{Fundamental domain $\Delta_{\T}$}
        \label{figure: tetrahedral fundamental domain}
    \end{minipage}
\end{figure}

\begin{itemize}
    \item In Figure \ref{figure: dihedral fundamental domain}, the three vertices of $\Delta_{\D_{2n}}$are $p_{+,1}$, $p_{-,1}$, and $p_{-,2}$. We have that $S^2$ is tessellated by $|\D_{2n}|=2n$ copies of $\Delta_{\D_{2n}}$ and the internal angles of $\Delta_{\D_{2n}}$ are $\pi/2$, $\pi/2$, and $2\pi/n$; $\Delta_{\D_{2n}}$ is isosceles.
    \item In Figure \ref{figure: tetrahedral fundamental domain},  the three vertices of $\Delta_{\T}$ are $\mathfrak{f}_1$, $\mathfrak{v}_1$, and $\mathfrak{v}_2$. We have that $S^2$ is tessellated by $|\T|=12$ copies of $\Delta_{\T}$, and the internal angles of $\Delta_{\T}$ are $\pi/3$, $\pi/3$, and $2\pi/3$; $\Delta_{\T}$ is isosceles.
\end{itemize}

The triangular fundamental domains $\Delta_{\Oc}$ and $\Delta_{\I}$ for the $\Oc$ and $\I$ actions on $S^2$ are constructed  analogously to $\Delta_{\T}$. Ultimately, in every (non-cyclic) case, we have a closed, isosceles, geodesic triangle serving as a fundamental domain for the $H$-action on $S^2$. Applying the $H$-identifications on the boundary of $\Delta_{H}$ produces $S^2/H$, a quotient that is homeomorphic to $S^2$ with three orbifold points. Specifically, under the surjective quotient map restricted to the closed $\Delta_H$, depicted in Figure \ref{figure: fundamental domain}, \[\pi_{H}|_{\Delta_H}:\Delta_{H}\subset S^2\to S^2/H.\]
In terms of the critical points we obtain:

\begin{itemize}
    \item (\textcolor{blue}{blue maximum}) one orbifold point of $S^2/H$ has a preimage consisting of a single vertex of $\Delta_H$, this is an index 2 critical point in $S^2$;
    \item (\textcolor{violet}{violet saddle}) one orbifold point of $S^2/H$ has a preimage consisting of a single midpoint of an edge of $\Delta_H$, this is an index 1 critical point in $S^2$;
    \item (\textcolor{red}{red minimum}) one orbifold point of $S^2/H$ has a preimage consisting of two vertices of $\Delta_H$; both are index 0 critical points in $S^2$. 
\end{itemize}
Figure \ref{figure: fundamental domain} depicts the attaching map for $\Delta_H$ along the boundary and these points.

\begin{figure}[h]
    \centering
    \includegraphics[width=0.8\textwidth]{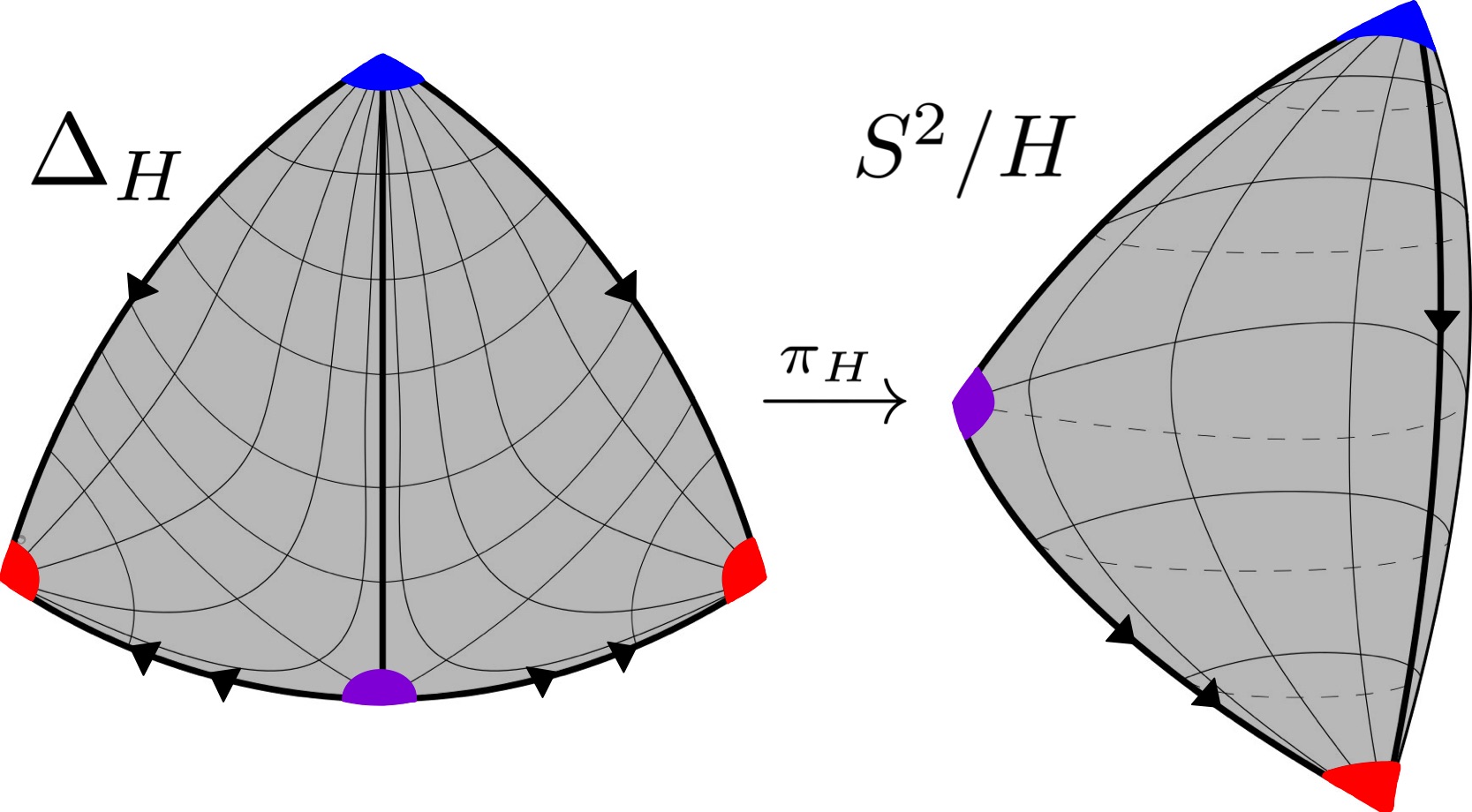}
    \caption{A triangular fundamental domain $\Delta_H$ produces $S^2/H$, a topological $S^2$ with three orbifold points when $H$ is non-cyclic. The directional markers on the boundary of $\Delta_H$ indicating the identifications under the $H$-action simultaneously realize the orbifold Morse trajectories.}
    \label{figure: fundamental domain}
\end{figure}

We now specialize to the case $H=\T$ and study the geometry of $S^3/\T^*$ and $S^2/\T$. This choice $H=\T$ makes the examples and diagrams concrete; note that a choice of $H=\D$, $\Oc$, or $\I$ produces similar geometric scenarios.  More generally it is expected that for prequantization orbibundles that the orbifold Morse flow lines are in correspondence with $S^1$-invariant holomorphic cylinders, but this has only been established in certain cases, cf. Haney-Mark \cite{HM} and Nelson-Weiler \cite{NW2}.

Using the notation to be introduced in Section \ref{subsection: polyhedral}, the three orbifold points of $S^2/\T$ are denoted $\mathfrak{v}$, $\mathfrak{e}$, and $\mathfrak{f}$, which are critical points of the orbifold Morse function $f_{\T}$ of index 0, 1, and 2, respectively. Furthermore, for small $\varepsilon>0$, we have three embedded nondegenerate Reeb orbits, $\mathcal{V}$, $\mathcal{E}$, and $\mathcal{F}$ in $S^3/\T^*$ of $\lambda_{\T^*,\varepsilon}$, projecting to the respective orbifold critical points under $\fp:S^3/\T^*\to S^2/\T$. Figure \ref{figure: calzone} illustrates this data. 

\begin{figure}[h]
    \centering
    \includegraphics[width=0.85\textwidth]{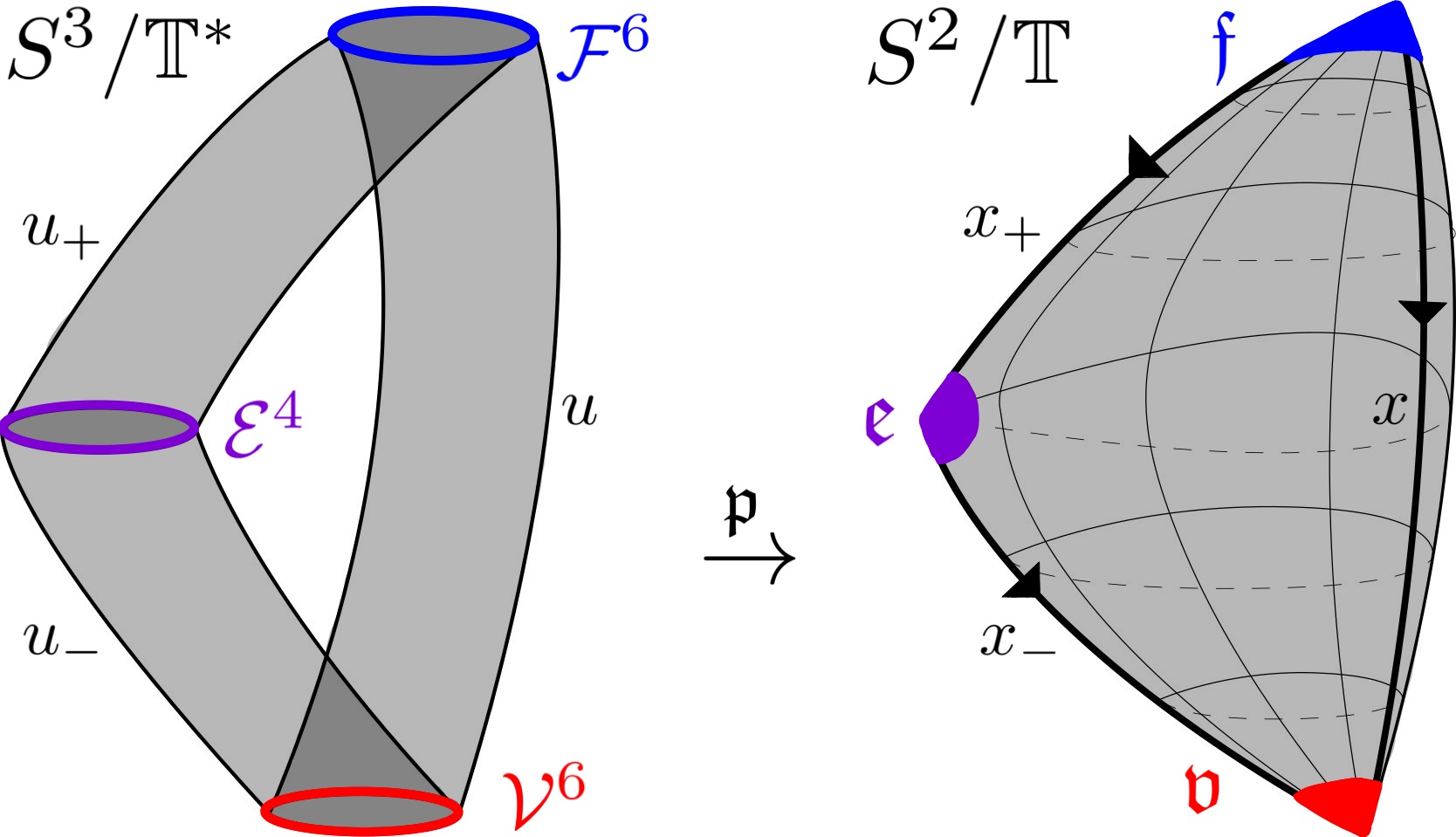}
    \caption{The Seifert projection $\fp$ takes Reeb orbits and cylinders of $S^3/\T^*$ to orbifold critical points and Morse trajectories of $S^2/\T$. The depicted cylinders in $S^3/\T^*$ should be understood as the images of infinite cylinders in the symplectization under the projection $S^3/\T^*\times\R\to S^3/\T^*$.}
    \label{figure: calzone}
\end{figure}

In Section \ref{subsection: cylindrical contact homology as an analogue of orbifold Morse homology} we explained how bad Reeb orbits in cylindrical contact homology are analogous to non-orientable critical points in orbifold Morse theory. We explicitly realize this analogy geometrically with $\fp:S^3/\T^*\to S^2/\T$. In Section \ref{subsection: polyhedral}, we will show that the even iterates $\mathcal{E}^{2k}$ are examples of bad Reeb orbits, and by Remark \ref{remark: nonorientable1}, $\mathfrak{e}$ is a non-orientable critical point of $f_{\T}$. The projection $\fp$ maps the bad Reeb orbits $\mathcal{E}^{2k}$ to the non-orientable critical point $\mathfrak{e}$.

 Next we consider the relationships between the moduli spaces of $J$-holomorphic cylinders and gradient flow lines. The orders of $\mathcal{V}$, $\mathcal{E}$, and $\mathcal{F}$ in $\pi_1(S^3/\T^*)$ are 6, 4, and 6, respectively (see Section \ref{subsubsection:  binary polyhedral}). Thus, by the correspondence \eqref{equation: induced cylinder 2} in Section \ref{subsection: cylinders over orbifold Morse trajectories}, we have the following identifications between moduli spaces of orbifold Morse flow lines of $S^2/\T$ and $J$-holomorphic cylinders in $\R\times S^3/\T^*$ (with respect to the complex structure $J$ described in Remark \ref{remark: the specific almost complex structure}):
\begin{align}
    \mathcal{M}(\mathfrak{f},\mathfrak{e}) &\cong\mathcal{M}^J(\mathcal{F}^6,\mathcal{E}^4)/\R \label{equation: correspondence +},\\
    \mathcal{M}(\mathfrak{e},\mathfrak{v}) &\cong\mathcal{M}^J(\mathcal{E}^4,\mathcal{V}^6)/\R \label{equation: correspondence -},\\
    \mathcal{M}(\mathfrak{f},\mathfrak{v}) &\cong\mathcal{M}^J(\mathcal{F}^6,\mathcal{V}^6)/\R \label{equation: correspondence 0}.
\end{align}
Correspondences \eqref{equation: correspondence +} and \eqref{equation: correspondence -} are between singleton sets. Indeed, let $x_+$ be the unique orbifold Morse flow line from $\mathfrak{f}$ to $\mathfrak{e}$ in $S^2/\T$, and let $x_-$ be the unique orbifold Morse flow line from $\mathfrak{e}$ to $\mathfrak{v}$, depicted in Figure \ref{figure: calzone}. Then we have corresponding cylinders $u_+$ and $u_-$ from $\mathcal{F}^6$ to $\mathcal{E}^4$, and from $\mathcal{E}^4$ to $\mathcal{V}^6$, respectively:

\begin{align}
    \{x_+\}=\mathcal{M}(\mathfrak{f},\mathfrak{e}) &\cong\mathcal{M}^J(\mathcal{F}^6,\mathcal{E}^4)/\R=\{u_+\},\label{equation: singleton1}\\
    \{x_-\}=\mathcal{M}(\mathfrak{e},\mathfrak{v}) &\cong\mathcal{M}^J(\mathcal{E}^4,\mathcal{V}^6)/\R=\{u_-\}.\label{equation: singleton2}
\end{align}
These cylinders $u_{\pm}$ are depicted in Figure \ref{figure: calzone}. Additionally, note that the indices of the corresponding objects agree:
\begin{align*}
    \text{ind}(x_+)&=\text{ind}_{f_{\T}}(\mathfrak{f})-\text{ind}_{f_{\T}}(\mathfrak{e})=2-1=1, \\
    \text{ind}(u_+)&=\mu_{\CZ}(\mathcal{F}^6)-\mu_{\CZ}(\mathcal{E}^4)=5-4=1,
\end{align*}
and 
\begin{align*}
    \text{ind}(x_-)&=\text{ind}_{f_{\T}}(\mathfrak{e})-\text{ind}_{f_{\T}}(\mathfrak{v})=1-0=1, \\
    \text{ind}(u_-)&=\mu_{\CZ}(\mathcal{E}^4)-\mu_{\CZ}(\mathcal{V}^6)=4-3=1.
\end{align*}

Next, we consider the third correspondence of moduli spaces in \eqref{equation: correspondence 0}. As in Figure \ref{figure: calzone}, we have that $\mathcal{M}(\mathfrak{f},\mathfrak{v})$ is diffeomorphic to a 1-dimensional open interval. For any $x\in\mathcal{M}(\mathfrak{f},\mathfrak{v})$, we have that \[\text{ind}(x)=\text{ind}_{f_{\T}}(\mathfrak{f})-\text{ind}_{f_{\T}}(\mathfrak{v})=2,\] algebraically verifying that the moduli space of orbifold flow lines must be 1-dimensional. On the other hand, take any cylinder $u\in\mathcal{M}^J(\mathcal{F}^6,\mathcal{V}^6)/\R$. Now, 
\[\text{ind}(u)=\mu_{\CZ}(\mathcal{F}^6)-\mu_{\CZ}(\mathcal{V}^6)=5-3=2,\]
 verifying that this moduli space of cylinders is $1$-dimensional, as expected by \eqref{equation: correspondence 0}. 

Note that both of the open 1-dimensional moduli spaces in \eqref{equation: correspondence 0} admit a compactification by broken objects. We can see explicitly from Figure \ref{figure: calzone} that both ends of the 1-dimensional moduli space $\mathcal{M}(\mathfrak{f},\mathfrak{v})$ converge to the {same} once-broken orbifold Morse trajectory, $(x_+,x_-)\in\mathcal{M}(\mathfrak{f},\mathfrak{e})\times\mathcal{M}(\mathfrak{e},\mathfrak{v})$. In particular, the compactification $\overline{\mathcal{M}(\mathfrak{f},\mathfrak{v})}$ is a topological $S^1$, obtained by adding the single point $(x_+,x_-)$ to an open interval, and we write \[\overline{\mathcal{M}(\mathfrak{f},\mathfrak{v})}=\mathcal{M}(\mathfrak{f},\mathfrak{v})\,\bigsqcup\, \bigg(\mathcal{M}(\mathfrak{f},\mathfrak{e})\times\mathcal{M}(\mathfrak{e},\mathfrak{v})\bigg). \]

An identical phenomenon occurs for the compactification of the cylinders. That is, both ends of the 1-dimensional interval $\mathcal{M}^J(\mathcal{F}^6,\mathcal{V}^6)/\R$ converge to the {same} once broken building of cylinders $(u_+,u_-)$. The compactification $\overline{\mathcal{M}^J(\mathcal{F}^6,\mathcal{V}^6)/\R}$ is a topological $S^1$, obtained by adding a single point $(u_+,u_-)$ to an open interval, and we write \[\overline{\mathcal{M}^J(\mathcal{F}^6,\mathcal{V}^6)/\R}=\mathcal{M}^J(\mathcal{F}^6,\mathcal{V}^6)/\R\,\bigsqcup\,\bigg(\mathcal{M}^J(\mathcal{F}^6,\mathcal{E}^4)/\R\times\mathcal{M}^J(\mathcal{E}^4,\mathcal{V}^6/\R)\bigg).\]

In Section \ref{subsection: cylindrical contact homology as an analogue of orbifold Morse homology}, we argued that the differentials of cylindrical contact homology and orbifold Morse homology are structurally identical due to the similarities in the compactifications of the 1-dimensional moduli spaces. This is due to the fact that in both theories, a once broken building can serve as a limit of {multiple} ends of a 1-dimensional moduli space. Our examples depict this phenomenon:
\begin{itemize}
    \item The broken building $(x_+,x_-)$ of orbifold Morse flow lines serves as the limit of {both} ends of the open interval $\mathcal{M}(\mathfrak{f},\mathfrak{v})$.
    \item The broken building $(u_+,u_-)$ of $J$-holomorphic cylinders serves as the limit of {both} ends of the open interval $\mathcal{M}^J(\mathcal{F}^6,\mathcal{V}^6)/\R$.
\end{itemize}

Another analogy is highlighted in this example. In both homology theories, it is possible for a sequence of flow lines or cylinders between orientable objects to break along an intermediate non-orientable object (see \cite[Example 2.10]{CH}). For example:
\begin{itemize}
    \item There is a  sequence $x_n\in\mathcal{M}(\mathfrak{f},\mathfrak{v})$ converging to the broken building $(x_+,x_-)$, which breaks at $\mathfrak{e}$. The critical points $\mathfrak{f}$ and $\mathfrak{v}$ are orientable, whereas $\mathfrak{e}$ is non-orientable.
    \item There is a  sequence $u_n\in\mathcal{M}^J(\mathcal{F}^6,\mathcal{V}^6)/\R$ converging to the broken building $(u_+,u_-)$, which breaks along the orbit $\mathcal{E}^4$. The Reeb orbits $\mathcal{F}^6$ and $\mathcal{V}^6$ are good, whereas $\mathcal{E}^4$ is bad.
\end{itemize}

As explained in Remark \ref{remark: nonorientable2}, we cannot assign a value $\epsilon(x_{\pm})\in\{\pm1\}$ nor a value $\epsilon(u_{\pm})\in\{\pm1\}$ to the objects $x_{\pm}$ and $u_{\pm}$, because they have a non-orientable limiting object. This difficulty complicates the proof that $\partial^2=0$, as we cannot write the signed count as a sum involving terms of the form $\epsilon(x_+)\epsilon(x_-)$ or $\epsilon(u_+)\epsilon(u_-)$. One can nevertheless show that a once broken building breaking along a non-orientable object is utilized by an {even} number of ends of the 1-dimensional moduli space, and that a cyclic group action on this set of even number of ends interchanges the orientations. Using this fact, one shows $\partial^2=0$ (see \cite[Remark 5.3]{CH} and \cite[\S 4.4]{HN}). We see this explicitly in our examples:
\begin{itemize}
    \item The once broken building $(x_+,x_-)$ is the limit of {two} ends of the 1-dimensional moduli space of orbifold trajectories; one positive and one negative end.
    \item The once broken building $(u_+,u_-)$ is the limit of {two} ends of the 1-dimensional moduli space of orbifold trajectories; one positive and one negative end.
\end{itemize}
\begin{remark}\label{remark: when d squared is not zero} (Including non-orientable objects complicates $\partial^2=0$) In Section \ref{subsection: cylindrical contact homology as an analogue of orbifold Morse homology} we saw that one discards bad Reeb orbits and non-orientable orbifold critical points as generators in cylindrical contact homology and orbifold Morse homology respectively, in order to achieve $\partial^2=0$. Using our understanding of moduli spaces from  Figure \ref{figure: calzone}, we show why $\partial^2=0$ could not reasonably hold if we were to include the non-orientable critical point $\mathfrak{e}$ and bad Reeb orbit $\mathcal{E}^4$ in the corresponding chain complexes. Suppose that we have some coherent way of assigning $\pm1$ to the trajectories $x_{\pm}$ and cylinders $u_{\pm}$. Now, due to equations \eqref{equation: singleton1} and \eqref{equation: singleton2}, we would have in the orbifold case
\begin{align*}
    \partial^{\text{orb}}\mathfrak{f}&=\dfrac{\epsilon(x_+)|\T_{\mathfrak{f}}|}{|\T_{x_+}|}\mathfrak{e}=3\epsilon(x_+)\mathfrak{e}\\
    \partial^{\text{orb}}\mathfrak{e}&=\dfrac{\epsilon(x_-)|\T_{\mathfrak{e}}|}{|\T_{x_-}|}\mathfrak{v}=2\epsilon(x_-)\mathfrak{v}\\
    \implies &\langle \left( \partial^{\text{orb}} \right )^2 \mathfrak{f},\mathfrak{v}\rangle=6\epsilon(x_+)\epsilon(x_-)\neq0,
\end{align*}
where we have used that $|\T_{\mathfrak{f}}|=3$, $|\T_{\mathfrak{e}}|=2$, and $|\T_{x_{\pm}}|=1$. Similarly, again due to equations \eqref{equation: singleton1} and \eqref{equation: singleton2}, we would have
\begin{align*}
    \partial\mathcal{F}^6&=\dfrac{\epsilon(u_+)d(\mathcal{F}^6)}{d(u_+)}\mathcal{E}^4=3\epsilon(u_+)\mathcal{E}^4\\
    \partial\mathcal{E}^4&=\dfrac{\epsilon(u_-)d(\mathcal{E}^4)}{d(u_-)}\mathcal{V}^6=2\epsilon(u_-)\mathcal{V}^6\\
    \implies &\langle  \partial^2\mathcal{F}^6,\mathcal{V}^6\rangle=6\epsilon(u_+)\epsilon(u_-)\neq0,
\end{align*}
where we have used that $d(\mathcal{F}^6)=6$, $d(\mathcal{E}^4)=4$, and $d(u_{\pm})=2$.
\end{remark}

Recall that the multiplicity of a $J$-holomorphic cylinder $u$ divides the multiplicity of the limiting Reeb orbits $\gamma_{\pm}$. The following remark uses the example of this section to demonstrate it need not be the case that $d(u)=\text{GCD}(d(\gamma_+),d(\gamma_-))$. 
\begin{remark}\label{remark: multiplicity of u is not always gcd} 
By \eqref{equation: correspondence 0}, $\mathcal{M}^J(\mathcal{F}^6,\mathcal{V}^6)/\R$ is nonempty, and must contain some $u$. We see that $m(\mathcal{F}^6)=m(\mathcal{V}^6)=6$, so that \[\text{GCD}(d(\mathcal{F}^6),d(\mathcal{V}^6))=6.\] Suppose for contradiction's sake that that $d(u)=6$, and consider the underlying somewhere injective $J$-holomorphic cylinder $v$. It must be the case that $u=v^6$ and that $v\in\mathcal{M}^J(\mathcal{F},\mathcal{V})/\R$. The existence of such a $v$ implies that $\mathcal{F}$ and $\mathcal{V}$ represent the same free homotopy class of loops in $S^3/\T^*$. However, we will determine in Section \ref{subsubsection: binary polyhedral} that these Reeb orbits represent distinct homotopy classes (see Table \ref{table: Tetrahedral homotopy classes of Reeb orbits}). Thus, $d(u)$ is not equal to the GCD.
\end{remark}

\section{Filtered cylindrical contact homology}\label{section: computation of filtered contact homology}
A finite subgroup of $\SU(2)$ is either cyclic, conjugate to the binary dihedral group $\D^*_{2n}$, or conjugate to a binary polyhedral group $\T^*$, $\Oc^*$, or $\I^*$. If  subgroups $G_1$ and $G_2$ satisfy $A^{-1}G_2A=G_1$, for $A\in\SU(2)$, then the map $S^3\to S^3$, $p\mapsto A\cdot p$, descends to a strict conactomorphism   $(S^3/G_1,\lambda_{G_1}) \to (S^3/G_2,\lambda_{G_2})$, preserving the Reeb dynamics. Thus, we compute the contact homology of $S^3/G$ for a particular choice of $G$. 
The action threshold used to compute the filtered homology depends on $G$; for $N\in\N$, $L_N\in\R$ is given by: 
\begin{enumerate}[(i)]
     \itemsep-.35em
    \item $2\pi N-\frac{\pi}{n}$ when $G$ is cyclic of order $n$;
    \item $2\pi N-\frac{\pi}{2n}$ when $G$ is conjugate to $\D_{2n}^*$;
    \item $2\pi N-\frac{\pi}{10}$ when $G$ is conjugate to $\T^*$, $\Oc^*$, or $\I^*$.
\end{enumerate}

\subsection{Cyclic subgroups} \label{subsection: cyclic}

From \cite[Theorem 1.5]{AHNS}, the positive $S^1$-equivariant symplectic homology of the link of the $A_n$ singularity, $L_{A_n}\subset\C^3$, with contact structure $\xi_0=TL_{A_n}\cap J_{\C^3}(TL_{A_n})$, satisfies:
\[SH_*^{+,S^1}(L_{A_n},\xi_0)=\begin{cases} \Q^{n} & *=1,\\ \Q^{n+1} & *\geq3\,\, \mbox{and odd} \\ 0 &\mbox{else.}\end{cases}\]
Furthermore, \cite{BO} prove that there are restricted classes of contact manifolds whose cylindrical contact homology (with a degree shift) is isomorphic to its positive $S^1$-equivariant symplectic homology, when both are defined over $\Q$ coefficients. Indeed, we note an isomorphism by inspection when we compare this symplectic homology with the cylindrical contact homology of $(L_{A_n},\xi_0)\cong(S^3/G,\xi_G)$ for $G\cong\Z_{n+1}$ from Theorem \ref{theorem: main}. Although this cylindrical contact homology is computed in \cite[Theorem 1.36]{N2}, we recompute these groups using a direct limit of filtered contact homology to present the general structure of the computations to come in the dihedral and polyhedral cases.

Let $G\cong\Z_{n}$ be a finite cyclic subgroup of $G$ of order $n$. If $|G|=n$ is even, with $n=2m$, then $P:G\to H:=P(G)$ has nontrivial two element kernel, and $H$ is cyclic of order $m$. Otherwise, $n$ is odd, $P:G\to H=P(G)$ has trivial kernel, and $H$ is cyclic of order $n$. 

By conjugating $G$ if necessary, we can assume that $H$ acts on $S^2$ by rotations around the vertical axis through $S^2$. The height function  $f:S^2\to[-1,1]$ is Morse, $H$-invariant, and provides precisely two fixed points; the north pole, $\mathfrak{n}\in S^2$ featuring $f(\mathfrak{n})=1$,  and the south pole, $\mathfrak{s}\in S^2$ where $f(\mathfrak{s})=-1$.  For small $\varepsilon$, we can expect to see iterates of two embedded Reeb orbits, denoted $\gamma_{\mathfrak{n}}$ and $\gamma_{\mathfrak{s}}$, of $\lambda_{G,\varepsilon}:=(1+\varepsilon\mathfrak{p}^*f_H)\lambda_{G}$ in $S^3/G$ as the only generators of the filtered chain groups. Both $\gamma_{\mathfrak{n}}$ and $\gamma_{\mathfrak{s}}$ are elliptic and parametrize the exceptional fibers in $S^3/G$ over the two orbifold points of $S^2/H$.

Select $N\in \N$. Lemma \ref{lemma: ActionThresholdLink} produces an $\varepsilon_N>0$ for which if $\varepsilon\in(0,\varepsilon_N]$, then all orbits in $\mathcal{P}^{L_N}(\lambda_{G,\varepsilon})$ are nondegenerate and are iterates of $\gamma_{\mathfrak{s}}$ or $\gamma_{\mathfrak{n}}$,  whose actions satisfy 
\begin{equation}\label{equation: action cyclic}
    \mathcal{A}(\gamma_{\mathfrak{s}}^k)=\frac{2\pi k(1-\varepsilon)}{n},\,\,\,\,\,\mathcal{A}(\gamma_{\mathfrak{n}}^k)=\frac{2\pi k(1+\varepsilon)}{n}.
\end{equation} Thus the $L_N$-filtered chain complex is $\Q$-generated by the Reeb orbits $\gamma_{\mathfrak{s}}^k$ and $\gamma_{\mathfrak{n}}^k$, for $1\leq k<nN$. With respect to the trivialization $\tau_G$, the rotation numbers of $\gamma_{\mathfrak{s}}$ and $\gamma_{\mathfrak{n}}$  satisfy
 \[\theta_{\mathfrak{s}}=\frac{2}{n}-\varepsilon\frac{1}{n(1-\varepsilon)},\,\,\,\,\,\,\theta_{\mathfrak{n}}=\frac{2}{n}+\varepsilon\frac{1}{n(1+\varepsilon)}.\] 
 (See \cite[\S 2.2]{N2} for a definition of rotation numbers.) If  $\varepsilon_N$ is sufficiently small then \begin{equation}\label{equation: CZ indices cyclic}
    \mu_{\CZ}(\gamma_{\mathfrak{s}}^k)=2\Bigl\lceil \frac{2k}{n} \Bigr\rceil-1,\,\,\,\,\,\mu_{\CZ}(\gamma_{\mathfrak{n}}^k)=2\Bigl\lfloor \frac{2k}{n} \Bigr\rfloor+1.
\end{equation}
For $i\in\Z$, let $c_i$ denote the number of $\gamma\in\mathcal{P}^{L_N}(\lambda_{G,\varepsilon_N})$  with $|\gamma|=\mu_{\CZ}(\gamma)-1=i$. Then:
\begin{itemize}
    \itemsep-.35em
    \item $c_0=n-1$,    
    \item $c_{2i}=n$ for $0<i<2N-1$,
    \item $c_{4N-2}=n-1$,
    \item $c_i=0$ for all other $i$ values.
\end{itemize} 
By \eqref{equation: CZ indices cyclic}, if $\mu_{\CZ}(\gamma^k_{\mathfrak{n}})=1$, then $k<n/2$, so the orbit $\gamma^k_{\mathfrak{n}}$ is not contractible. If $\mu_{\CZ}(\gamma^k_{\mathfrak{s}})=1$, then by \eqref{equation: CZ indices cyclic}, $k\leq n/2$ and  $\gamma^k_{\mathfrak{s}}$ is not contractible. Thus,  $\lambda_N:=\lambda_{G,\varepsilon_N}$ is $L_N$-dynamically convex and so by \cite[Thm.~1.3]{HN},\footnote{One hypothesis of \cite[Thm.~1.3]{HN} requires that all contractible Reeb orbits $\gamma$ satisfying $\mu_{\CZ}(\gamma)=3$ must be embedded. This fails in our case by considering the contractible $\gamma_{\mathfrak{s}}^n$, which is not embedded yet satisfies $\mu_{\CZ}(\gamma_{\mathfrak{s}}^n)=3$. In the arXiv v2 of this paper, we proved in \S 4.3 why we do not need this additional hypothesis.} a generic choice $J_N\in\mathcal{J}(\lambda_N)$ provides a well defined filtered chain complex, yielding the isomorphism 
\begin{align*}
    CH_*^{L_N}(S^3/G,\lambda_N, J_N)&\cong\bigoplus_{i=0}^{2N-1}\Q^{n-2}[2i]\oplus\bigoplus_{i=0}^{2N-2} H_*(S^2;\Q)[2i]\\
    &=\begin{cases} \Q^{n-1} & *=0, \, 4N-2 \\
\Q^{n} & *=2i, 0<i<2N-1 \\ 0 &\mbox{else.}\end{cases}
\end{align*}
 This follows from the good contributions to  $c_i$, which is 0 for odd $i$, implying $\partial^{L_N}=0$. This proves Theorem \ref{theorem: main} in the cyclic case, because $G$ is abelian and $|\text{Conj}(G)|=n$, after appealing to Theorem \ref{theorem: cobordisms induce inclusions}, which permits taking a direct limit over inclusions of these groups.

\subsection{Binary dihedral groups $\D^*_{2n}$}\label{subsection: dihedral}

The binary dihedral group $\D^*_{2n}\subset \SU(2)$ has order $4n$ and projects to the dihedral group $\D_{2n}\subset \SO(3)$, which has order $2n$, under the cover $P:\SU(2)\to \SO(3)$. With the quantity $n$ understood, these groups will respectively be denoted $\D^*$ and $\D$. The group $\D^*$ is generated by the two matrices\[A=\begin{pmatrix} \zeta_n & 0 \\ 0 & \overline{\zeta_n}\end{pmatrix} \hspace{1cm} B=\begin{pmatrix} 0 & -1 \\ 1 & 0\end{pmatrix},\]
where $\zeta_n:=\exp(i\pi/n)$ is a primitive $2n^{\text{th}}$ root of unity. These matrices satisfy the relations $B^2=A^n=-\text{Id}$ and $BAB^{-1}=A^{-1}$. The group elements may be enumerated as follows: \[\D^*=\{A^kB^l:0\leq k< 2n, \, 0\leq l\leq 1 \}.\] 
By applying \eqref{equation: P in coordinates}, the following matrices generate $\D\subset\SO(3)$:
\[a:=P(A)=\begin{pmatrix} 1 & 0 & 0\\ 0 & \cos{(2\pi/n)} & -\sin{(2\pi/n)} \\ 0 & \sin{(2\pi/n)} & \cos{(2\pi/n)} \end{pmatrix} \hspace{1cm} b:=P(B)=\begin{pmatrix}-1 & 0 & 0 \\ 0 & 1 & 0 \\ 0 & 0 & -1\end{pmatrix}.\]
 There are three types of fixed points in $S^2$ of the $\D$-action, categorized as follows:
 
\vspace{.25cm}

\textbf{Morse index 0}: $p_{-,k}:=(0,\cos{((\pi+2k\pi)/n)}, \sin{((\pi+2k\pi)/n)})\in \text{Fix}(\D)$, for $k\in\{1,\dots, n\}$. We have that $p_{-,n}$ is a fixed point of $ab$ and $p_{-,k}=a^k\cdot p_{-,n}$. Thus, $p_{-,k}$ is a fixed point of $a^k(ab)a^{-k}=a^{2k+1}b$. These $n$ points enumerate a $\D$-orbit in $S^2$, and so the isotropy subgroup of $\D$ associated to any of the $p_{-,k}$ is of order 2 and is generated by $a^{2k+1}b\in\D$. The point $p_-\in S^2/\D$ denotes the image of any $p_{-,k}$ under $\pi_{\D}$.
    
    \vspace{.25cm}

    \textbf{Morse index 1}: $p_{h,k}:=(0,\cos{(2k\pi/n)}, \sin{(2k\pi/n)})\in \text{Fix}(\D)$, for $k\in\{1,\dots, n\}$. We have that $p_{h,n}$ is a fixed point of of $b$ and  $p_{h,k}=a^k\cdot p_{h,n}$. Thus, $p_{h,k}$ is a fixed point of $a^k(b)a^{-k}=a^{2k}b$. These $n$ points enumerate a $\D$-orbit in $S^2$, and so the isotropy subgroup of $\D$ associated to any of the $p_{h,k}$ is of order 2 and is generated by  $a^{2k}b\in\D$. The point $p_h\in S^2/\D$ denotes the image of any $p_{h,k}$ under $\pi_{\D}$. 
    
        \vspace{.25cm}

    \textbf{Morse index 2}: $p_{+,1}=(1,0,0)$ and $p_{+,2}=(-1,0,0)$. These  are the fixed points of $a^k$, for $0<k<n$, and together enumerate a single two element $\D$-orbit. The isotropy subgroup associated to either of the points is cyclic of order $n$ in $\D$, generated by $a$.  The point $p_+\in S^2/\D$ denotes the image of any one of these two points under $\pi_{\D}$.
     \vspace{.25cm}

There exists a $\D$-invariant, Morse-Smale function $f$ on $(S^2,\omega_{\FS}(\cdot,j \cdot))$, with $\text{Crit}(f)=\text{Fix}(\D)$, which descends to an orbifold Morse function $f_{\D}:S^2/\D\to\R$, constructed in Section \ref{appendix: constructing morse functions}. Furthermore, there are stereographic coordinates at:
\begin{enumerate}[(i)]
    \itemsep-.35em
    \item the points $p_{-,k}$, in which $f$ takes the form $(x^2+y^2)/2-1$ near $(0,0)$;
    \item the points $p_{h,k}$, in which $f$  takes the form $(x^2-y^2)/2$  near $(0,0)$;
    \item the  points $p_{+,k}$, in which $f$ takes the form $1-(x^2+y^2)/2$  near $(0,0)$.
\end{enumerate}
The orbifold surface $S^2/\D$ is homeomorphic to $S^2$ and has three orbifold points. Lemma \ref{lemma: CZdihedral} identifies the Reeb orbits of $\lambda_{\D^*, \varepsilon}=(1+\varepsilon \fp^*f_{\D})\lambda_{\D^*}$ that appear in the filtered chain complex and computes their Conley Zehnder indices. 

\begin{lemma}\label{lemma: CZdihedral}
Fix $N\in\N$. Then there exists an $\varepsilon_N>0$ such that for all $\varepsilon\in(0,\varepsilon_N]$, every $\gamma\in\mathcal{P}^{L_N}(\lambda_{\D^*,\varepsilon})$ is nondegenerate and projects to an orbifold critical point of $f_{\D}$ under $\fp$, where $L_N=2\pi N-\pi/2n$. If $c_i$ denotes the number of $\gamma\in \mathcal{P}^{L_N}(\lambda_{\D^*,\varepsilon})$ with $|\gamma|=i$, then 
\begin{itemize}
 \itemsep-.35em
    \item $c_i=0$ if $i<0$ or $i> 4N-2$;
    \item $c_i=n+2$ for $i=0$ and $i=4N-2$, with all $n+2$ contributions by good Reeb orbits;
    \item $c_i=n+3$ for even $i$, $0<i<4N-2$, with all $n+3$ contributions by good Reeb orbits;
    \item $c_i=1$ for odd $i$, $0<i<4N-2$, and this contribution is by a bad Reeb orbit.
\end{itemize}
\end{lemma}

\begin{proof}

Apply Lemma \ref{lemma: ActionThresholdLink} to $L_N=2\pi N-\frac{\pi}{2n}$ to obtain $\varepsilon_N$. Now, if $\varepsilon\in(0,\varepsilon_N]$, we have that every $\gamma\in\mathcal{P}^{L_N}(\lambda_{\P^*,\varepsilon})$ is nondegenerate and projects to an orbifold critical point of $f_{\D}$. We now study the actions and indices of these  orbits.

\vspace{.25cm}
    \textbf{Orbits over $p_-$}: Let $e_-$ denote the embedded Reeb orbit of $\lambda_{\D^*,\varepsilon}$ which projects to $p_-\in S^2/\D$. By Lemmas \ref{lemma: covering multiplicity} and \ref{lemma: orbits}, $e_-^4$ lifts to an embedded Reeb orbit of $\lambda_{\varepsilon}$ in $S^3$ with action $2\pi(1-\varepsilon)$, projecting to some $p_{-,j}$. Thus, $\mathcal{A}(e_- )=\pi(1-\varepsilon)/2$, so $\mathcal{A}(e^k_-)=k\pi (1-\varepsilon)/2$. Hence, our chain complex will be generated by only the $k=1,2,\dots, 4N-1$ iterates. Using Remark \ref{remark: same local models} and Proposition \ref{proposition: flow}, we see that the linearized Reeb flow of $\lambda_{\D^*,\varepsilon}$ along $e^k_-$ with respect to trivialization $\tau_{\D^*}$ is given by the family of matrices $M_t=\mathcal{R}(2t-t\varepsilon)$ for $t\in[0,k\pi(1-\varepsilon)/2]$, where we have used that $f(p_{-,j})=-1$ and that we have stereographic coordinates $\psi$ at the point $p_{-,j}$ such that $H(f,\psi)=\text{Id}$. We see that $e_-^k$ is elliptic with rotation number $\theta_-^k=k/2-k\varepsilon/4(1-\varepsilon)$, thus
    \[\mu_{\CZ}(e_-^k)=2\Bigl\lceil \frac{k}{2}-\frac{k\varepsilon}{4(1-\varepsilon)} \Bigr\rceil-1=2\Bigl\lceil \frac{k}{2} \Bigr\rceil-1,\] 
    where the last step is valid by reducing $\varepsilon_{N}$ if necessary.
    
    \vspace{.25cm}
    
    \textbf{Orbits over $p_h$}: Let $h$ denote the embedded Reeb orbit of $\lambda_{\D^*,\varepsilon}$ which projects to $p_h\in S^2/\D$. By Lemmas \ref{lemma: covering multiplicity} and \ref{lemma: orbits}, $h^4$ lifts to an embedded Reeb orbit of $\lambda_{\varepsilon}$ in $S^3$ with action $2\pi$, projecting to some $p_{h,j}$. Thus, $\mathcal{A}(h)=\pi/2$, so $\mathcal{A}(h^k)=k\pi/2$. Hence, our chain complex will be generated only the  $k=1,2,\dots, 4N-1$ iterates.
    
    To see that $h$ is a hyperbolic Reeb orbit, we  consider its 4-fold cover $h^4$. By again using Remark \ref{remark: same local models} and Proposition \ref{proposition: flow}, one may compute the linearized Reeb flow, noting that the lifted Reeb orbit projects to $p_{h,j}$ where $f(p_{h,j})=0$, and that we have stereographic coordinates $\psi$ at $p_{h,j}$ such that $H(f,\psi)=\text{Diag}(1,-1)$. We evaluate the matrix at $t=2\pi$ to see that the linearized return map associated to $h^4$ is 
    \[\exp\begin{pmatrix}0 & 2\pi\varepsilon \\ 2\pi\varepsilon&0\end{pmatrix}=\begin{pmatrix}\cosh{(2\pi\varepsilon)} &\sinh{(2\pi\varepsilon)}\\\sinh{(2\pi\varepsilon)}&\cosh{(2\pi\varepsilon)}\end{pmatrix}.\] 
    The eigenvalues of this matrix are $\cosh{(2\pi\varepsilon)}\pm\sinh{(2\pi\varepsilon)}$. So long as $\varepsilon$ is small, these eigenvalues are real and positive, so that $h^4$ is positive hyperbolic, implying that $h$ is also hyperbolic. If $\mu_{\CZ}(h)=I$, then by Corollary \ref{corollary: CZ sphere}, $4I=\mu_{\CZ}(h^4)=4$. Hence $I=1$ and $h$ is negative hyperbolic. 
    \vspace{.25cm}
    
    \textbf{Orbits over $p_+$}: Let $e_+$ denote the embedded Reeb orbit of $\lambda_{\D^*,\varepsilon}$ which projects to $p_+\in S^2/\D$. By Lemmas \ref{lemma: covering multiplicity} and \ref{lemma: orbits}, the  $2n$-fold cover $e_-^{2n}$ lifts to some embedded Reeb orbit of $\lambda_{\varepsilon}$ in $S^3$ with action $2\pi(1+\varepsilon)$, projecting to some $p_{+,j}$. Thus, $\mathcal{A}(e_+)=\pi(1+\varepsilon)/n$, so $\mathcal{A}(e^k_+)=k\pi(1+\varepsilon)/n$ and so our chain complex will be generated by only the $k=1,2,\dots, 2nN-1$ iterates. Using Remark \ref{remark: same local models} and Proposition \ref{proposition: flow}, we see that the linearized Reeb flow of $\lambda_{\D^*,\varepsilon}$ along $e^k_+$ with respect to trivialization $\tau_{\D^*}$ is given by the family of matrices \[M_t=\mathcal{R}\bigg(\frac{2t}{1+\varepsilon}+\frac{t\varepsilon}{(1+\varepsilon)^2}\bigg)\,\,\, \text{for}\,\,\, t\in[0,k\pi(1+\varepsilon)/n],\] where we have used that $f(p_{+,j})=1$ and that we have stereographic coordinates $\psi$ at  $p_{+,j}$ such that $H(f,\psi)=-\text{Id}$. Thus, $e_+^k$ is elliptic with  \[\theta_+^k=\frac{k}{n}+\frac{\varepsilon k}{2n(1+\varepsilon)}, \,\,\mu_{\CZ}(e_+^k)=1+2\Bigl\lfloor \frac{k}{n}+\frac{\varepsilon k}{2n(1+\varepsilon)} \Bigr\rfloor=1+2\Bigl\lfloor \frac{k}{n} \Bigr\rfloor,\]
    where the last step is valid for sufficiently small $\varepsilon_N$.
\end{proof}

Lemma \ref{lemma: CZdihedral} produces the sequence $(\varepsilon_N)_{N=1}^{\infty}$, which we can assume  decreases monotonically to 0 in $\R$. Define the sequence of 1-forms $(\lambda_N)_{N=1}^{\infty}$ on $S^3/\D^*$ by $\lambda_N:=\lambda_{\D^*,\varepsilon_N}$.

\begin{summary*}(Dihedral data). We have
\begin{equation} \label{equation: CZ indices dihedral}
    \mu_{\CZ}(e_-^k)=2\Bigl\lceil \frac{k}{2} \Bigr\rceil-1,\,\,\,\mu_{\CZ}(h^k)=k,\,\,\,\mu_{\CZ}(e_+^k)=2\Bigl\lfloor \frac{k}{n} \Bigr\rfloor+1,
\end{equation}
\begin{equation}\label{equation: action dihedral}
        \mathcal{A}(e_-^k)=\frac{k\pi(1-\varepsilon)}{2},\,\,\,\mathcal{A}(h^k)=\frac{k\pi}{2},\,\,\,\mathcal{A}(e_+^k)=\frac{k\pi(1+\varepsilon)}{n}.
\end{equation}
\end{summary*}

\begin{table}[h!]
\centering
 \begin{tabular}{||c | c | c | c ||} 
 \hline
Grading & Index & Orbits & $c_i$ \\ [0.5ex] 
 \hline\hline
 0 & 1 & $e_-,e_-^2,h,e_+,e_+^2,\dots,e_+^{n-1}$ & $n+2$\\ 
 \hline
 1 & 2 & $h^2$ & 1\\
 \hline
 2 & 3 & $e_-^3,e_-^4, h^3, e_+^{n}, \dots, e_+^{2n-1}$ & $n+3$\\
  \hline
 \vdots & \vdots & \vdots & \vdots \\
  \hline
 $4N-4$ & $4N-3$ & $e_-^{4N-3},e_-^{4N-2},h^{4N-3}, e_+^{(2N-2)n},\dots e_+^{(2N-1)n-1}$ & $n+3$\\
  \hline
 $4N-3$ & $4N-2$ & $h^{4N-2}$ & 1\\
    \hline
 $4N-2$ & $4N-1$ & $e_-^{4N-1},h^{4N-1}, e_+^{(2N-1)n},\dots,e_+^{2Nn-1}$ & $n+2$\\ 
 \hline
\end{tabular}
\caption{Reeb orbits of $\mathcal{P}^{L_N}(\lambda_{\D_{2n}^*,\varepsilon_N})$ and their Conley Zehnder indices}
\label{table: Dihedral CZ indices}
\end{table}

None of the orbits in the first two rows of Table \ref{table: Dihedral CZ indices} are contractible, thus $\lambda_N=\lambda_{\D^*,\varepsilon_N}$ is $L_N$-dynamically convex and by \cite[Thm.~1.3]{HN},\footnote{One hypothesis of \cite[Thm.~1.3]{HN} requires that all contractible Reeb orbits $\gamma$ satisfying $\mu_{\CZ}(\gamma)=3$ must be embedded. This fails in our case by considering the contractible $e_-^4$, which is not embedded yet satisfies $\mu_{\CZ}(e_-^4)=3$.  In the arXiv v2 of this paper, we proved in \S 4.3 why we do not need this additional hypothesis.} a generic choice $J_N\in\mathcal{J}(\lambda_N)$ provides a well-defined filtered chain complex, yielding the  isomorphism of $\Z$-graded vector spaces 
\begin{align*}
    CH_*^{L_N}(S^3/\D^*,\lambda_N, J_N)&\cong\bigoplus_{i=0}^{2N-1}\Q^{n+1}[2i]\oplus\bigoplus_{i=0}^{2N-2} H_*(S^2;\Q)[2i]\\
    &=\begin{cases} \Q^{n+2} & *=0, \, 4N-2 \\
\Q^{n+3} & *=2i, 0<i<2N-1 \\ 0 &\mbox{else.}\end{cases}.
\end{align*}
 This follows from investigating the good contributions to  $c_i$, which is 0 for odd $i$, implying $\partial^{L_N}=0$. This proves Theorem \ref{theorem: main} in the dihedral case because $|\text{Conj}(\D_{2n}^*)|=n+3$, after appealing to Theorem \ref{theorem: cobordisms induce inclusions}, which permits taking a direct limit over inclusions of these  groups.

\subsection{Binary polyhedral groups $\T^*$, $\Oc^*$, and $\I^*$} \label{subsection: polyhedral}
In contrast to the dihedral case, we opt not work with explicit matrix generators of the polyhedral groups, because the computations of the fixed points are too involved. Instead, we will take a more geometric approach. Let $\P^*\subset\SU(2)$ be some binary polyhedral group so that it is congruent to either $\T^*$, $\Oc^*$ or $\I^*$, with $|\P^*|=24$, 48, or 120, respectively. Let $\P\subset\SO(3)$ denote the image of $\P^*$ under the group homomorphism $P$. This group $\P$ is conjugate to one of $\T$, $\Oc$, or $\I$ in $\SO(3)$, and its order satisfies $|\P^*|=2|\P|$. It is known that the $\P$ action on $S^2$ is given by the symmetries of a regular polyhedron inscribed in $S^2$. The fixed point set $\text{Fix}(\P)$ is partitioned into three $\P$-orbits. Let the number of vertices, edges, and faces of the polyhedron in question be $\mathscr{V}$, $\mathscr{E}$, and $\mathscr{F}$ respectively (see Table \ref{table:polyhedral xyz}).

\vspace{.25cm}

\textbf{Vertex type fixed points}: The set $\{\mathfrak{v}_1, \mathfrak{v}_2, \dots, \mathfrak{v}_\mathscr{V}\}\subset \text{Fix}(\P)$ constitutes a single $\P$-orbit, where each $\mathfrak{v}_i$ is an inscribed vertex of the polyhedron in $S^2$. Let $\mathscr{I}_{\mathscr{V}}\in\N$ denote $|\P|/\mathscr{V}$, so that the  isotropy subgroup associated to any of the $\mathfrak{v}_i$ is cyclic of order $\mathscr{I}_{\mathscr{V}}$. Let $\mathfrak{v}\in S^2/\P$ denote the image of any of the $\mathfrak{v}_i$ under the orbifold covering map $\pi_{\P}:S^2\to S^2/\P$.

\vspace{.25cm}

\textbf{Edge type fixed points}: The set $\{\mathfrak{e}_1, \mathfrak{e}_2, \dots, \mathfrak{e}_\mathscr{E}\}\subset \text{Fix}(\P)$  constitutes a single $\P$-orbit, where each $\mathfrak{e}_i$ is the image of a midpoint of one of the edges of the polyhedron under the radial projection $\R^3\setminus\{0\}\to S^2$. Let $\mathscr{I}_{\mathscr{E}}\in\N$ denote $|\P|/\mathscr{E}$, so that the isotropy subgroup associated to any of the $\mathfrak{e}_i$ is cyclic of order $\mathscr{I}_{\mathscr{E}}$. One can see that $\mathscr{I}_{\mathscr{E}}=2$ for any choice of $\P$. Let $\mathfrak{e}\in S^2/\P$ denote the image of any of the $\mathfrak{e}_i$ under the orbifold covering map $\pi_{\P}:S^2\to S^2/\P$.

\vspace{.25cm}

\textbf{Face type fixed points}: The set $\{\mathfrak{f}_1, \mathfrak{f}_2, \dots, \mathfrak{f}_{\mathscr{F}}\}\subset \text{Fix}(\P)$ constitutes a single $\P$-orbit, where each $\mathfrak{f}_i$ is the image of a barycenter of one of the faces of the polyhedron under the radial projection $\R^3\setminus\{0\}\to S^2$. Let $\mathscr{I}_{\mathscr{F}}$ denote $|\P|/\mathscr{F}$, so that the isotropy subgroup associated to any of the $\mathfrak{f}_i$ is cyclic of order $\mathscr{I}_{\mathscr{F}}$. One can see that $\mathscr{I}_{\mathscr{F}}=3$ for any choice of $\P$. Let $\mathfrak{f}\in S^2/\P$ denote the image of any of the $\mathfrak{f}_i$ under the orbifold covering map $\pi_{\P}:S^2\to S^2/\P$.

\begin{table}[h!]
\centering
 \begin{tabular}{|| c | c | c | c | c | c | c | c | c ||} 
 \hline
 Group & Group order & $\mathscr{V}$ & $\mathscr{E}$ & $\mathscr{F}$ & $\mathscr{I}_{\mathscr{V}}$ & $\mathscr{I}_{\mathscr{E}}$ & $\mathscr{I}_{\mathscr{F}}$  & $|\text{Conj}(\P^*)|$\\ [0.5ex] 
 \hline\hline
 $\T$ & 12 & 4 & 6 & 4  & 3 & 2 & 3 & 7\\ 
 \hline
 $\Oc$ & 24 & 6 & 12 & 8 & 4 & 2 & 3 & 8\\
 \hline
 $\I$ & 60 & 12 & 30 & 20 & 5 & 2 & 3 & 9\\
 \hline
\end{tabular}
 \caption{Polyhedral quantities. Note $|\text{Cong}(\P^*)|=\mathscr{I}_{\mathscr{V}}+\mathscr{I}_{\mathscr{E}}+\mathscr{I}_{\mathscr{F}}-1.$}
 \label{table:polyhedral xyz}
\end{table}

\begin{remark}
(Dependence on choice of $\P^*$). The coordinates of the fixed point set of $\P$ are determined by the initial selection of $\P^*\subset\SU(2)$. More precisely, if $A^{-1}\P_2^*A=\P_1^*$  for $A\in\SU(2)$, then the rigid motion of $\R^3$ given by $P(A)\in\SO(3)$ takes the fixed point set of $\P_1$ to that of $\P_2$.
\end{remark}

There exists a $\P$-invariant, Morse-Smale function $f$ on $(S^2,\omega_{\FS}(\cdot,j \cdot))$, with $\text{Crit}(f)=\text{Fix}(\P)$, which descends to an orbifold Morse function $f_{\P}:S^2/\P\to\R$, constructed in Section \ref{appendix: constructing morse functions}. Furthermore, there are stereographic coordinates at
\begin{enumerate}[(i)]
    \itemsep-.35em
    \item the points $\mathfrak{v}_i$, in which $f$ takes the form $(x^2+y^2)/2-1$ near $(0,0)$;
    \item the points $\mathfrak{e}_i$, in which $f$ takes the form $(x^2-y^2)/2$ near $(0,0)$;
    \item the points $\mathfrak{f}_i$, in which $f$ takes the form $1-(x^2+y^2)/2$ near $(0,0)$.
\end{enumerate}
The orbifold surface $S^2/\P$ is homeomorphic to $S^2$ and has three orbifold points. Lemma \ref{lemma: CZpolyhedral} identifies the Reeb orbits of $\lambda_{\P^*, \varepsilon}=(1+\varepsilon \fp^*f_{\P})\lambda_{\P^*}$ that appear in the filtered chain complex and computes their Conley Zehnder indices. Let $m\in \N$ denote the integer $\mathscr{I}_{\mathscr{V}}+\mathscr{I}_{\mathscr{E}}+\mathscr{I}_{\mathscr{F}}-1$ (equivalently, $m=|\text{Conj}(\P^*)|$, see Table \ref{table:polyhedral xyz}).

\begin{lemma} \label{lemma: CZpolyhedral}
Fix $N\in\N$. Then there exists an $\varepsilon_N>0$ such that, for all $\varepsilon\in(0,\varepsilon_N]$, every $\gamma\in\mathcal{P}^{L_N}(\lambda_{\P^*,\varepsilon})$ is nondegenerate and projects to an orbifold critical point of $f_{\P}$ under $\fp$, where $L_N:=2\pi N-\pi/10$. If $c_i$ denotes the number of $\gamma\in \mathcal{P}^{L_N}(\lambda_{\P^*,\varepsilon})$ with $|\gamma|=i$, then

\begin{enumerate}
    \itemsep-.35em
    \item $c_i=0$ if $i<0$ or $i> 4N-2$,
    \item $c_i=m-1$ for $i=0$ and $i=4N-2$, with all contributions by good Reeb orbits;
    \item $c_i=m$ for even $i$, $1<i<4N-1$, with all contributions by good Reeb orbits;
     \item $c_i=1$ for odd $i$, $0<i<4N-2$, and this contribution is by a bad Reeb orbit.
\end{enumerate}
\end{lemma}

\begin{proof}
Apply Lemma \ref{lemma: ActionThresholdLink} to $L_N=2\pi N-\frac{\pi}{10}$ to obtain $\varepsilon_N$. If $\varepsilon\in(0,\varepsilon_N]$, then every $\gamma\in\mathcal{P}^{L_N}(\lambda_{\P^*,\varepsilon})$ is nondegenerate and projects to an orbifold critical point of $f_{\P}$. We  investigate the actions and Conley Zehnder indices of these three types of orbits. Our reasoning will largely follow that used in the proof of Lemma \ref{lemma: CZdihedral}, and so some details will be omitted.

\vspace{.25cm}

    \textbf{Orbits over $\mathfrak{v}$}: Let $\mathcal{V}$ denote the embedded Reeb orbit of $\lambda_{\P^*,\varepsilon}$ in $S^3/\P^*$ which projects to $\mathfrak{v}\in S^2/\P$. One computes that $\mathcal{A}(\mathcal{V}^k)=k\pi(1-\varepsilon)/\mathscr{I}_{\mathscr{V}}$, and so the iterates $\mathcal{V}^k$ are included for all $k<2N\mathscr{I}_{\mathscr{V}}$. The orbit $\mathcal{V}^k$ is elliptic with:
    \[\theta_{\mathcal{V}}^k=\frac{k}{\mathscr{I}_{\mathscr{V}}}-\frac{\varepsilon k}{2\mathscr{I}_{\mathscr{V}}(1-\varepsilon)}, \,\,\,\,\,\mu_{\CZ}(\mathcal{V}^k)=2\Bigl\lceil \frac{k}{\mathscr{I}_{\mathscr{V}}} \Bigr\rceil-1.\]
    
    \vspace{.25cm}

    \textbf{Orbits over $\mathfrak{e}$}: Let $\mathcal{E}$ denote the embedded Reeb orbit of $\lambda_{\P^*,\varepsilon}$ in $S^3/\P^*$ which projects to $\mathfrak{e}\in S^2/\P$. By a similar study of the orbit $h$ of Lemma \ref{lemma: CZdihedral}, one sees that $\mathcal{A}(\mathcal{E}^k)=k\pi/2$, so the iterates $\mathcal{E}^k$ are included for all $k<4N$. Like the dihedral Reeb orbit $h$, $\mathcal{E}$ is negative hyperbolic with $\mu_{\CZ}(\mathcal{E})=1$, thus $\mu_{\CZ}(\mathcal{E}^k)=k$. The even iterates of $\mathcal{E}$ are bad Reeb orbits.
    
    \vspace{.25cm}
    
    \textbf{Orbits over $\mathfrak{f}$}: Let $\mathcal{F}$ denote the embedded Reeb orbit of $\lambda_{\P^*,\varepsilon}$ in $S^3/\P^*$ which projects to $\mathfrak{f}\in S^2/\P$. One computes that $\mathcal{A}(\mathcal{F}^k)=k\pi(1+\varepsilon)/Z$, and so the iterates $\mathcal{F}^k$ are included for all $k<6N$. The orbit $\mathcal{F}^k$ is elliptic with:
    \[\theta_{\mathcal{F}}^k=\frac{k}{3}+\frac{\varepsilon k}{6(1+\varepsilon)}, \,\,\,\,\,\mu_{\CZ}(\mathcal{F}^k)=2\Bigl\lfloor \frac{k}{3} \Bigr\rfloor+1=2\Bigl\lfloor \frac{k}{3} \Bigr\rfloor+1.\]
\end{proof}

Lemma \ref{lemma: CZpolyhedral} produces the sequence $(\varepsilon_N)_{N=1}^{\infty}$. Define the sequence of 1-forms $(\lambda_N)_{N=1}^{\infty}$ on $S^3/\P^*$  by $\lambda_N:=\lambda_{\P^*,\varepsilon_N}$.

\begin{summary*}
(Polyhedral data). We have
\begin{equation} \label{equation: CZ indices polyhedral}
    \mu_{\CZ}(\mathcal{V}^k)=2\Bigl\lceil \frac{k}{\mathscr{I}_{\mathscr{V}}} \Bigr\rceil-1,\,\,\,\mu_{\CZ}(\mathcal{E}^k)=k,\,\,\,\mu_{\CZ}(\mathcal{F}^k)=2\Bigl\lfloor \frac{k}{3} \Bigr\rfloor+1,
\end{equation}

\begin{equation}\label{equation: action polyhedral}
        \mathcal{A}(\mathcal{V}^k)=\frac{k\pi(1-\varepsilon)}{\mathscr{I}_{\mathscr{V}}},\,\,\,\mathcal{A}(\mathcal{E}^k)=\frac{k\pi}{2},\,\,\,\mathcal{A}(\mathcal{F}^k)=\frac{k\pi(1+\varepsilon)}{3}.
\end{equation}
\end{summary*}

\begin{table}[h!]
\centering
 \begin{tabular}{||c | c | c | c ||} 
 \hline
Grading & Index & Orbits & $c_i$ \\ [0.5ex]
 \hline\hline
 0 & 1 & $\mathcal{V}, \dots, \mathcal{V}^{\mathscr{I}_{\mathscr{V}}}, \mathcal{E}, \mathcal{F}, \mathcal{F}^2$ & $m-1$\\ 
 \hline
 1 & 2 & $\mathcal{E}^2$ & 1\\
 \hline
 2 & 3 & $\mathcal{V}^{\mathscr{I}_{\mathscr{V}}+1}, \dots, \mathcal{V}^{2\mathscr{I}_{\mathscr{V}}}, \mathcal{E}^3, \mathcal{F}^3, \mathcal{F}^4, \mathcal{F}^5$ & $m$\\
  \hline
 \vdots & \vdots & \vdots & \vdots \\
  \hline
 $4N-4$ & $4N-3$ & $\mathcal{V}^{(2N-2)\mathscr{I}_{\mathscr{V}}+1}, \dots, \mathcal{V}^{(2N-1)\mathscr{I}_{\mathscr{V}}}, \mathcal{E}^{4N-3}, \mathcal{F}^{6N-6}, \mathcal{F}^{6N-5}, \mathcal{F}^{6N-4}$ & $m$\\
  \hline
 $4N-3$ & $4N-2$ & $\mathcal{E}^{4N-2}$ & 1\\
    \hline
 $4N-2$ & $4N-1$ & $\mathcal{V}^{(2N-1)\mathscr{I}_{\mathscr{V}}+1}, \dots, \mathcal{V}^{2N\mathscr{I}_{\mathscr{V}}-1}, \mathcal{E}^{4N-1}, \mathcal{F}^{6N-3}, \mathcal{F}^{6N-2}, \mathcal{F}^{6N-1}$ & $m-1$\\ 
 \hline
\end{tabular}
\caption{Reeb orbits of $\mathcal{P}^{L_N}(\lambda_{\P^*,\varepsilon_N})$ and their Conley Zehnder indices}
\label{table: polyhedral CZ indices}
\end{table}

None of the orbits in the first two rows of Table \ref{table: polyhedral CZ indices} are contractible, so $\lambda_N=\lambda_{\P^*,\varepsilon_N}$ is $L_N$-dynamically convex and so by \cite[Theorem 1.3]{HN},\footnote{One hypothesis of \cite[Th. 1.3]{HN} requires that all contractible Reeb orbits $\gamma$ satisfying $\mu_{\CZ}(\gamma)=3$ must be embedded. This fails in our case by considering the contractible $\mathcal{V}^{2\mathscr{I}_{\mathscr{V}}}$, which is not embedded yet satisfies $\mu_{\CZ}(\mathcal{V}^{2\mathscr{I}_{\mathscr{V}}})=3$.  In the arXiv v2 of this paper, we proved in \S 4.3 why we do not need this additional hypothesis.} a generic choice $J_N\in\mathcal{J}(\lambda_N)$ provides a well defined filtered chain complex, yielding the isomorphism of $\Z$-graded vector spaces

\begin{align*}
    CH_*^{L_N}(S^3/\P^*,\lambda_N, J_N)&\cong\bigoplus_{i=0}^{2N-1}\Q^{m-2}[2i]\oplus\bigoplus_{i=0}^{2N-2} H_*(S^2;\Q)[2i]\\
    &=\begin{cases} \Q^{m-1} & *=0, \, 4N-2 \\
\Q^{m} & *=2i, 0<i<2N-1 \\ 0 &\mbox{else.}\end{cases}
\end{align*}
 This follows from investigating the good contributions to  $c_i$, which is 0 for odd $i$, implying $\partial^{L_N}=0$. This proves Theorem \ref{theorem: main} in the polyhedral case because $|\text{Conj}(\P^*)|=m$, after appealing to Theorem \ref{theorem: cobordisms induce inclusions}, which permits taking a direct limit over inclusions of these  groups.

\section{Direct limits of filtered  cylindrical contact homology}\label{section: direct limits of filtered homology}
In the previous section we computed the action filtered cylindrical contact homology of the links of the simple singularities. For any finite, nontrivial subgroup $G\subset\SU(2)$, we have a sequence $(L_N,\lambda_N,J_N)_{N=1}^{\infty}$, where $L_N\to\infty$ monotonically in $\R$, $\lambda_N$ is an $L_N$-dynamically convex contact form on $S^3/G$ with kernel $\xi_G$, and $J_N\in\mathcal{J}(\lambda_N)$ is generically chosen so that 
\begin{equation}\label{equation: identification}
    CH_*^{L_N}(S^3/G,\lambda_N, J_N)\cong\bigoplus_{i=0}^{2N-1}\Q^{m-2}[2i]\oplus\bigoplus_{i=0}^{2N-2} H_*(S^2;\Q)[2i],
\end{equation}
where $m=|\text{Cong}(G)|$. For $N\leq M$, there is a natural inclusion of the vector spaces on the right hand side of \eqref{equation: identification}. 

This section establishes Theorem \ref{theorem: cobordisms induce inclusions}, yielding that exact symplectic cobordisms obtained from decreasing the Morse-Bott perturbation induce well-defined maps on filtered homology, which can be identified with these inclusions, which  completes the proof of Theorem \ref{theorem: main} as 
\[
\varinjlim_N CH_*^{L_N}(S^3/G,\lambda_N, J_N)\cong\bigoplus_{i\geq0}\Q^{m-2}[2i]\oplus\bigoplus_{i\geq0} H_*(S^2;\Q)[2i].
\]
We first explain the cobordisms and the maps they induce.

\begin{definition}
An \emph{exact symplectic cobordism} is from $(Y_+,\lambda_+)$ to $(Y_-,\lambda_-)$ is a pair $(\overline{X},\lambda)$ where $\overline{X}$ is a compact symplectic manifold with boundary $\partial \overline{X}= Y_+ - Y_-$ and $d\lambda$ is a symplectic form on $\overline{X}$ with $\lambda|_{Y_\pm} = \lambda_\pm$.  Given an exact symplectic cobordism $(\overline{X},\lambda)$, we form its \emph{completion}
\[
{X} = ((-\infty,0] \times Y_-) \sqcup_{Y_-} \overline{X}\sqcup_{Y_+} ( [0,\infty) \times Y_+)
\]
using the gluing under the following identifications.  A neighborhood of $Y_+$ in $(\overline{X},\lambda)$ can be canonically identified with $(-\epsilon, 0]_s \times Y_+$ for some $\epsilon>0$ so that $\lambda$ is identified with $e^s\lambda_+$.  Moreover, this identification is defined so that $\partial_s$ corresponds to the unique vector field $V$ such that $\iota_Vd\lambda = \lambda$.  Similarly, a neighborhood of $Y_-$ can be canonically identified with $[0,\epsilon) \times Y_-$ so that $\lambda$ is identified with $e^s\lambda_-$.

 An almost complex structure $J$ on the completion ${X}$ is said to be \emph{cobordism compatible} if $J$ is $d\lambda$-compatible on $X$ (meaning that $d\lambda(\cdot, J \cdot)$ is a Riemannian metric on $X$), and there are $\lambda_\pm$-compatible almost complex structures $J_\pm$ on $\R \times Y_\pm$ such that $J$ agrees with $J_+$ on $[0,\infty) \times Y_+$ and with $J_-$ on $(-\infty,0] \times Y_-$.
\end{definition}


We will make use of some further notation. 

\begin{notation}\label{notation: equivalence of reeb orbits}
For $\gamma_+\in\mathcal{P}^{L_N}(\lambda_N)$ and $\gamma_-\in\mathcal{P}^{L_M}(\lambda_M)$, we write $\gamma_+\sim\gamma_-$ whenever $m(\gamma_+)=m(\gamma_-)$ and both orbits project to the same orbifold point  under $\fp$. If either condition does not hold, we write $\gamma_+\nsim\gamma_-$.  Note that $\sim$ defines an equivalence relation on the disjoint union $\bigsqcup_{N\in\N}\mathcal{P}^{L_N}(\lambda_N)$. If $\gamma$ is contractible in $S^3/G$, then we write $[\gamma]=0$ in $[S^1,S^3/G]$.
\end{notation}

In the following, we consider cobordism maps $\Phi^M_N$ induced by decreasing the Morse-Bott perturbation.   We more precisely define the exact completed cobordism and space of cobordism compatible almost complex structures that we will use below.

\begin{definition}\label{def:Jnm}
We want to ensure that there are `obvious' index zero cylinders, which we henceforth refer to as \emph{cobordism trivial cylinders} connecting the nondegenerate Reeb orbits $\gamma_+\in\mathcal{P}^{L_N}(\lambda_N)$ and $\gamma_-\in\mathcal{P}^{L_M}(\lambda_M)$ satisfying $\gamma_+\sim\gamma_-$.
 To do so, we will restrict to the following contractible space of $J$ satisfying the following conditions.

Since $\ker \lambda_N = \ker \lambda_{M}$, we know $\lambda_N = e^{g_N} \lambda_{M}$ for some $g_N \in \C^\infty(Y, \R)$.  We may assume $g_N > 0$ everywhere.  Let $g\in \C^\infty(\R \times Y,\R)$ such that
\begin{itemize}
\itemsep-.35em
\item $g(s,y) = s$ for $s \in (-\infty, \varsigma)$ for some $\varsigma > 0$;
\item $g(s,y) = g_N(y) + s -1$ for $s \in (1- \varsigma, \infty)$ for some $\varsigma > 0$;
\item $\partial_s g >0$.
\end{itemize}
Since $d(e^{g(s,\cdot)}\lambda_{M})$ is symplectic, $([0,1] \times S^3/G, e^{g(s,\cdot)}\lambda_{M})$ is an exact symplectic cobordism from $(S^3/G,\lambda_{N})$ to $(S^3/G,\lambda_{M})$. 

Let $J$ be a cobordism compatible almost complex structure on $\R \times S^3/G$, which agrees with $J_N$, a generic $\lambda_N$-compatible almost complex structure on $[1,\infty)\times S^3/G$, and with $J_M$, a generic $\lambda_M$-compatible almost complex structure on $(-\infty,0]\times S^3/G$.  
We can choose $J$ so that the cobordism trivial cylinder connecting $\gamma_+$ to $\gamma_-$, where $\gamma_+ \sim \gamma_-$  is a $J$-holomorphic curve because $[1,\infty) \times {\gamma_+}$ is a $J_N$-holomorphic submanifold, $(-\infty,0] \times \gamma_-$ is a $J_M$-holomorphic submanifold, {and the compact portion connecting them, defined by the union of the Reeb orbits of $e^{g(s,\cdot)}\lambda_{M}$ in the same equivalence class, is a symplectic submanifold of $([0,1] \times S^3/G, e^{g(s,\cdot)}\lambda_{M})$}.  Note that the space of $J$ satisfying these conditions is contractible.
\end{definition}



Next we consider the space of Fredholm index zero $J$-holomorphic cylinders in $X$, denoted by $\mathcal{M}_0^J(\gamma_+,\gamma_-)$, where $J$ is subject to the conditions in Definition \ref{def:Jnm}, $\gamma_{+}\in\mathcal{P}_{\text{good}}(\lambda_N)$, and $\gamma_{-}\in\mathcal{P}_{\text{good}}(\lambda_{M})$.  We will show that $\mathcal{M}_0^J(\gamma_+,\gamma_-)$ is a compact 0-manifold, which is nonempty only when $\gamma_+ \sim \gamma_-$ and that $\# \mathcal{M}_0^J(\gamma_+,\gamma_-) =1$. From this, the map $\Phi_N^M$, defined by
\[
\Phi_N^M:CC_*(S^3/G,\lambda_N,J_N)\to CC_*(S^3/G,\lambda_M,J_M),\,\,\,\,\,\,\,\langle\Phi_N^M(\gamma_+),\gamma_-\rangle:=\sum_{u\in\mathcal{M}_0^J(\gamma_+,\gamma_-)}\epsilon(u)\frac{m(\gamma_+)}{m(u)},
\] 
is well defined. That $\Phi_N^M$ is a chain map follows from a careful analysis of the moduli spaces of index 1 cylinders that appear in these completed cobordisms. These chain maps induce \emph{continuation homomorphisms} on the cylindrical contact homology groups. Our main result in Section \ref{section: direct limits of filtered homology} is the following:

\begin{theorem}\label{theorem: cobordisms induce inclusions}
An exact completed symplectic cobordism $(X,\lambda, J)$ from $(S^3/G,\lambda_N,J_N)$ to $(S^3/G,\lambda_M,J_M)$ as in Definition \ref{def:Jnm}, for $N\leq M$, induces a well defined chain map between filtered chain complexes. The induced maps on homology $\Psi_N^M:CH_*^{L_N}(S^3/G,\lambda_N,J_N)\to CH_*^{L_M}(S^3/G,\lambda_M,J_M)$ may be identified with the standard inclusions \[\bigoplus_{i=0}^{2N-1}\Q^{m-2}[2i]\oplus\bigoplus_{i=0}^{2N-2} H_*(S^2;\Q)[2i]\hookrightarrow \bigoplus_{i=0}^{2M-1}\Q^{m-2}[2i]\oplus\bigoplus_{i=0}^{2M-2} H_*(S^2;\Q)[2i]\]
and form a directed system of graded $\Q$-vector spaces over $\N$.
\end{theorem}

There are two main steps in the proof of Theorem  \ref{theorem: cobordisms induce inclusions}.  The first is to establish compactness of the 0-dimensional moduli spaces. This is shown in Section \ref{subsection: holomorphic buildings in cobordisms} via the study of free homotopy classes of Reeb orbits in Section \ref{subsection: free homotopy classes represented by reeb orbits}.  The second is to establish that the so-called cobordism trivial cylinders are unique, namely that $\# \mathcal{M}_0^J(\gamma_+,\gamma_-) =1$. This is carried out in Section \ref{ss:unique} and relies on intersection theoretic arguments of Hutchings \cite{index,revisit} and Siefring \cite{s2}, as elucidated by Wendl \cite{wint}, in conjunction with establishing that the theoretical bounds on the winding of asymptotic eigenfunctions are achieved via a strengthening of results due to Hofer, Wysocki, and Zehnder \cite{hwz2}.  

The proof of Theorem \ref{theorem: cobordisms induce inclusions} is as follows. Automatic transversality will be used in Corollary \ref{corollary: moduli spaces are finitie} to prove that $\mathcal{M}_0^J(\gamma_+,\gamma_-)$ is a 0-dimensional manifold. By  Proposition \ref{proposition: buildings in cobordisms},  $\mathcal{M}_0^J(\gamma_+,\gamma_-)$ will be shown to be compact. From this, one concludes that $\mathcal{M}_0^J(\gamma_+,\gamma_-)$ is a finite set for $\gamma_+\in\mathcal{P}_{\text{good}}^{L_N}(\lambda_N)$ and $\gamma_-\in\mathcal{P}_{\text{good}}^{L_M}(\lambda_M)$.
Thus for $N\leq M$, we have the well-defined chain map: 
\[
\Phi_N^M:\big(CC_*^{L_N}(S^3/G,\lambda_N,J_N),\partial^{L_N}\big)\to\big(CC_*^{L_N}(S^3/G,\lambda_M,J_M),\partial^{L_N}\big).
\] 
The uniqueness result of Section \ref{ss:unique} allows us to conclude that the finite set of  cylinders in $X$,  equals that of  $\mathcal{M}^{J_N}(\gamma_+,\gamma_+)$ and that of $\mathcal{M}^{J_M}(\gamma_-,\gamma_-)$, the moduli spaces of cylinders in the symplectizations of $\lambda_N$ and $\lambda_M$, which is known to be 1, given by the  contribution of a single trivial cylinder. If $\gamma_+\nsim\gamma_-$, then Corollary \ref{corollary: moduli spaces are finitie} will imply that $\mathcal{M}_0^J(\gamma_+,\gamma_-)$ is empty, and so $\langle \Phi_N^M(\gamma_+),\gamma_-\rangle=0$. Ultimately we conclude that, given $\gamma_+\in\mathcal{P}_{\text{good}}^{L_N}(\lambda_N)$, our chain map $\Phi_N^M$ takes the form $\Phi_N^M(\gamma_+)=\gamma_-$, where $\gamma_-\in \mathcal{P}_{\text{good}}^{L_N}(\lambda_M)$ is the unique Reeb orbit satisfying $\gamma_+\sim\gamma_-$.

We let $\iota_N^M$ denote the chain map given by the inclusion of subcomplexes,
\[\iota_N^M:\big(CC_*^{L_N}(S^3/G,\lambda_M,J_M),\partial^{L_N}\big)\hookrightarrow\big(CC_*^{L_M}(S^3/G,\lambda_M,J_M),\partial^{L_M}\big).\]
The composition $\iota_N^M\circ\Phi_N^M$ a chain map. Let $\Psi_N^M$ denote the map on homology induced by this composition, that is, $\Psi_N^M=(\iota_N^M\circ\Phi_N^M)_*$. We see that \[\Psi_N^M:CH_*^{L_N}(S^3/G,\lambda_N,J_N)\to CH_*^{L_M}(S^3/G,\lambda_M,J_M)\]
satisfies $\Psi_N^M([\gamma_+])=[\gamma_-]$ whenever $\gamma_+\sim\gamma_-$, and thus $\Psi_N^M$ takes the form of the standard inclusions after making the identifications \eqref{equation: identification}.

\subsection{Compactness}\label{subsection: holomorphic buildings in cobordisms}
A \emph{holomorphic building} $B$ is a tuple $(u_1,u_2,\dots,u_n)$, where each $u_i$ is a potentially disconnected holomorphic curve in a completed (exact) symplectic cobordism 
equipped with a cobordism compatible almost complex structure. The curve $u_i$ is the $i^{\text{th}}$ \emph{level of} $B$. Each $u_i$ has a set of positive (resp. negative) ends which are positively (resp. negatively) asymptotic to a set of  Reeb orbits. For  $i\in\{1,\dots, n-1\}$, there is a bijection between the negative ends of $u_i$ and the positive ends of $u_{i+1}$, such that paired ends are asymptotic to the same Reeb orbit. The \emph{height} of $B$ is $n$.  The \emph{positive ends} of $B$ are given by the positive ends of $u_1$ and the \emph{negative ends} of $B$ are given by the negative ends of $u_n$.  The \emph{genus} of $B$ is the genus of the Riemann surface $S$ obtained by attaching ends of the domains of the $u_i$ to those of $u_{i+1}$ according to the bijections; $B$ is \emph{connected} if $S$ is connected. The \emph{index} of $B$ is $\text{ind}(B):=\sum_{i=1}^n\text{ind}(u_i)$.

All buildings that we consider will have at most one level in a nontrivial exact symplectic cobordism, with the rest in a symplectization.  Unless otherwise stated, we require that no level of $B$ be solely given in terms of a union of trivial cylinders in a symplectization, e.g. $B$ is {without trivial levels}.

\begin{remark}\label{remark: index of buildings}
The index of a connected  genus 0 building $B$ with one positive end at $\alpha$, and with $k$ negative ends at $\beta_1,...\beta_k$,  is given by: \begin{equation}\label{equation: index of buildings}\text{ind}(B)=k-1+\mu_{\CZ}(\alpha)-\sum_{i=1}^k\mu_{\CZ}(\beta_i).\end{equation}
This fact follows from an inductive argument applied to the height of $B$.  \end{remark}

The following proposition considers the relationships between the Conley Zehnder indices and actions of a pair of Reeb orbits representing the same free homotopy class in  $[S^1,S^3/G]$. 

\begin{proposition} \label{proposition: CZ and action}
Suppose $[\gamma_+]=[\gamma_-]\in[S^1,S^3/G]$ for $\gamma_+\in\mathcal{P}^{L_N}(\lambda_N)$ and $\gamma_-\in\mathcal{P}^{L_M}(\lambda_M)$.
\begin{enumerate}[\em (a)]
\itemsep-.35em
    \item If $\mu_{\CZ}(\gamma_+)=\mu_{\CZ}(\gamma_-)$ then $\gamma_+\sim\gamma_-$.
    \item If $\mu_{\CZ}(\gamma_+)<\mu_{\CZ}(\gamma_-)$ then $\mathcal{A}(\gamma_+)<\mathcal{A}(\gamma_-)$.
\end{enumerate}
\end{proposition}

The proof of Proposition \ref{proposition: CZ and action} is postponed to Section \ref{subsection: free homotopy classes represented by reeb orbits}, where we will show that it holds for cyclic, dihedral, and polyhedral groups $G$ (Lemmas \ref{lemma: cyclic cobordisms} and \ref{lemma: dihedral cobordisms}, and Proposition \ref{lemma: polyhedral cobordisms}). For any exact symplectic cobordism $(X,\lambda,J)$, Proposition \ref{proposition: CZ and action} (a) implies that $\mathcal{M}_0^J(\gamma_+,\gamma_-)$ is empty whenever $\gamma_+\nsim\gamma_-$, and (b) crucially implies that there do not exist cylinders of negative Fredholm index in $X$. Using Proposition \ref{proposition: CZ and action}, we now prove a compactness argument:

\begin{proposition}\label{proposition: buildings in cobordisms}
Fix $N< M$, $\gamma_+\in\mathcal{P}^{L_N}(\lambda_N)$, and $\gamma_-\in\mathcal{P}^{L_M}(\lambda_M)$. For $n_{\pm}\in\Z_{\geq0}$, consider a connected, genus zero building $B=(u_{n_+},\dots,  u_0, \dots, u_{-n_-})$, where $u_i$ is in the symplectization of $\lambda_N$ for $i>0$, $u_i$ is in the symplectization of $\lambda_M$ for $i<0$, and $u_0$ is in a generic, completed,  exact symplectic cobordism $(X,\lambda, J)$ from $(\lambda_N,J_N)$ to $(\lambda_M,J_M)$. If  $\mbox{\em ind}(B)=0$, with single positive puncture at $\gamma_+$ and single negative puncture at $\gamma_-$, then $n_+=n_-=0$ and $u_0\in\mathcal{M}_0^J(\gamma_+,\gamma_-)$.
\end{proposition}
\begin{proof}
This building $B$ provides the following sub-buildings, some of which may be empty:

\medskip

\begin{minipage}[c]{0.40\textwidth}
\includegraphics[width=\textwidth]{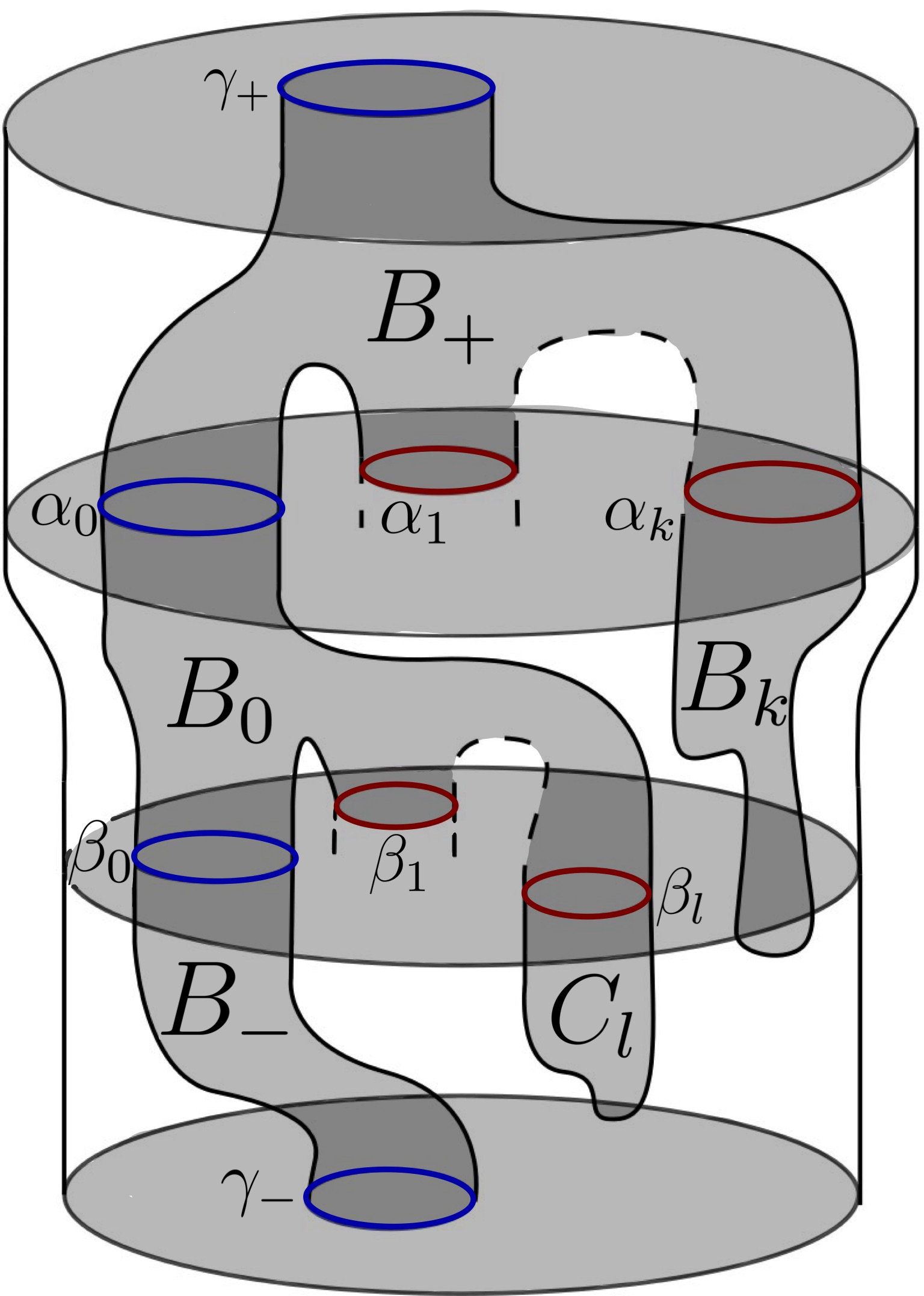}
\end{minipage}
\begin{minipage}[c]{0.5\textwidth}
\begin{itemize}[font=\small]
\itemsep-.35em
    \item $B_+$, a  building in the symplectization of $\lambda_N$, with no level consisting entirely of trivial cylinders, with one positive puncture at $\gamma_+$, $k+1$ negative punctures at  $\alpha_i\in\mathcal{P}^{L_N}(\lambda_N)$, for $i=0,\dots, k$, with $[\alpha_i]=0$ for $i>0$, and $[\gamma_+]=[\alpha_0]$.
    \item $B_0$, a height 1 building in the cobordism with one positive puncture at $\alpha_0$, $l+1$ negative punctures at $\beta_i\in\mathcal{P}^{L_N}(\lambda_M)$, for $i=0,\dots,l$, with $[\beta_i]=0$ for $i>0$, and $[\alpha_0]=[\beta_0]$.
    \item $B_i$, a  building with one positive end at $\alpha_i$, without negative ends, for $i=1,\dots, k$.
    \item $C_i$, with one positive end at $\beta_i$, without negative ends, for $i=1,2,\dots,l$.
    \item $B_-$, a  building in the symplectization of $\lambda_M$, with one positive puncture at $\beta_0$ and one negative puncture at $\gamma_-$.
\end{itemize}
\end{minipage}

\medskip

Contractible Reeb orbits are shown in red, while Reeb orbits representing the free homotopy class $[\gamma_{\pm}]$ are shown in blue. Each sub-building is connected and has genus zero. Although no level of $B$ consists entirely of trivial cylinders, some of these sub-buildings may have entirely trivial levels in a symplectization. Note that $0=\text{ind}(B)$ equals the sum of indices of the above buildings. Write $0=\text{ind}(B)=U+V+W$, where
\[U:=\text{ind}(B_+)+\sum_{i=1}^k\text{ind}(B_i),\,\,\,\,V:=\text{ind}(B_0)+\sum_{i=1}^l\text{ind}(C_i),\,\,\text{and}\,\, W:=\text{ind}(B_-).\]
We will first argue that $U$, $V$, and $W\geq0$. 

To see $U\geq0$, apply the index formula \eqref{equation: index of buildings} to each summand to compute $U=\mu_{\CZ}(\gamma_+)-\mu_{\CZ}(\alpha_0)$. If $U<0$, then  Proposition \ref{proposition: CZ and action} (b) implies $\mathcal{A}(\gamma_+)<\mathcal{A}(\alpha_0)$, which would violate the fact that action decreases along holomorphic buildings. We must have that $U\geq0$. 

To see $V\geq0$, again apply the index formula \eqref{equation: index of buildings} to find $V=\mu_{\CZ}(\alpha_0)-\mu_{\CZ}(\beta_0)$. Suppose $V<0$. Now Proposition \ref{proposition: CZ and action} (b) implies $\mathcal{A}(\alpha_0)<\mathcal{A}(\beta_0)$, contradicting the decrease of action. 

To see $W\geq 0$, consider that either $B_-$ consists entirely of trivial cylinders, or it doesn't. In the former case $W=0$. In the latter case, \cite[Prop.~2.8]{HN} implies that $W>0$. 

Because 0 is written as the sum of three non-negative integers, we conclude that $U=V=W=0$. We will combine this fact with Proposition \ref{proposition: CZ and action} to conclude that $B_{\pm}$ are empty buildings and that $B_0$ is a cylinder, concluding the proof. 

Note that $U=0$ implies $\mu_{\CZ}(\gamma_+)=\mu_{\CZ}(\alpha_0)$. Because $[\gamma_+]=[\alpha_0]$, Proposition \ref{proposition: CZ and action} (a) implies that $\gamma_+\sim\alpha_0$. This is enough to conclude $\gamma_+=\alpha_0\in\mathcal{P}(\lambda_N)$, and importantly, $\mathcal{A}(\gamma_+)=\mathcal{A}(\alpha_0)$. Noting that $\mathcal{A}(\gamma_+)\geq\sum_{i=0}^k\mathcal{A}(\alpha_i)$ (again, by decrease of action), we must have that $k=0$ and that this inequality is an equality. Thus, the buildings $B_i$ are empty for $i\neq0$, and the building $B_+$ has index 0 with only one negative end, $\alpha_0$. If $B_+$ has some nontrivial levels then \cite[Prop.~2.8]{HN} implies $0=\text{ind}(B_+)>0$. Thus, $B_+$ consists only of trivial levels. Because $B_+$ has no trivial levels, it is empty, and $n_+=0$.

Similarly, $V=0$ implies $\mu_{\CZ}(\alpha_0)=\mu_{\CZ}(\beta_0)$. Again, because $[\alpha_0]=[\beta_0]$, Proposition \ref{proposition: CZ and action} (a) implies that $\alpha_0\sim\beta_0$. Although we cannot write $\alpha_0=\beta_0$, we \emph{can} conclude that the difference $\mathcal{A}(\alpha_0)-\mathcal{A}(\beta_0)$ may be made arbitrarily small, by rescaling  $\{\varepsilon_K\}_{K=1}^{\infty}$ by some $c\in(0,1)$. Thus, the inequality $\mathcal{A}(\alpha_0)\geq\sum_{i=0}^l\mathcal{A}(\beta_i)$ forces $l=0$, implying that each $C_i$ is empty, and that $B_0$ has a single negative puncture at $\beta_0$, i.e. $B_0=(u_0)$ for $u_0\in\mathcal{M}_0^J(\alpha_0,\beta_0)$.

Finally, we consider our index 0 building $B_-$ in the symplectization of $\lambda_M$, with one positive end at $\beta_0$ and one negative end at $\gamma_-$. Again, \cite[Prop.~2.8]{HN} tells us that if $B_-$ has nontrivial levels, then $\text{ind}(B_-)>0$. Thus, all levels of $B_-$ must be trivial (this implies $\beta_0=\gamma_-$). However, because the $B_i$ and $C_j$ are empty, a trivial level of $B_-$ is a trivial level of $B$ itself, contradicting our hypothesis on $B$. Thus, $B_-$ is empty and $n_-=0$.
\end{proof}

\begin{corollary}\label{corollary: moduli spaces are finitie}
The moduli space $\mathcal{M}_0^J(\gamma_+,\gamma_-)$ is a compact, 0-dimensional manifold for generic $J$, where $\gamma_+\in\mathcal{P}^{L_N}(\lambda_N)$ and $\gamma_-\in\mathcal{P}^{L_M}(\lambda_M)$ are good, and $N<M$. Furthermore, $\mathcal{M}_0^J(\gamma_+,\gamma_-)$ is empty if $\gamma_+\nsim\gamma_-$.
\end{corollary}

\begin{proof}
To prove regularity of the cylinders $u\in\mathcal{M}_0^J(\gamma_+,\gamma_-)$, we invoke automatic transversality (\cite[Thm.~1]{W}), providing regularity of all such $u$, because the inequality \[0=\text{ind}(u)>2b-2+2g(u)+h_+(u)=-2\]
holds.  Here, $b=0$ is the number of branched points of $u$ over its underlying somewhere injective cylinder, $g(u)=0$ is the genus of the curve, and $h_+(u)$ is the number of ends of $u$ asymptotic to positive hyperbolic Reeb orbits. Because good Reeb orbits in $\mathcal{P}^{L_N}(\lambda_N)$ and $\mathcal{P}^{L_M}(\lambda_M)$ are either elliptic or negative hyperbolic, we have that $h_+(u)=0$. To prove compactness of the moduli spaces, note that a sequence in $\mathcal{M}_0^J(\gamma_+,\gamma_-)$ has a subsequence converging in the sense of \cite{BEHWZ} to a building with the properties detailed in Proposition \ref{proposition: buildings in cobordisms}. Proposition \ref{proposition: buildings in cobordisms} proves that such an object is single cylinder in $\mathcal{M}_0^J(\gamma_+,\gamma_-)$, proving compactness of $\mathcal{M}_0^J(\gamma_+,\gamma_-)$. Finally, by Proposition \ref{proposition: CZ and action} (a), the existence of $u\in\mathcal{M}_0^J(\gamma_+,\gamma_-)$ implies that $\gamma_+\sim\gamma_-$. Thus, $\gamma_+\nsim\gamma_-$ implies that $\mathcal{M}_0^J(\gamma_+,\gamma_-)$ is empty.
\end{proof}

\subsection{Uniqueness of cobordism trivial cylinders}\label{ss:unique}

The purpose of this section is to prove uniqueness of cobordism trivial cylinders.   The main technical obstruction to carrying out a proof of this is that in the $N = \infty$ Morse-Bott limit, these cylinders are no longer isolated but come in families, so one cannot naively invoke \cite[Thm. 10.32]{wendl-sft}.  However, since we are in dimension four, we can appeal to some intersection theory due to Hutchings and Siefring, as summarized by Wendl to establish the desired result, namely that $\#\mathcal{M}^J_0 (\gamma_+ , \gamma_- ) = 1$.  Our arguments rely on the special situation at hand, because the theoretical bounds on the winding of asymptotic eigenfunctions controlling the approach of the cylinders at infinity is achieved, by way of a strengthening of results originating in the work of Hofer, Wysocki, and Zehnder.  We first give a brief, mostly contained summary of the results utilized, and then combine them give the proof of our main result below.

\begin{proposition}\label{prop:unique-triv}
 Let $(X,\lambda, J)$ be an exact completed symplectic cobordism from $(S^3/G,\lambda_N,J_N)$ to $(S^3/G,\lambda_M,J_M)$ as in Definition \ref{def:Jnm}, for $N\leq M$. Then $\#\mathcal{M}^J_0 (\gamma_+ , \gamma_- ) = 1$.
 \end{proposition}

First, we recall some results about the asymptotics of holomorphic curves and the winding of the associated eigenfunctions. Next we give a mostly self contained review of intersection theory.  Finally, we put everything together to give a proof of uniqueness of cobordism trivial cylinders.

\subsubsection{Winding of asymptotic eigenfunctions}\label{sss:wind}

Let $\gamma$ be a nondegenerate embedded Reeb orbit. W define the {\em asymptotic operator\/} $L_\gamma: C^\infty(\gamma^*\xi) \longrightarrow C^\infty(\gamma^*\xi)$ from the space of smooth sections of $\xi|_{\gamma^d}$ to itself by
\[
L_\gamma=-J_{\gamma(t)}\nabla_t^R,
\]
where $\nabla^R$ denotes the connection on $\xi|_{\gamma^d}$ defined by the linearized Reeb flow along $\gamma$. The operator $L$ is symmetric and so its eigenvalues are real. Nondegeneracy of the Reeb orbit $\gamma$ implies that $\gamma^*\xi$ does not have any nonzero section which is parallel with respect to $\nabla^R$, and so zero is not an eigenvalue of $L$.

We first recall some facts pertaining to the spectral properties of the asymptotic operator and the winding numbers of its eigenfunctions, going back to \cite[\S 3]{hwz2}.  

\begin{lemma}
\label{lem:eigenfunctions}
\begin{enumerate}[\em(a)]
	\item Every (nonzero) eigenfunction $\varphi$ of $L_\gamma$ is a nonvanishing section of $\gamma^*\xi$, and hence has a well defined winding number (with respect to a chosen unitary trivialization of $\gamma^*\xi$), {which we denote by $\wind(\varphi)$.  Any two nontrivial eigenfunctions in the same eigenspace have the same winding number. }
	
\item Let $\varphi_1$ and $\varphi_2$ be linearly independent eigenfunctions of $L$ with the same winding number. Then any nontrivial linear combination of $\varphi_1$ and $\varphi_2$ is a nonvanishing section of $\gamma^*\xi$.

\item {If $\varphi_1$ and $\varphi_2$ are eigenfunctions of $L$ with eigenvalues $\lambda_1 \leq \lambda_2$ then $\wind(\varphi_1)  {\leq} \wind(\varphi_2)$. } 

\item {For every $w \in \Z$, $L_\gamma$ has exactly two eigenvalues (counting multiplicity) for which the corresponding eigenfunctions have winding number equal to $w$.}

\end{enumerate}
\end{lemma}

\begin{definition}
Given a nondegenerate closed Reeb orbit $\gamma$ and a trivialization $\tau: \gamma^*\xi \to S^1 \times \R^2$, we define
\[
\wind^\tau : \sigma(L_\gamma) \to \Z
\]
by $\wind^\tau(\lambda):=\wind^\tau(\varphi)$ where $\varphi:S^1\to\R^2$ is the expression via $\tau$ of any nontrivial  eigenfunction $\varphi \in \Gamma(\gamma^*\xi)$ with $L_\gamma \varphi = \lambda \varphi$.  By way of Lemma \ref{lem:eigenfunctions}, we can sensibly define the \emph{positive / negative extremal winding numbers} to be:
\begin{align*}
    \alpha_{+}^{\tau}(\gamma)&:=\min\{\wind_{\tau}(\eta) \ | \ \eta\in \sigma(L_\gamma)\cap (0,\infty)\},\\
    \alpha_{-}^{\tau}(\gamma)&:=\max\{\wind_{\tau}(\eta) \ | \ \eta\in \sigma(L_\gamma)\cap (-\infty, 0)\}.
\end{align*}
The parity of a Reeb orbit is given by:
\[
p(\gamma):= \alpha_{+}^{\tau}(\gamma) - \alpha_{-}^{\tau}(\gamma),
\]
which is equal to 0 when $\gamma$ is positive hyperbolic or an even iterate of a negative hyperbolic orbit and 1 when $\gamma$ is elliptic or an odd iterate of a negative hyperbolic orbit, cf. \cite[\S 3]{wendl-sft}.
\end{definition}

From \cite[\S 3]{hwz2}, we also have the following definition for the Conley-Zehnder index in terms of the winding invariants.  For any nondegenerate Reeb orbit $\gamma$ and any trivialization $\tau$ of $\gamma^*\xi$ we have
\begin{equation}\label{e:windcz}
\mu_{\CZ}^\tau(\gamma) = 2 \alpha_{-}^{\tau}(\gamma) + p(\gamma) =  2\alpha_{+}^{\tau}(\gamma) - p(\gamma)
\end{equation}

Next we review how the asymptotic behavior of the end of an asymptotically cylindrical $J$-holomorphic curve at a Reeb orbit is determined by an asymptotic eigenfunction of the corresponding asymptotic operator.  
First, to warm up and set up some notation, we recall what it means for a $J$-holomorphic half cylinders to be asymptotically cylindrical to a Reeb orbit $\gamma$.  Denote the positive and negative half-cylinders by
\[
Z_+:= [0,\infty) \times S^1, \ \ \ Z_-:=(-\infty,0] \times S^1.
\]
A $J$-holomorphic half-cylinder $u: Z_\pm \to \R \times M$ is said to be (positively / negatively) \emph{asymptotic} to a nondegenerate $T$-periodic Reeb orbit $\gamma$ if, after a positive reparametrization near infinity,
\begin{equation}\label{halfas}
u(s,t) = \exp_{(Ts,\gamma(t))}\eta(s,t) \ \ \ \mbox{ for $|s|$ large},
\end{equation}
where the exponential map is assumed to be translation invariant and $\eta(s,t)$ is a vector field along the trivial cylinder\footnote{The trivial cylinder is the $J$-holomorphic ``orbit cylinder" $\R \times \gamma$ where $J$ is $\R$-invariant, which is expressed as $u: \R \times S^1 \to \R \times Y: (s,t) \mapsto (Ts, \gamma(Tt))$, where $T$ is the period of the orbit $\gamma$.} with $\eta(s,\cdot) \to 0$ in $C^\infty(S^1)$ as $s \to \pm \infty$.  Thus as $|s|$ approaches infinity, $u(s,t)$ becomes $C^\infty$ close to the trivial cylinder, which is an immersion with normal bundle equivalent to $\gamma^*\xi.$

After a further reparametrization of $Z_\pm$, we can arrange for \eqref{halfas} to hold for a unique section $\eta_u(s,t) \in \xi_{\gamma(t)}$; the section $\eta_u(s,t)$ is called the \emph{asymptotic representative} of $u$.  The uniqueness of $\eta$ depends on the choice of parametrization of the Reeb orbit $\gamma: S^1 \to Y$; different choices change $\eta_u$ by a shift in the $t$-coordinate.

\begin{proposition}{\em \cite{hwz1}, \cite[Prop.~2.4]{ht2}}\label{prop:orbitas}
Suppose $u: Z_\pm \to \R \times Y$ is a $J$-holomorphic half cylinder which is positively/negatively  asymptotic to the nondegenerate Reeb orbit $\gamma : \R/\Z \to M$ and let $\eta(s,t) \in \xi_{\gamma(t)}$ denote its asymptotic representative.  Then if $\eta$ is not identically zero, there exists a unique nontrivial eigenfunction $\varphi$ of $L_\gamma$ with 
\[
L_\gamma \varphi = \lambda \varphi, \ \ \  \pm \lambda <0 
\]
and a section $r(s,t) \in \xi_{\gamma(t)}$ satisfying $r(s, \cdot) \to 0$ uniformly as $s \to \infty$ such that for sufficiently large $|s|$,
\begin{equation}\label{eq:usefulin4}
\eta(s,t) = e^{\lambda s}\left(\varphi + r(s,t) \right). 
\end{equation}

\end{proposition}

In the context of Proposition \ref{prop:orbitas}, we say that $u$ approaches the Reeb orbit $\gamma$ along the asymptotic eigenfunction $\varphi_\lambda$ with decay rate $|\lambda|$.  This proposition establishes the asymptotic approach between the half cylinder given by the end of $u$ and the trivial cylinder $\R \times \gamma$.  A similar result holds for any two curves approaching the same orbit, and from this one can establish a lower bound on the resulting ``relative" decay rate and a generalization of this asymptotic formula  to ``higher orders" was established in \cite[\S 2]{s1}.

\begin{remark}\label{rem:finiteint}
An important consequence of Proposition \ref{prop:orbitas}, and its relative decay rate improvement for any two curves asymptotic to the same Reeb orbit, is that any two asymptotically cylindrical $J$-holomorphic having nonidentical images must have at most finitely many intersections, cf. \cite[Cor.~3.13]{wint}
  \end{remark}

In particular, the asymptotic representative encodes the isotopy class of the braid $\zeta$, which is given in terms of the intersection of the end of the $J$-holomorphic curve asymptotic to $\gamma^d$ and a tubular neighborhood $N$ of $\gamma$ for $|s|\gg0$.  Let $N \simeq (\R/\Z) \times D^2 $ be a tubular neighborhood of $\gamma$ which has been identified with a disk bundle in the normal bundle to $\gamma$ as well as in $\gamma^*\xi$ so that $\gamma$ corresponds to $(\R/\Z) \times \{ 0\}$ and the derivative of the identification along $\gamma$ sends $\gamma^*\xi$ to $\{ 0\} \oplus \C$ in agreement with the unitary trivialization $\tau$.  Let  $u$ be an asymptotically cylindrical $J$-holomorphic curve with a positive (or negative) end at $\gamma^d$, which is not part of a multiply covered component.  Then results of Siefring \cite[Cor.~2.5, 2.6]{s1} allow us to conclude that if $|s|$ is sufficiently large then the intersection of this end of $u$ with $\{ s\} \times N \subset \{ s \} \times Y$ is a braid $\zeta$, whose isotopy class is independent of $s$.

We conclude by providing the following bounds on the winding number in terms of the Conley-Zehnder index.  We deduce that the ends of the $J$-holomorphic curves in $\M_0^J(\gamma_+,\gamma_-)$ meet the extremal bounds, by way of a mild generalization of arguments previously appearing in \cite[\S 3]{hwz2}, \cite[Lem. 3.2]{HN}, \cite[\S 3]{ht2}.


\begin{lemma}{\em \cite[\S 3]{hwz2}, \cite[Lem. 3.2]{HN}}\label{lem:dc-wind}
Let $(X,\lambda, J)$ be an exact completed symplectic cobordism from $(S^3/G,\lambda_N,J_N)$ to $(S^3/G,\lambda_M,J_M)$ as in Definition \ref{def:Jnm}, for $N\leq M$.  Let $\gamma$ be an embedded Reeb orbit, let $u$ be a $J$-holomorphic curve in $X$ with a positive end at $\gamma^d$ and a single negative end.  Let $\zeta$ denote the intersection of an end with $\{ s \} \times Y$.  If $s$ is sufficiently large, then 
\begin{enumerate}[\em (a)]
\item $\zeta$ is the graph in $N$ of a nonvanishing section of $\xi|_{\gamma^d}$ and thus has a well-defined winding number $\wind_\tau(\zeta)$.
\item $\wind_\tau(\zeta) \leq \lfloor \mu_{\CZ}^\tau(\gamma^d)/2 \rfloor. $ 
\item If $\mu_{CZ}^\tau(\gamma^d)$ is odd and $\ind(u) =0 $, then equality holds in {(b)}.
\end{enumerate}
\end{lemma}

\begin{proof}
Choose an identification $N\simeq (\R/\Z)\times D^2$ compatible with the trivialization $\tau$. The asymptotic behavior of  pseudoholomorphic curves as described in conjunction with Lemma \ref{lem:eigenfunctions} above implies (a).

Hofer, Wysocki, and Zehnder demonstrated in \cite[\S3]{hwz2} that for each integer $n$, there are exactly two eigenvalues of $L_\gamma$ whose eigenfunctions have winding number $n$. (Here and below we count eigenvalues with multiplicity.) Moreover, larger winding numbers correspond to larger eigenvalues, and the largest possible winding number for a negative eigenvalue is $\floor{\mu_{\CZ}^\tau(\gamma^d)/2}$, thus (b) follows.

To establish (c), note that the same argument in \cite[\S3]{hwz2} also shows that the smallest possible winding number of an eigenfunction of $L_\gamma$ with positive  eigenvalue is $\ceil{\mu_{\CZ}^\tau(\gamma^d)/2}$. Since $\mu_{\CZ}^\tau(\gamma^d)$ is assumed odd, we have a strict inequality $\floor{\mu_{\CZ}^\tau(\gamma^d)/2} < \ceil{\mu_{\CZ}^\tau(\gamma^d)/2}$. Consequently, the two (possibly equal) eigenvalues of $L_\gamma$ whose eigenfunctions have winding number $\floor{\mu_{\CZ}^\tau(\gamma^d)/2}$ are both negative. Thus, if equality does not hold in (b), then the associated eigenvalue $\lambda$  is not one of the two largest negative eigenvalues of $L_\gamma$ (counted with multiplicity as usual). 

Wendl observed that one can use exponentially weighted Sobolev spaces to set up the moduli space of irreducible holomorphic curves in which the eigenvalue $\lambda$ is not one of the two largest negative eigenvalues, see  \cite[Rmk.\ 3.3]{ht2}, but beware of their sign convention for the asymptotic operator.

Automatic transversality holds under the assumptions in (c), e.g. \cite[Thm.~1]{W}, because 
\[
0 = \ind(u) > -2 +h_+(u),
\]
since $h_+(u)=0$, the number of ends of $u$ asymptotic to positive hyperbolic Reeb orbits.
Thus any curves in this moduli space are cut out transversely, but the dimension of the moduli space is $2$ less than usual. Consequently there are no $J$-holomorphic curves $u$ in this moduli space with $\text{ind}(u)= 0$.

\end{proof}

Symmetrically to Lemma~\ref{lem:dc-wind}, we also have the following:

\begin{lemma}
\label{lem:neg-dc}
Let $\gamma$ be an embedded Reeb orbit, let $u$ be a $J$-holomorphic curve in $X$ with a negative end at $\gamma^d$ which is not part of a trivial cylinder in $\R \times Y$ with $J$ chosen to be $\R$-invariant, and a single  positive end let $\zeta$ denote the intersection of this end with $\{s\}\times Y$. If $s<<0$, then the following hold:
\begin{enumerate}[\em (a)]
\item $\zeta$ is the graph of a nonvanishing section of $\xi|_{\gamma^d}$, and thus has a well-defined winding number $\wind_\tau(\zeta)$.
\item $\wind_\tau(\zeta) \ge \ceil{\mu_{CZ}^\tau(\gamma^d)/2}$.
\item If $\mu_{CZ}^\tau(\gamma^d)$ is odd and $\ind(u)= 0$, then equality holds in (b).
\end{enumerate}
\end{lemma}

\subsubsection{Recap of intersection theory}\label{sss:int}

We now briefly review some intersection theory due to Hutchings \cite{index,revisit} and Siefring \cite{s1,s2}, which has been nicely summarized by Wendl \cite{wint}. The Siefring intersection pairing gives a relation between homotopy invariant quantities that have geometric meanings independent of any choice of trivializations and surface representatives of the holomorphic curves.  This improves upon Hutchings' work, which defined a relative intersection pairing and provided a relative adjunction formula for somewhere injective asymptotically cylindrical curves.

\begin{theorem}[{\cite{s2}, \cite[Thm. 4.1]{wint}}]\label{thm:*}
For any two asymptotically cylindrical maps in a 4-dimensional completed symplectic cobordism with $u$ and $u'$ with nondegenerate asymptotic Reeb orbits, there exists the so-called Siefring intersection pairing 
\[
u*v \in \mathbb{Z}
\]
which satisfies the following properties.
\begin{enumerate}[\em(a)]
\item The pairing $u*v$ depends only on the homotopy classes of $u$ and $v$ as asymptotically cylindrical  maps.
\item If $u$ and $v$ are asymptotically cylindrical $J$-holomorphic curves that have nonidentical images then
\begin{equation}
u * v := u \cdot v + \iota_\infty(u,v),
\end{equation}
where $u \cdot v$ is the usual algebraic count of intersections which is nonnegative by intersection positivity and $ \iota_\infty(u,v)$ is a nonnegative integer realizing ``hidden intersections at infinity."
\end{enumerate}
\end{theorem}

\begin{remark}
There is the following \emph{Folk Theorem}, which we do not make use of, but helps elucidate the Siefring intersection pairing, whose statement is as follows. In Theorem \ref{thm:*}(b), there exists a perturbation $J_\epsilon$ which is $C^\infty$ close to $J$ and a pair of $J_\epsilon$-holomorphic asymptotically cylindrical curves $u_\epsilon$ and $v_\epsilon$ having the same domains as $u$ and $v$ and close to $u$ and $v$ in their respective modui spaces, such that
\[
u_\epsilon \cdot v_\epsilon = u*v.
\]
Thus we should interpret $u*v$ as the count of intersections between generic curves homotopic to $u$ and $v$.  
\end{remark}

An important corollary of Theorem \ref{thm:*} that we will make use of is as follows.

\begin{corollary}\label{cor:0disjoint}
If $u$ and $v$ are $J$-holomorphic curves satisfying $u* v = 0$ then any two $J$-holomorphic curves that have nonidentical images and are homotopic to $u$ and $v$ respectively are disjoint.
\end{corollary}

The final foundational result that we need concerns nonnegativity of asymptotic intersections.

\begin{theorem}[{\cite[Thm. 4.13]{wint}}]\label{thm:iota0}
If $u$ and $v$ are asymptotically cylindrical $J$-holomorphic maps with nonidentical images, then $\iota_\infty(u,v) \geq 0$ with equality if and only if for all pairs of ends of $u$ and $v$ respectively asymptotic to covers of the same Reeb orbit, all of the resulting relative asymptotic eigenfunctions have extremal winding.
\end{theorem}

To actually compute $u*v$, we need a few more notions, including the relative intersection number, which is due to Hutchings \cite{index, revisit}.   The \emph{relative intersection number} is denoted in the work of Siefring and Wendl by $ u \bullet_\tau v$.   The relative intersection pairing previously appeared in Hutchings' work as one of the a key ingredients to the definition of embedded contact homology, where it is denoted by $Q_\tau( [u],[v])$ and was established to depend only on the homotopy class of the trivialization $\tau$ and relative homology class of the curves.  These intersection numbers coincide:
\[
 u \bullet_\tau v = Q_\tau( [u],[v]).
 \]
    (We go into further details in regards to how to more precisely define and compute the relative intersection number for $J$-holomorphic cylinders \`a la Hutchings in the proof of Lemma \ref{lem:negint}.)

In the interim, a colloquial description is that the relative intersection number is an algebraic count of intersections between $u$ and a generic perturbation of $v$ so that they have at most finitely many intersections and the perturbation of $v$ is shifted by an arbitrarily small positive distance in directions determined by the chosen trivialization $\tau$ at infinity.  (One could symmetrically perturb $u$ and leave $v$ alone.)  

If $u$ and $v$ have finitely many intersections then the algebraic intersection number $u \cdot v$ is well-defined.  In this situation, 
\begin{equation}\label{e:rel-int}
 u \bullet_\tau v = u \cdot v + \iota_\infty^\tau(u,v),
\end{equation}
where $\iota_\infty^\tau(u,v)$ is a count of additional intersections near infinity that appear when $v$ is perturbed, because the perturbation of $v$ can be assumed to be nontrivial only in some neighborhood of infinity where $u$ and $v$ are disjoint.  (Again, the content of \eqref{e:rel-int} previously appeared in the work of Hutchings with different notation, which we elucidate in the proof of Lemma \ref{lem:negint}, including the definition of $\iota_\infty^\tau(u,v)$.)

Modulo having deferred more precise definitions of  $u \bullet_\tau v$ and $\iota_\infty^\tau(u,v)$ to later, we are nearly ready to define the trivialization independent asymptotic intersection count, $\iota_\infty(u,v).$  The last ingredient we need are the terms $\Omega_\pm^\tau(\gamma^k,\gamma^m)$, which are the theoretical bounds on the possible asymptotic winding of ends of $u$ around the ends of $v$, as discussed in Section \ref{sss:wind}, and defined below. 

\begin{definition} 
Assume $u$ and $v$ are asymptotically cylindrical $J$-holomorphic curves, so that $u(s,t)$ and $v(s,t)$ approach their respective covers of $\gamma$ (both at either a negative end or both at a positive end) along asymptotic eigenfunctions $\varphi_u$ and $\varphi_v$ with decay rates $|\lambda_u|$ and $|\lambda_v|$, respectively.  We define the theoretical extremal bounds at a positive / negative puncture by
\[
\Omega_\pm^\tau(\gamma^k,\gamma^m) := \operatorname{min} \left\{ \mp k \alpha_\mp^\tau(\gamma^m),\mp m\alpha_\mp^\tau(\gamma^k) \right\}
\] 
We extend the definition of $\Omega_\pm^\tau$ by setting
\[
\Omega_\pm^\tau(\gamma^k_1,\gamma^m_2):=0 \quad \mbox{whenever } \gamma_1 \neq \gamma_2.
\]
\end{definition}
As explained in \cite[\S 4.2]{wint}, these theoretical extremal bounds provides a universal lower bound on $\iota_\infty^\tau(u,v)$, namely
\[
\iota_\infty^\tau(u,v) \geq \sum_{(z,z')\in\Gamma_u^\pm\times\Gamma_v^\pm}\Omega_\pm^\tau(\gamma^{k_z}_z,\gamma^{k_{z'}}_{z'}),
\]
where $\Gamma_u^\pm$ denote the positive \ negative punctures of $u$ and $\gamma^{k_z}_z$ denotes the associated Reeb orbit to which the puncture $z$ asymptotes, and likewise for $v$.  

We can now define the Siefring intersection pairing in terms of these more concrete and computable terms as follows.

\begin{definition}\label{def:intstuff}
For asymptotically cylindrical maps $u$ and $v$ with finitely many intersections, we define
\[
\iota_\infty(u,v) := \iota_\infty^\tau(u,v) - \sum_{(z,z')\in\Gamma_u^\pm\times\Gamma_v^\pm}\Omega_\pm^\tau(\gamma^{k_z}_z,\gamma^{k_{z'}}_{z'}).
\]
and similarly for any  asymptotically cylindrical maps $u$ and $v$ (with)not necessarily with finitely many intersections), we can define
\[
u*v :=  u \bullet_\tau v  - \sum_{(z,z')\in\Gamma_u^\pm\times\Gamma_v^\pm}\Omega_\pm^\tau(\gamma^{k_z}_z,\gamma^{k_{z'}}_{z'})
\]
\end{definition}
Crucially, the results in Section \ref{sss:wind} will allow us to conclude that in the case of interest, the winding numbers are extremal in the sense that they achieve the corresponding theoretical bound.

We have one last computational ingredient that we will rely on, cf. \cite[Ex. 4.19(b)]{wint}, detailing some of the aforementioned notions along the way as we proceed with the proof.
\begin{lemma}\label{lem:negint}
Let $(X,\lambda, J)$ be an exact completed symplectic cobordism from $(S^3/G,\lambda_N,J_N)$ to $(S^3/G,\lambda_M,J_M)$ as in Definition \ref{def:Jnm}, for $N\leq M$. Let $u_{[\gamma]}$ be the cobordism trivial cylinder between orbits in the same equivalence class, which are further assumed to be either both elliptic or both odd iterates of a negative hyperbolic Reeb orbit.  Then the Siefring self-intersection pairing is given by the multiplicity of the Reeb orbit equivalence class:
\begin{equation}
u_{[\gamma]}*u_{[\gamma]} = -m([\gamma]).
\end{equation}
\end{lemma}

\begin{proof}
By Definition \ref{def:intstuff}
\begin{equation}\label{e:*triv}
u_{[\gamma]}*u_{[\gamma]} = u_{[\gamma]}\bullet_\tau u_{[\gamma]} +  { \Omega_+^\tau - \Omega_-^\tau}.
\end{equation}
Here  $ u_{[\gamma]}\bullet_\tau u_{[\gamma]}$ is the relative intersection number, which appears in the ECH literature as $Q_\tau( u_{[\gamma]}, u_{[\gamma]}),$ which we detail further and compute to be 0.
We compute $\Omega_\pm^\tau$ momentarily in \eqref{e:omega}.

Using the definition of the relative intersection pairing by Hutchings \cite[\S 2.7]{revisit}, it follows that the relative intersection number of any cobordism trivial cylinder is 0, e.g. that
\[
{Q_\tau(u_{[\gamma]}) = u_{[\gamma]}\bullet_\tau u_{[\gamma]}=0}
\]
  To elucidate this point, recall that the \emph{relative intersection number} $Q_\tau(u)$  of any $J$-holomorphic curve $u$ is a well-defined algebraic count of interior intersections of two compact oriented surfaces whose interiors are transverse and do not intersect at the boundary $S, S' \in [-1,1] \times Y$ representing $[u] \in H_2(Y,\beta_+,\beta_-)$, where $[\beta_+]=[\beta_-] \in H_1(Y)$, whose definition in the case when $u$ is a cylinder we more precisely explain as follows.  
  
  If $S$ is one of these so called admissible representatives of a $J$-holomorphic cylinder $u$ (with one positive end at $\beta_+^{m_+}$ and one negative end at $\beta_-^{n_-}$), then for $\epsilon$ sufficiently small, $S \cap (\{1-\epsilon\} \times Y)$ consists of a braid $\zeta_+$ with $m_+$ strands in a tubular neighborhood of $\beta_+$, which is well-defined up to isotopy.  Likewise $S \cap (\{-1+\epsilon\} \times Y)$ consists of a braid $\zeta_-$ with $m_-$ strands in a tubular neighborhood of $\beta_-$, which is well-defined up to isotopy. (Here $\beta_\pm$ are embedded Reeb orbits.)

The linking number  of two disjoint braids $\zeta_1$ and $\zeta_2$ around an embedded Reeb orbit $\beta$ is denoted by $\ell_\tau(\zeta_1,\zeta_2) \in \mathbb{Z}$ and is defined to be the linking number of the oriented links $\phi_\tau(\zeta_1)$ and  $\phi_\tau(\zeta_2)$ in $\R^3$, where $\phi_\tau$ is an embedding of the tubular neighborhood $N$ of $\beta$ which has been identified via the trivialization $\tau$ with $S^1\times D^2$, so that the projection of $\zeta_i$ to to $S^1$ is a submersion, which has further been identified with a solid torus in $\R^3$, cf. \cite[\S 2.6]{revisit} for more details.  The linking number is by definition, one half of the signed count of crossings of a strand of $\phi_\tau(\zeta_1)$ with a strand of $\phi_\tau(\zeta_2)$ in the projection to $\R^2\times\{0\}$.  The sign convention is that counterclockwise twists count positively, which differs from knot theory literature \cite[\S 3.1]{index}.

In the case of the aforementioned cylinder $u$, we define the linking number
\[
\ell_\tau(S,S'):=\ell_\tau(\zeta_+,\zeta_+') - \ell_\tau(\zeta_-,\zeta_-');
\]
for more general $J$-holomorphic curves, one obtains a collection of disjoint braids, whose linking numbers one sums over.  The linking number $\ell_\tau$ as defined above is Siefring's relative asymptotic intersection number $\iota_\infty^\tau$ with an opposite sign convention for the two punctures.

The relative intersection pairing is defined by
 \[
Q_\tau([u],[u]):=\#(\dot{S} \cap \dot{S}') - \ell_\tau(S,S'),
 \]
and one further denotes $Q_\tau([u]):=Q_\tau([u],[u]).$ With the proceeding understood, we can now see by the definition detailed above that
\[
{Q_\tau(u_{[\gamma]}) = u_{[\gamma]}\bullet_\tau u_{[\gamma]}=0}
\]
as claimed.

The remaining terms in \eqref{e:*triv}, $\Omega_\pm$ are the theoretical bounds on the possible relative asymptotic winding of ends of $u$ around itself at the appropriate puncture.  By definition, for $u_{[\gamma]}$ we have that
\begin{equation}\label{e:omega}
\Omega_\pm^\tau(\gamma_{z_\pm}^m,\gamma_{z_\pm}^m):=\mp m \alpha_\mp^\tau(\gamma^m).
\end{equation}
Here $\Omega_+^\tau$ corresponds to the positive asymptotic orbit of $u_{[\gamma]}$ at the puncture $z_+$ and $\Omega_-^\tau$ corresponds to the negative asymptotic orbit of $u_{[\gamma]}$ at the puncture $z_-$.

In the below, let 
\[ 
m = m([\gamma]).
\]
Thus \eqref{e:*triv} can be computed to be
\[
\begin{split}
u_{[\gamma]}*u_{[\gamma]} & =  u_{[\gamma]}\bullet_\tau u_{[\gamma]} -  { \Omega_+^\tau + \Omega_-^\tau}\\
& = 0 + m \ \alpha_-^\tau([\gamma]) - m \ \alpha_+^\tau([\gamma]) \\
& = m.
\end{split}
\]
Here the last line follows from a basic manipulation of the Conley-Zehnder index formula \eqref{e:windcz} in terms of the positive and negative extremal winding numbers and parity of the orbit, yielding
\[
0 = 2\alpha_-(\gamma) - 2\alpha_+(\gamma) + 2p(\gamma).
\]
Here $p(\gamma)$ is the parity of the Reeb orbit $\gamma$, which is by definition is 1 whenever $\gamma$ is elliptic or an odd multiple of a negative hyperbolic orbit.
\end{proof}

\subsubsection{Proof of Proposition \ref{prop:unique-triv}}

We now combine Sections \ref{sss:wind} and \ref{sss:int} to establish the desired uniqueness of cobordism trivial cylinders as follows. 


\begin{proof}[Proof of Proposition \ref{prop:unique-triv}]
  In the below, denote $T$ to be the cobordism trivial cylinder between $\gamma_+$ and $\gamma_-$, where $\gamma_+ \sim \gamma_-$, $\gamma_+\in\mathcal{P}^{L_N}_{\text{good}}(\lambda_N)$, and $\gamma_-\in\mathcal{P}^{L_M}_{\text{good}}(\lambda_M)$.  Let $v$ be a curve with the same asymptotics, which is not a cobordism trivial cylinder.  In particular, $T$ and $v$ have nonidentical images.

Recall that by Theorem \ref{thm:iota0} {the count of ``hidden intersections at infinity"} satisfies
\[
\iota_\infty(T,v) \geq 0
\]
with equality if and only if for all pairs of ends of $T$ and $v$ respectively asymptotic to covers of the same Reeb orbit, all of the resulting relative asymptotic eigenfunctions have extremal winding.  We established that the resulting relative asymptotic eigenfunctions have extremal winding in \S \ref{sss:wind}, cf. Lemmas  \ref{lem:dc-wind} and \ref{lem:neg-dc}.

Thus the Siefring intersection `*' number coincides with the algebraic intersection `$\cdot$' number:
\[
T*v = T \cdot v,
\]
which must be non-negative by the positivity of intersections (cf. \cite[App. B]{wint}), with equality if and only if $T$ and $v$ are disjoint by Corollary \ref{cor:0disjoint}.  However, the Siefring intersection number $T*v$ depends only on the homotopy classes of $T$ and $v$ as asymptotically cylindrical maps by Theorem \ref{thm:*}(a). We computed that all asymptotically cylindrical maps with the same asymptotics and relative homology class have negative Siefring intersection number in Lemma \ref{lem:negint}, a contradiction.  Thus $T$ and $v$ must have identical images, hence $\#\mathcal{M}^J_0 (\gamma_+ , \gamma_- ) = 1$, as desired.
\end{proof}

\subsection{Homotopy classes of Reeb orbits and proof of Proposition \ref{proposition: CZ and action}}\label{subsection: free homotopy classes represented by reeb orbits}

Proposition \ref{proposition: CZ and action} was key in arguing  compactness of $\mathcal{M}_0^J(\gamma_+,\gamma_-)$. To prove Proposition \ref{proposition: CZ and action}, we will make use of a bijection $[S^1,S^3/G]\cong\text{Conj}(G)$ to identify the free  homotopy classes represented by orbits in $S^3/G$ in terms of $G$. A loop in $S^3/G$ is a map $\gamma:[0,T]\to S^3/G$, satisfying  $\gamma(0)=\gamma(T)$. Selecting a lift $\widetilde{\gamma}:[0,T]\to S^3$ of $\gamma$ to $S^3$ determines a unique $g\in G$, for which $g\cdot\widetilde{\gamma}(0)=\widetilde{\gamma}(T)$. The map $[\gamma]\in[S^1,S^3/G]\mapsto[g]\in\text{Conj}(G)$ is well defined,  bijective, and respects iterations; that is, if $[\gamma]\cong[g]$, then  $[\gamma^m]\cong[g^m]$ for $m\in\N$.

\subsubsection{Cyclic subgroups}
We may assume that $G=\langle g\rangle\cong\Z_n$, where $g$ is the diagonal matrix $g=\text{Diag}(\epsilon, \overline{\epsilon})$, and $\epsilon:=\text{exp}{(2\pi i/n)}\in\C$. We have  $\text{Conj}(G)=\{[g^m]:0\leq m<n\}$, each class is a singleton because $G$ is abelian. We select explicit lifts $\widetilde{\gamma}_{\mathfrak{s}}$ and $\widetilde{\gamma}_{\mathfrak{n}}$ of $\gamma_{\mathfrak{s}}$ and $\gamma_{\mathfrak{n}}$ to $S^3$ given by:
\[\widetilde{\gamma}_{\mathfrak{n}}(t)=(e^{it},0),\,\,\, \text{and}\,\,\,\widetilde{\gamma}_{\mathfrak{s}}(t)=(0,e^{it}),\,\, \text{for}\,\, t\in[0,2\pi/n].\]
Because $g\cdot\widetilde{\gamma}_{\mathfrak{n}}(0)=\widetilde{\gamma}_{\mathfrak{n}}(2\pi/n)$,  $[\gamma_{\mathfrak{n}}]\cong[g]$,  and thus $[\gamma_{\mathfrak{n}}^m]\cong[g^m]$ for $m\in\N$. Similarly, because $g^{n-1}\cdot\widetilde{\gamma}_{\mathfrak{s}}(0)=\widetilde{\gamma}_{\mathfrak{s}}(2\pi/n)$, $[\gamma_{\mathfrak{s}}]\cong[g^{n-1}]=[g^{-1}]$, and thus $[\gamma_{\mathfrak{s}}^m]\cong[g^{-m}]$ for  $m\in\N$.

\begin{center}
    \begin{tabular}{||c|c||}
    \hline
    Class & Represented orbits\\
    \hline\hline
        $[g^m]$, for $0\leq m<n$ & $\gamma_{\mathfrak{n}}^{m+kn},\,\,\gamma_{\mathfrak{s}}^{n-m+kn}$ \\
        \hline
    \end{tabular}
\end{center}

\begin{lemma}\label{lemma: cyclic cobordisms}
Suppose $[\gamma_+]=[\gamma_-]\in[S^1,S^3/\Z_n]$ for $\gamma_+\in\mathcal{P}^{L_N}(\lambda_N)$ and $\gamma_-\in\mathcal{P}^{L_M}(\lambda_M)$.
\begin{enumerate}[\em (a)]
\itemsep-.35em
    \item If $\mu_{\CZ}(\gamma_+)=\mu_{\CZ}(\gamma_-)$ then $\gamma_+\sim\gamma_-$.
    \item If $\mu_{\CZ}(\gamma_+)<\mu_{\CZ}(\gamma_-)$ then $\mathcal{A}(\gamma_+)<\mathcal{A}(\gamma_-)$.
\end{enumerate}
\end{lemma}

\begin{proof}
Write $[\gamma_{\pm}]\cong[g^m]$, for some $0\leq m<n$. To prove (a), assume $\mu_{\CZ}(\gamma_+)=\mu_{\CZ}(\gamma_-)$. Recall the Conley Zehnder index formulas \eqref{equation: CZ indices cyclic} from Section \ref{subsection: cyclic}: \begin{equation}\label{equation: CZ indices cyclic repeated}
    \mu_{\CZ}(\gamma_{\mathfrak{s}}^k)=2\Bigl\lceil \frac{2k}{n} \Bigr\rceil-1,\,\,\,\,\,\mu_{\CZ}(\gamma_{\mathfrak{n}}^k)=2\Bigl\lfloor \frac{2k}{n} \Bigr\rfloor+1.
\end{equation} We first argue that $\gamma_{\pm}$ both project to the same orbifold point. To see why,  note that an iterate of $\gamma_{\mathfrak{n}}$ representing the same homotopy class as an iterate of  $\gamma_{\mathfrak{s}}$ cannot share the same Conley Zehnder indices, as  the equality $\mu_{\CZ}(\gamma_{\mathfrak{n}}^{m+k_1n})=\mu_{\CZ}(\gamma_{\mathfrak{s}}^{n-m+k_2n})$ implies, by \eqref{equation: CZ indices cyclic repeated}, that $2 \lfloor \frac{2m}{n}\rfloor+1=2+2k_2-2k_1$. This equality cannot hold because the left hand side is odd while the right hand side is even. So, without loss of  generality, suppose both $\gamma_{\pm}$ are iterates of $\gamma_{\mathfrak{n}}$. Then we must have that $\gamma_{\pm}=\gamma_{\mathfrak{n}}^{m+k_{\pm}n}$. Again by \eqref{equation: CZ indices cyclic repeated}, the equality $\mu_{\CZ}(\gamma_+)=\mu_{\CZ}(\gamma_-)$ implies $k_+=k_-$ and thus $\gamma_+\sim\gamma_-$.

To prove (b), we suppose $\mu_{\CZ}(\gamma_+)<\mu_{\CZ}(\gamma_-)$. Recall the action formulas \eqref{equation: action cyclic} from Section \ref{subsection: cyclic}: \begin{equation}\label{equation: action cyclic repeat}
    \mathcal{A}(\gamma_{\mathfrak{s}}^k)=\frac{2\pi k(1-\varepsilon)}{n},\,\,\,\,\,\mathcal{A}(\gamma_{\mathfrak{n}}^k)=\frac{2\pi k(1+\varepsilon)}{n}.
\end{equation} If both $\gamma_{\pm}$ project to the same point, then  $m(\gamma_+)<m(\gamma_-)$, because the Conley Zehnder index is a non-decreasing function of the multiplicity, and thus, $\mathcal{A}(\gamma_+)<\mathcal{A}(\gamma_-)$. In the case that $\gamma_{\pm}$ project to different orbifold points there are two possibilities for the pair $(\gamma_+,\gamma_-)$:

\vspace{.25cm}

\emph{Case 1}:  $(\gamma_+,\gamma_-)=(\gamma_{\mathfrak{n}}^{m+nk_+},\gamma_{\mathfrak{s}}^{n-m+nk_-})$. By \eqref{equation: action cyclic repeat}, \[\mathcal{A}(\gamma_-)-\mathcal{A}(\gamma_+)=x+\delta,\,\,\,\, \text{where}\,\,\,\, x=2\pi\bigg(1+k_--k_+-\frac{2m}{n}\bigg)\] and $\delta$ can be made arbitrarily small, independent of $m$, $n$, and $k_{\pm}$. Thus, $0<x$ would imply $\mathcal{A}(\gamma_+)<\mathcal{A}(\gamma_-)$, after reducing $\varepsilon_N$ and $\varepsilon_M$ if necessary. By  \eqref{equation: CZ indices cyclic repeated}, the inequality $\mu_{\CZ}(\gamma_+)<\mu_{\CZ}(\gamma_-)$ yields \begin{equation}\label{equation: case 1 cyclic inequality}
    k_++\Bigl\lfloor\frac{2m}{n}\Bigr\rfloor\leq k_-.
\end{equation} 
If $2m<n$, then $\lfloor\frac{2m}{n}\rfloor=0$ and \eqref{equation: case 1 cyclic inequality} implies $k_+\leq k_-$. Now we see that \[x=2\pi\bigg(1+k_--k_+-\frac{2m}{n}\bigg)\geq2\pi\bigg(1-\frac{2m}{n}\bigg)>0,\] thus $\mathcal{A}(\gamma_+)<\mathcal{A}(\gamma_-)$. If $2m\geq n$, then $\lfloor\frac{2m}{n}\rfloor=1$ and \eqref{equation: case 1 cyclic inequality} implies $k_++1\leq k_-$. Now we see that \[x=2\pi\bigg(1+k_--k_+-\frac{2m}{n}\bigg)\geq2\pi\bigg(2-\frac{2m}{n}\bigg)>0,\] hence $\mathcal{A}(\gamma_+)<\mathcal{A}(\gamma_-)$.

\vspace{.25cm}

\emph{Case 2}: $(\gamma_+,\gamma_-)=(\gamma_{\mathfrak{s}}^{n-m+nk_+},\gamma_{\mathfrak{n}}^{m+nk_-})$. By  \eqref{equation: action cyclic repeat},  \[\mathcal{A}(\gamma_-)-\mathcal{A}(\gamma_+)=x+\delta,\,\,\,\, \text{where}\,\,\,\, x=2\pi\bigg(\frac{2m}{n}+k_--k_+-1\bigg)\] and $\delta$ is a small positive number. Thus, $0\leq x$ would imply $\mathcal{A}(\gamma_+)<\mathcal{A}(\gamma_-)$. Applying \eqref{equation: CZ indices cyclic repeated}, $\mu_{\CZ}(\gamma_+)<\mu_{\CZ}(\gamma_-)$ yields \begin{equation}\label{equation: case 2 cyclic inequality}
    k_+-\Bigl\lfloor\frac{2m}{n}\Bigr\rfloor<k_-.
\end{equation} 
If $2m<n$, then $\lfloor\frac{2m}{n}\rfloor=0$ and \eqref{equation: case 2 cyclic inequality} implies $k_++1\leq k_-$. Now we see that \[x=2\pi\bigg(\frac{2m}{n}+k_--k_+-1\bigg)\geq2\pi\bigg(\frac{2m}{n}\bigg)\geq0,\] thus $\mathcal{A}(\gamma_+)<\mathcal{A}(\gamma_-)$. If $2m\geq n$, then $\lfloor\frac{2m}{n}\rfloor=1$ and \eqref{equation: case 2 cyclic inequality} implies $k_+\leq k_-$. Now we see that \[x=2\pi\bigg(\frac{2m}{n}+k_--k_+-1\bigg)\geq2\pi\bigg(\frac{2m}{n}-1\bigg)\geq0,\] thus $\mathcal{A}(\gamma_+)<\mathcal{A}(\gamma_-)$.
\end{proof}

\subsubsection{Binary dihedral groups $\D_{2n}^*$}

The nonabelian group $\D_{2n}^*$ has $4n$ elements and has $n+3$  conjugacy classes. For  $0<m<n$,  the conjugacy class of $A^m$ is $[A^m]=\{A^m, A^{2n-m}\}$ (see Section \ref{subsection: dihedral} for notation). Because $-\text{Id}$ generates the center, we have $[-\text{Id}]=\{-\text{Id}\}$. We also have the conjugacy class $[\text{Id}]=\{\text{Id}\}$. The final two conjugacy classes are \[[B]=\{B, A^2B, \dots, A^{2n-2}B\}, \,\,\,\text{and}\,\,\,[AB]=\{AB, A^3B, \dots, A^{2n-1}B\}.\]
Note that $B^{-1}=B^3=-B=A^nB$ is conjugate to $B$ if and only if $n$ is even. Table \ref{table: lifts of dihedral reeb orbits} records lifts of our three embedded Reeb orbits, $e_-, h,$ and $e_+$, to paths $\widetilde{\gamma}$ in $S^3$, along with the group element $g\in\D^*_{2n}$ satisfying $g\cdot\widetilde{\gamma}(0)=\widetilde{\gamma}(T)$.

\begin{table}[h!]
\centering
\begingroup
\setlength{\tabcolsep}{10pt}
\renewcommand{\arraystretch}{1.4}
 \begin{tabular}{||c | c | c | c ||} 
 \hline
$\widetilde{\gamma}$ & $S^3$ expression & interval & $g\in\D^*_{2n}$ \\ [0.5ex] 
 \hline\hline
 $\widetilde{e_-}(t)$ & $t\mapsto \frac{e^{it}}{\sqrt{2}}\cdot (1,-i\epsilon)$ & $t\in[0,\frac{\pi}{2}]$ & $AB$ \\ 
 \hline
 $\widetilde{h}(t)$ & $t\mapsto \frac{e^{it}}{\sqrt{2}}\cdot(e^{i\pi/4},-e^{i\pi/4})$ &  $t\in[0,\frac{\pi}{2}]$ & $B$ \\
 \hline
 $\widetilde{e_+}(t)$ & $t\mapsto (e^{it},0)$ & $t\in[0,\frac{\pi}{n}]$ & $A$ \\
 \hline
\end{tabular}
\endgroup
 \caption{Lifts of Dihedral Reeb orbits to $S^3$}
 \label{table: lifts of dihedral reeb orbits}
\end{table}

The homotopy classes of $e_-$, $h$, and $e_+$ and their iterates are determined by this data and given in Table \ref{table: Dihedral homotopy classes of Reeb orbits}. We now prove Lemma \ref{lemma: dihedral cobordisms}, the dihedral case of Proposition \ref{proposition: CZ and action}.
\begin{table}[h!]
 \begin{tabular}{||c | c | c ||} 
 \hline
Free homotopy/conjugacy class & Represented orbits ($n$ even) &Represented orbits ($n$ odd) \\ [0.5ex] 
 \hline\hline
 $[\text{Id}]$ & $e_-^{4k},\,h^{4k},\,e_+^{2nk}$ & $e_-^{4k},\,h^{4k},\,e_+^{2nk}$ \\ 
 \hline
 $[-\text{Id}]$ &  $e_-^{2+4k},\,h^{2+4k},\,e_+^{n+2nk}$ & $e_-^{2+4k},\,h^{2+4k},\,e_+^{n+2nk}$ \\
 \hline
 $[A^m]$, for $0<m<n$ & $e_+^{m+2nk},\, e_+^{2n-m+2nk}$ & $e_+^{m+2nk},\, e_+^{2n-m+2nk}$ \\
 \hline
  $[B]$& $h^{1+4k},\,h^{3+4k}$ & $h^{1+4k},\,e_-^{3+4k}$ \\
 \hline
   $[AB]$ & $e_-^{1+4k},\,e_-^{3+4k}$ & $e_-^{1+4k},\,h^{3+4k}$ \\
 \hline
\end{tabular}
 \caption{Dihedral homotopy classes of Reeb orbits}
 \label{table: Dihedral homotopy classes of Reeb orbits}
\end{table}

\begin{lemma}\label{lemma: dihedral cobordisms}
Suppose $[\gamma_+]=[\gamma_-]\in[S^1,S^3/\D^*_{2n}]$ for $\gamma_+\in\mathcal{P}^{L_N}(\lambda_N)$ and $\gamma_-\in\mathcal{P}^{L_M}(\lambda_M)$.
\begin{enumerate}[\em (a)]
\itemsep-.35em
    \item If $\mu_{\CZ}(\gamma_+)=\mu_{\CZ}(\gamma_-)$ then $\gamma_+\sim\gamma_-$.
    \item If $\mu_{\CZ}(\gamma_+)<\mu_{\CZ}(\gamma_-)$ then $\mathcal{A}(\gamma_+)<\mathcal{A}(\gamma_-)$.
\end{enumerate}
\end{lemma}
\begin{proof}
We first prove (a). Recall the Conley Zehnder index formulas \eqref{equation: CZ indices dihedral} from Section \ref{subsection: dihedral}: \begin{equation} \label{equation: CZ indices dihedral repeat}
    \mu_{\CZ}(e_-^k)=2\Bigl\lceil \frac{k}{2} \Bigr\rceil-1,\,\,\,\mu_{\CZ}(h^k)=k,\,\,\,\mu_{\CZ}(e_+^k)=2\Bigl\lfloor \frac{k}{n} \Bigr\rfloor+1.
\end{equation} By Table \ref{table: Dihedral homotopy classes of Reeb orbits}, there are five possible values of the class $[\gamma_{\pm}]$:

        \vspace{.25cm}

        \emph{Case 1:} $[\gamma_{\pm}]\cong[\text{Id}]$. Then $\{\gamma_{\pm}\}\subset\{e_-^{4k_1}, h^{4k_2}, e_+^{2nk_3}\,|\,k_i\in\N\}$. The Conley Zehnder index modulo 4 of $e_-^{4k_1}$, $h^{4k_2}$, or $e_+^{2nk_3}$ is $-1$, 0, or $1$, respectively. Thus, $\gamma_{\pm}$ must project to the same orbifold point. Now, because (1) $\gamma_{\pm}$ are both type $e_-$, type $h$, or type $e_+$, and (2) share a homotopy class, $\mu_{\CZ}(\gamma_+)=\mu_{\CZ}(\gamma_-)$ and \eqref{equation: CZ indices dihedral repeat} imply $m(\gamma_+)=m(\gamma_-)$, thus $\gamma_+\sim\gamma_-$.
        
        \vspace{.25cm}
        
        \emph{Case 2:} $[\gamma_{\pm}]\cong[-\text{Id}]$. Then $\{\gamma_{\pm}\}\subset\{e_-^{2+4k_1}, h^{2+4k_2}, e_+^{n+2nk_3}\,|\,k_i\in\Z_{\geq0}\}$. The Conley Zehnder index modulo 4 of $e_-^{4k_1}$, $h^{4k_2}$, or $e_+^{2nk_3}$ is $1$, 2, or $3$, respectively. By the reasoning as in the previous case, we obtain $\gamma_+\sim\gamma_-$.
        
        \vspace{.25cm}
        
        \emph{Case 3:} $[\gamma_{\pm}]\cong[A^m]$ for some $0<m<n$. If $\gamma_+\nsim\gamma_-$, then by Table \ref{table: Dihedral homotopy classes of Reeb orbits} and \eqref{equation: CZ indices dihedral repeat}, we must have  $k_1, k_2\in\Z$ such that $\mu_{\CZ}(e_+^{m+2nk_1})=\mu_{\CZ}(e_+^{2n-m+2nk_2})$. This equation becomes $\lfloor\frac{m}{n}\rfloor+\lceil\frac{m}{n}\rceil=2(1+k_2-k_1)$. By the bounds on $m$, the ratio $\frac{m}{n}$ is not an integer, and so the quantity on the left hand side is odd, which is impossible. Thus, we must have $\gamma_+\sim\gamma_-$.
        
        \vspace{.25cm}
        
        \emph{Case 4:} $[\gamma_{\pm}]\cong[B]$. If $n$ is even, then both $\gamma_{\pm}$ are iterates of $h$, and by \eqref{equation: CZ indices dihedral repeat}, these multiplicities agree, so $\gamma_+\sim\gamma_-$. If $n$ is odd and $\gamma_{\pm}$ project to the same orbifold point, then because their homotopy classes and Conley Zehnder indices agree, Table \ref{table: Dihedral homotopy classes of Reeb orbits} and \eqref{equation: CZ indices dihedral repeat} imply their multiplicities must agree, so $\gamma_+\sim\gamma_-$.  If they project to different orbifold points then  $1+4k_+=\mu_{\CZ}(h^{1+4k_+})=\mu_{\CZ}(e_-^{3+4k_-})=3+4k_-$, which would imply $1=3$ mod 4, impossible.
        
        \vspace{.25cm}
        
        \emph{Case 5:} $[\gamma_{\pm}]\cong[AB]$. If $n$ is even, then both $\gamma_{\pm}$ are iterates of $e_-$, and by \eqref{equation: CZ indices dihedral repeat}, their multiplicities agree, thus $\gamma_+\sim\gamma_-$. If $n$ is odd and $\gamma_{\pm}$ project to the same point, then $\gamma_+\sim\gamma_-$ for the same reasons as in Case 4. Thus $\gamma_+\nsim\gamma_-$ implies  $3+4k_+=\mu_{\CZ}(h^{3+4k_+})=\mu_{\CZ}(e_-^{1+4k_-})=1+4k_-$, impossible mod 4.

        \vspace{.25cm}

To prove (b), let $\mu_{\CZ}(\gamma_+)<\mu_{\CZ}(\gamma_-)$. Recall the action formulas \eqref{equation: action dihedral} from Section \ref{subsection: dihedral}: \begin{equation}\label{equation: action dihedral repeat}
        \mathcal{A}(e_-^k)=\frac{k\pi(1-\varepsilon)}{2},\,\,\,\mathcal{A}(h^k)=\frac{k\pi}{2},\,\,\,\mathcal{A}(e_+^k)=\frac{k\pi(1+\varepsilon)}{n}.
\end{equation} First, note that if $\gamma_{\pm}$ project to the same orbifold critical point,  then $m(\gamma_+)<m(\gamma_-)$, because the Conley Zehnder index as a function of multiplicity of the iterate is non-decreasing. This implies  $\mathcal{A}(\gamma_+)<\mathcal{A}(\gamma_-)$, because the action strictly increases as a function of the iterate. We now prove (b) for pairs of orbits $\gamma_{\pm}$ projecting to different orbifold critical points.

        \vspace{.25cm}
        
        \emph{Case 1:} $[\gamma_{\pm}]\cong[\text{Id}]$. Then $\{\gamma_{\pm}\}\subset\{e_-^{4k_1}, h^{4k_2}, e_+^{2nk_3}\,|\,k_i\in\N\}$. There are six  combinations of the value of $(\gamma_+,\gamma_-)$, whose index and action are compared using \eqref{equation: CZ indices dihedral repeat} and \eqref{equation: action dihedral repeat}:
        \begin{enumerate}[(i)]
        \itemsep-.35em
            \item $\gamma_+=e_-^{4k_+}$ and $\gamma_-=h^{4k_-}$. The inequality $4k_+-1=\mu_{\CZ}(\gamma_+)<\mu_{\CZ}(\gamma_-)=4k_-$ implies $k_+\leq k_-$. This verifies that $\mathcal{A}(\gamma_+)=2\pi k_+(1-\varepsilon_N)<2\pi k_-=\mathcal{A}(\gamma_-)$.
            \item $\gamma_+=h^{4k_+}$ and $\gamma_-=e_-^{4k_-}$. The inequality $4k_+=\mu_{\CZ}(\gamma_+)<\mu_{\CZ}(\gamma_-)=4k_--1$ implies $k_+< k_-$. This verifies that $\mathcal{A}(\gamma_+)=2\pi k_+<2\pi k_-(1-\varepsilon_M)=\mathcal{A}(\gamma_-)$.
            \item $\gamma_+=e_-^{4k_+}$ and $\gamma_-=e_+^{2nk_-}$. The inequality $4k_+-1=\mu_{\CZ}(\gamma_+)<\mu_{\CZ}(\gamma_-)=4k_-+1$ implies $k_+\leq k_-$. This verifies that $\mathcal{A}(\gamma_+)=2\pi k_+(1-\varepsilon_N)<2\pi k_-(1+\varepsilon_M)=\mathcal{A}(\gamma_-)$.
            \item $\gamma_+=e_+^{2nk_+}$ and $\gamma_-=e_-^{4k_-}$. The inequality $4k_++1=\mu_{\CZ}(\gamma_+)<\mu_{\CZ}(\gamma_-)=4k_--1$ implies $k_+< k_-$. This verifies that $\mathcal{A}(\gamma_+)=2\pi k_+(1+\varepsilon_N)<2\pi k_-(1-\varepsilon_M)=\mathcal{A}(\gamma_-)$.
            \item $\gamma_+=h^{4k_+}$ and $\gamma_-=e_+^{2nk_-}$. The inequality $4k_+=\mu_{\CZ}(\gamma_+)<\mu_{\CZ}(\gamma_-)=4k_-+1$ implies $k_+\leq k_-$. This verifies that $\mathcal{A}(\gamma_+)=2\pi k_+<2\pi k_-(1+\varepsilon_M)=\mathcal{A}(\gamma_-)$.
            \item $\gamma_+=e_+^{2nk_+}$ and $\gamma_-=h^{4k_-}$. The inequality $4k_++1=\mu_{\CZ}(\gamma_+)<\mu_{\CZ}(\gamma_-)=4k_-$ implies $k_+< k_-$. This verifies that $\mathcal{A}(\gamma_+)=2\pi k_+(1+\varepsilon_N)<2\pi k_-=\mathcal{A}(\gamma_-)$.
        \end{enumerate}
                
        \vspace{.25cm}
        
        \emph{Case 2:} $[\gamma_{\pm}]\cong[-\text{Id}]$. Then $\{\gamma_{\pm}\}\subset\{e_-^{2+4k_1}, h^{2+4k_2}, e_+^{n+2nk_3}\,|\,k_i\in\Z_{\geq0}\}$.  The possible values of $(\gamma_+,\gamma_-)$ and the  arguments for $\mathcal{A}(\gamma_+)<\mathcal{A}(\gamma_-)$ are identical to the above case.
                        
        \vspace{.25cm}
        
        \emph{Case 3:} $[\gamma_{\pm}]\cong[A^m]$ for $0<m<n$ and both $\gamma_{\pm}$ are iterates of $e_+$, so (b) holds.

        \vspace{.25cm}
        
        \emph{Case 4:} $[\gamma_{\pm}]\cong[B]$. If $n$ is even, then both $\gamma_{\pm}$ are iterates of $h$, and so (b) holds. Otherwise, $n$ is odd and the pair $(\gamma_+,\gamma_-)$ is either $(e_-^{3+4k_+},h^{1+4k_-})$ or $(h^{1+4k_+}, e_-^{3+4k_-})$. In the former case $\mu_{\CZ}(\gamma_+)<\mu_{\CZ}(\gamma_-)$ and \eqref{equation: CZ indices dihedral repeat} imply $k_+< k_-$, so, by \eqref{equation: action dihedral repeat},\\ $\mathcal{A}(\gamma_+)=\frac{(3+4k_+)\pi}{2}(1-\varepsilon_N)<\frac{(1+4k_-)\pi}{2}=\mathcal{A}(\gamma_-)$. In the latter case, $\mu_{\CZ}(\gamma_+)<\mu_{\CZ}(\gamma_-)$ and \eqref{equation: CZ indices dihedral repeat} imply $k_+\leq k_-$, thus $\mathcal{A}(\gamma_+)=\frac{(1+4k_+)\pi}{2}<\frac{(3+4k_-)\pi}{2}(1-\varepsilon_M)=\mathcal{A}(\gamma_-)$ by \eqref{equation: action dihedral repeat}. 

        \vspace{.25cm}
        
        \emph{Case 5:} $[\gamma_{\pm}]\cong[AB]$. If $n$ is even, then both $\gamma_{\pm}$ are iterates of $e_-$, and so $\mathcal{A}(\gamma_+)<\mathcal{A}(\gamma_-)$ holds.  Otherwise, $n$ is odd and the pair $(\gamma_+,\gamma_-)$ is either $(e_-^{1+4k_+},h^{3+4k_-})$ or $(h^{3+4k_+}, e_-^{1+4k_-})$. If the former holds, then $\mu_{\CZ}(\gamma_+)<\mu_{\CZ}(\gamma_-)$ and \eqref{equation: CZ indices dihedral repeat} imply $k_+\leq k_-$, and so, by \eqref{equation: action dihedral repeat},\\ $\mathcal{A}(\gamma_+)=\frac{(1+4k_+)\pi}{2}(1-\varepsilon_N)<\frac{(3+4k_-)\pi}{2}=\mathcal{A}(\gamma_-)$. If the latter holds, then $\mu_{\CZ}(\gamma_+)<\mu_{\CZ}(\gamma_-)$ and \eqref{equation: CZ indices dihedral repeat} imply $k_+< k_-$, and so $\mathcal{A}(\gamma_+)=\frac{(3+4k_+)\pi}{2}<\frac{(1+4k_-)\pi}{2}(1-\varepsilon_M)=\mathcal{A}(\gamma_-)$ by \eqref{equation: action dihedral repeat}. 
\end{proof}

\subsubsection{Binary polyhedral groups $\T^*$, $\Oc^*$, and $\I^*$}\label{subsubsection: binary polyhedral}
We will describe the homotopy classes of the Reeb orbits in $S^3/\P^*$ using a more geometric point of view than in the dihedral case. If a loop $\gamma$ in $S^3/\P^*$ and  $c\in\text{Conj}(G)$ satisfy $[\gamma]\cong c$, then the order of $\gamma$ in $\pi_1(S^3/\P^*)$ equals the \emph{group order} of $c$, defined to be the order of any $g\in\P^*$ representing $c$. If $\gamma$ has order $k$ in the fundamental group and $c\in\text{Conj}(\P^*)$ is the \emph{only} class with group order $k$, then we can immediately conclude that $[\gamma]\cong c$. Determining the free homotopy class of $\gamma$ via the group order is more difficult when there are multiple conjugacy classes of $\P^*$ with the same group order. 

Tables \ref{table: tetrahedral conjugacy}, \ref{table: octahedral conjugacy}, and \ref{table: icosahedral conjugacy} provide notation for the distinct conjugacy classes of $\T^*$, $\Oc^*$, and $\I^*$, along with their group orders. Our notation indicates when there exist multiple conjugacy classes featuring the same group order - the notation $P_{i,A}$ and $P_{i,B}$ provides labels for the two distinct conjugacy classes of $\P^*$ (for $P\in\{T,O,I\}$) of group order $i$. For $P\in\{T,O,I\}$, $P_{\text{Id}}$ and $P_{-\text{Id}}$ denote the singleton conjugacy classes $\{\text{Id}\}$ and $\{-\text{Id}\}$, respectively, and $P_i$ denotes the unique conjugacy class of group order $i$.

\begin{table}[h!]
\centering
    \begin{tabular}{||c ||c | c | c |c | c | c | c ||}
    \hline
       Conjugacy class  & $T_{\text{Id}}$ & $T_{-\text{Id}}$ & $T_4$ & $T_{6,A}$ & $T_{6,B}$ & $T_{3,A}$ & $T_{3,B}$ \\
    \hline
    Group order  & 1 & 2 & 4 & 6 & 6 & 3 & 3 \\
    \hline
    \end{tabular}
    \caption{The 7 conjugacy classes of $\T^*$}
    \label{table: tetrahedral conjugacy}
\end{table}
\begin{table}[h!]
\centering
    \begin{tabular}{||c ||c | c | c | c |c | c | c | c ||}
    \hline
       Conjugacy class  & $O_{\text{Id}}$ & $O_{-\text{Id}}$ & $O_{8,A}$ & $O_{8,B}$ & $O_{4,A}$ & $O_{4,B}$ & $O_{6}$ & $O_3$ \\
    \hline
    Group order  & 1 & 2 & 8 & 8 & 4 & 4 & 6 & 3 \\
    \hline
    \end{tabular}
        \caption{The 8 conjugacy classes of $\Oc^*$}
    \label{table: octahedral conjugacy}
\end{table}
\begin{table}[h!]
\centering
    \begin{tabular}{||c ||c | c | c | c |c | c | c | c | c ||}
    \hline
       Conjugacy class  & $I_{\text{Id}}$ & $I_{-\text{Id}}$ & $I_{10,A}$ & $I_{10,B}$ & $I_{5,A}$ & $I_{5,B}$ & $I_{4}$ & $I_6$  & $I_3$\\
    \hline
    Group order  & 1 & 2 & 10 & 10 & 5 & 5 & 4 & 6 & 3 \\
    \hline
    \end{tabular}
        \caption{The 9 conjugacy classes of $\I^*$}
    \label{table: icosahedral conjugacy}
\end{table}

Lemmas \ref{remark: identify}, \ref{remark: distinguish 1}, and \ref{remark: distinguish 2} allow us to record the homotopy classes represented by any iterate of $\mathcal{V}$, $\mathcal{E}$, or $\mathcal{F}$ in Tables \ref{table: Tetrahedral homotopy classes of Reeb orbits}, \ref{table: octahedral homotopy classes of Reeb orbits}, and \ref{table: icosahedral homotopy classes of Reeb orbits}. We explain how the table is set up in the tetrahedral case: it must be true that $[\mathcal{V}]\neq[\mathcal{F}]$, otherwise taking the 2-fold iterate would violate Lemma \ref{remark: distinguish 1}(i) so without loss of generality write $[\mathcal{V}]\cong T_{6,A}$ and $[\mathcal{F}]\cong T_{6,B}$, and similarly $[\mathcal{V}^2]\cong T_{3,A}$ and $[\mathcal{F}^2]\cong T_{3,B}$. By Lemma \ref{remark: identify}(i), we must have that $[\mathcal{F}^5]\cong T_{6,A}$,  $[\mathcal{F}^4]\cong T_{3,A}$. Taking the 4-fold iterate of $[\mathcal{V}]=[\mathcal{F}^5]$ provides $[\mathcal{V}^4]=[\mathcal{F}^2]\cong T_{3,B}$, and the 5-fold iterate provides $[\mathcal{V}^5]=[\mathcal{F}^1]\cong T_{6,B}$, and we have resolved all ambiguity regarding the tetrahedral classes of group orders 6 and 3. Analogous arguments apply in the octahedral and icosahedral cases.

\begin{table}
\centering
\makebox[0pt][c]{\parbox{1.2\textwidth}{%
    \begin{minipage}[b]{0.32\hsize}\centering
        \begin{tabular}{||c | c ||} 
 \hline
Class & Represented orbits  \\ [0.5ex] 
 \hline\hline
 $T_{\text{Id}}$ & $\mathcal{V}^{6k},\,\mathcal{E}^{4k},\,\mathcal{F}^{6k}$ \\ 
 \hline
 $T_{-\text{Id}}$ &  $\mathcal{V}^{3+6k},\,\mathcal{E}^{2+4k},\,\mathcal{F}^{3+6k}$\\
 \hline
 $T_{4}$ &  $\mathcal{E}^{1+4k},\,\mathcal{E}^{3+4k}$\\
 \hline
  $T_{6,A}$& $\mathcal{V}^{1+6k},\,\mathcal{F}^{5+6k}$\\
 \hline
   $T_{6,B}$& $\mathcal{F}^{1+6k},\,\mathcal{V}^{5+6k}$\\
 \hline
   $T_{3,A}$& $\mathcal{V}^{2+6k},\,\mathcal{F}^{4+6k}$\\
 \hline
    $T_{3,B}$& $\mathcal{F}^{2+6k},\,\mathcal{V}^{4+6k}$\\
 \hline
\end{tabular}
 \caption{$S^3/\T^*$ Reeb Orbits}
 \label{table: Tetrahedral homotopy classes of Reeb orbits}
    \end{minipage}
    \hfill
    \begin{minipage}[b]{0.32\hsize}\centering
         \begin{tabular}{||c | c ||} 
 \hline
Class & Represented orbits  \\ [0.5ex] 
 \hline\hline
 $O_{\text{Id}}$ & $\mathcal{V}^{8k},\,\mathcal{E}^{4k},\,\mathcal{F}^{6k}$ \\ 
 \hline
 $O_{-\text{Id}}$ &  $\mathcal{V}^{4+8k},\,\mathcal{E}^{2+4k},\,\mathcal{F}^{3+6k}$\\
 \hline
 $O_{8,A}$ &  $\mathcal{V}^{1+8k},\,\mathcal{V}^{7+8k}$\\
 \hline
  $O_{8,B}$& $\mathcal{V}^{3+8k},\,\mathcal{V}^{5+8k}$\\
 \hline
  $O_{4,A}$ &  $\mathcal{V}^{2+8k},\,\mathcal{V}^{6+8k}$\\
 \hline
  $O_{4,B}$& $\mathcal{E}^{1+4k},\,\mathcal{E}^{3+4k}$\\
 \hline
   $O_{6}$& $\mathcal{F}^{1+6k},\,\mathcal{F}^{5+6k}$\\
 \hline
   $O_{3}$& $\mathcal{F}^{2+6k},\,\mathcal{F}^{4+6k}$\\
 \hline
\end{tabular}
\caption{$S^3/\Oc^*$ Reeb Orbits}
\label{table: octahedral homotopy classes of Reeb orbits}
    \end{minipage}
    \hfill
    \begin{minipage}[b]{0.32\hsize}\centering
        \begin{tabular}{||c | c ||} 
 \hline
Class & Represented orbits  \\ [0.5ex] 
 \hline\hline
 $I_{\text{Id}}$ & $\mathcal{V}^{10k},\,\mathcal{E}^{4k},\,\mathcal{F}^{6k}$ \\ 
 \hline
 $I_{-\text{Id}}$ &  $\mathcal{V}^{5+10k},\,\mathcal{E}^{2+4k},\,\mathcal{F}^{3+6k}$\\
 \hline
 $I_{10,A}$ &  $\mathcal{V}^{1+10k},\,\mathcal{V}^{9+10k}$\\
 \hline
  $I_{10,B}$& $\mathcal{V}^{3+10k},\,\mathcal{V}^{7+10k}$\\
 \hline
  $I_{5,A}$ &  $\mathcal{V}^{2+10k},\,\mathcal{V}^{8+10k}$\\
 \hline
  $I_{5,B}$& $\mathcal{V}^{4+10k},\,\mathcal{V}^{6+10k}$\\
 \hline
   $I_{4}$& $\mathcal{E}^{1+4k},\,\mathcal{E}^{3+4k}$\\
 \hline
   $I_{6}$& $\mathcal{F}^{1+6k},\,\mathcal{F}^{5+6k}$\\
 \hline
    $I_{3}$& $\mathcal{F}^{2+6k},\,\mathcal{F}^{4+6k}$\\
 \hline
\end{tabular}
\caption{$S^3/\I^*$ Reeb Orbits}
\label{table: icosahedral homotopy classes of Reeb orbits}
    \end{minipage}%
}}
\end{table}


\vfill
\eject

\begin{lemma} \label{remark: identify}
We have the following identifications between free homotopy classes represented by orbits: 
 \begin{enumerate}[\em (i)]
 \itemsep-.35em
     \item $[\mathcal{V}]=[\mathcal{F}^5]$, for $\P^*=\T^*$,
     \item $[\mathcal{V}]=[\mathcal{V}^7]$, for $\P^*=\Oc^*$,
     \item $[\mathcal{V}]=[\mathcal{V}^9]$, for $\P^*=\I^*$.
 \end{enumerate}
 \end{lemma}
 
 \begin{proof}
 Suppose $p_1$, $p_2\in \text{Fix}(\P)\subset S^2$ are antipodal, i.e. $p_1=-p_2$. Select $z_i\in \fP^{-1}(p_i)\subset S^3$. The Hopf fibration $\fP$ has the property that $\fP(v_1)=-\fP(v_2)$ in $S^2$ if and only if $v_1$ and $v_2$ are orthogonal vectors in $\C^2$, so $z_1$ and $z_2$ must be orthogonal. Now, $p_i$ is either of vertex, edge, or face type, so let $\gamma_i$ denote the orbit $\mathcal{V}$, $\mathcal{E}$, or $\mathcal{F}$, depending on this type of $p_i$. Because $p_1$ and $p_2$ are antipodal, the order of $\gamma_1$ equals that of $\gamma_2$ in $\pi_1(S^3/\P^*)$, call this order $d$, and let $T:=\frac{2\pi}{d}$. Now, consider that the map \[\Gamma_1:[0,T]\to S^3,\,\,\, t\mapsto e^{it}\cdot z_1\] is a lift of $\gamma_1$ to $S^3$. Thus,  $z_1$ is an eigenvector with eigenvalue $e^{iT}$ for some $g\in\P^*$, and $[\gamma_1]\cong[g]$ because $g\cdot\Gamma_1(0)=\Gamma_1(T)$. Because $g$ is special unitary, we must also have that $z_2$ is an eigenvector of $g$ with eigenvalue $e^{i(2\pi-T)}=e^{i(d-1)T}$. Now the map \[\Gamma_2^{d-1}:[0,(d-1)T]\to S^3,\,\,\,t\mapsto e^{it}\cdot z_2\] is a lift of $\gamma_2^{d-1}$ to $S^3$. This provides $g\cdot\Gamma_2^{d-1}(0)=\Gamma_2^{d-1}((d-1)T)$ which implies $[\gamma_1]=[\gamma_2^{d-1}]$, as both are identified with $[g]$, and thereby establishing the above identifications.
 \end{proof}

\begin{lemma}\label{remark: distinguish 1}
The following identifications fail to hold between free homotopy classes represented by orbits: 
 
 \begin{enumerate}[\em (i)]
 \itemsep-.35em
     \item $[\mathcal{V}^2]\neq[\mathcal{F}^2]$, for $\P^*=\T^*$,
     \item $[\mathcal{E}]\neq[\mathcal{V}^2]$, for $\P^*=\Oc^*$
 \end{enumerate}
\end{lemma}

 \begin{proof}

 Suppose, for $i=1,2$, $\gamma_i$ is one of $\mathcal{V}$, $\mathcal{E}$, or $\mathcal{F}$, and suppose $\gamma_1\neq\gamma_2$. Let $d_i$ be the order of $\gamma_i$ in $\pi_1(S^3/\P^*)$, and select $k_i\in\N$ for $i=1,2$. If $\frac{2\pi k_1}{d_1}\equiv \frac{2\pi k_2}{d_2}$ modulo $2\pi\Z$ and if $\pi\nmid\frac{2\pi k_1}{d_1}$, then $[\gamma_1^{k_1}]\neq[\gamma_2^{k_2}]$. To prove this, consider that we have $g_i\in\P^*$ with eigenvector $z_i$ in $S^3$, with eigenvalue $\lambda:=e^{2\pi k_1 /d_1}=e^{2\pi k_2 i/d_2}$ so that $[\gamma_i^{k_i}]\cong[g_i]$. Note that  $\gamma_1\neq\gamma_2$ implies $\fP(z_1)$ is not in the same $\P$-orbit as $\fP(z_2)$ in $S^2$, i.e. $\pi_{\P}(\fP(z_1))\neq\pi_{\P}(\fP(z_2))$. Now, $[\gamma_1^{k_1}]=[\gamma_2^{k_2}]$ would imply $g_1=x^{-1}g_2x$, for some $x\in\P^*$, ensuring that $x\cdot z_1$ is a $\lambda$ eigenvector of $g_2$. Because $\lambda\neq\pm1$, we know that the $\lambda$-eigenspace of $g_2$ is complex 1-dimensional, so we must have that $x\cdot z_1$ and $z_2$ are co-linear. That is, $x\cdot z_1=\alpha z_2$ for some $\alpha\in S^1$, which implies that \[\pi_{\P}(\fP(z_1))=\pi_{\P}(\fP(x\cdot z_1))=\pi_{\P}(\fP(\alpha\cdot z_2))=\pi_{\P}(\fP(z_2)),\] a contradiction. This yields the desired inequivalences. 

 \end{proof}

\begin{lemma}\label{remark: distinguish 2}
The following identifications fail to hold between free homotopy classes represented by iterates of the same orbits: 

  \begin{enumerate}[\em (i)]
 \itemsep-.35em
     \item $[\mathcal{V}^1]\neq[\mathcal{V}^3]$, for $\P^*=\Oc^*$,
     \item $[\mathcal{V}^2]\neq[\mathcal{V}^4]$, for $\P^*=\I^*$.
 \end{enumerate}
\end{lemma}

 \begin{proof}
 
 For $i=1,2$, select $k_i\in\N$ and let $\gamma_i$ denote one of $\mathcal{V}$, $\mathcal{E}$, or $\mathcal{F}$. Let $d_i$ denote the order of $\gamma_i$ in $\pi_1(S^3/\P^*)$. Suppose that $\frac{2\pi k_i}{d_i}$ is not a multiple of $2\pi$, and $\frac{2\pi k_1}{d_1}\ncong\frac{2\pi k_2}{d_2}$ mod $2\pi\Z$. If $\frac{2\pi k_1}{d_1}+\frac{2\pi k_2}{d_2}$ is not a multiple of $2\pi$, then $[\gamma_1^{k_1}]\neq[\gamma_2^{k_2}]$. To prove this, consider that we have $g_i\in\P^*$ with $[\gamma_i^{k_i}]\cong[g_i]$. This tells us that $e^{2\pi k_j i/d_j}$ is an eigenvalue of $g_j$. If it were the case that $[\gamma_1^{k_1}]=[\gamma_2^{k_2}]$ in $[S^1,S^3/\P^*]$, then we would have $[g_1]=[g_2]$ in $\text{Conj}(G)$. Because conjugate elements share eigenvalues, we would have \[\text{Spec}(g_1)=\text{Spec}(g_2)=\{e^{2\pi k_1 i/d_1},e^{2\pi k_2 i/d_2}\}.\]  However, the product of these eigenvalues is not 1, contradicting that $g_i\in\SU(2)$. 
 \end{proof}

We are ready to prove Proposition \ref{lemma: polyhedral cobordisms}, which is the polyhedral case of Proposition \ref{proposition: CZ and action}. Lemma \ref{remark: same underlying embedded orbits} will streamline some casework.

\begin{lemma}\label{remark: same underlying embedded orbits}
If $\gamma_+\in\mathcal{P}^{L_N}(\lambda_N)$ and $\gamma_-\in\mathcal{P}^{L_M}(\lambda_M)$ project to the same orbifold critical point and are in the same free homotopy class, then 
\begin{enumerate}[(i)]
\itemsep-.35em
    \item $\mu_{\CZ}(\gamma_+)=\mu_{\CZ}(\gamma_-)$ implies $m(\gamma_+)=m(\gamma_-)$, i.e., $\gamma_{+}\sim\gamma_-$;
    \item  $\mu_{\CZ}(\gamma_+)<\mu_{\CZ}(\gamma_-)$ implies $m(\gamma_+)<m(\gamma_-)$, and so  $\mathcal{A}(\gamma_+)<\mathcal{A}(\gamma_-)$.
\end{enumerate}
\end{lemma}
 \begin{proof}
For $\P^*=\T^*$, $\Oc^*$, or $\I^*$, fix $c\in\text{Conj}(\P^*)$ and let $\gamma$ denote one of $\mathcal{V}$, $\mathcal{E}$, or $\mathcal{F}$. Define $S_{\gamma, c}:=\{\gamma^k\,|\,[\gamma^k]\cong c\}$; note that this set may potentially be empty. Observe that the map $S_{\gamma, c}\to\Z$, $\gamma^k\mapsto\mu_{\CZ}(\gamma^k)$, is injective. Thus the result holds.
 \end{proof}

\begin{proposition}\label{lemma: polyhedral cobordisms}
Suppose $[\gamma_+]=[\gamma_-]\in[S^1,S^3/\P^*]$ for $\gamma_+\in\mathcal{P}^{L_N}(\lambda_N)$ and $\gamma_-\in\mathcal{P}^{L_M}(\lambda_M)$.
\begin{enumerate}[\em (a)]
\itemsep-.35em
    \item If $\mu_{\CZ}(\gamma_+)=\mu_{\CZ}(\gamma_-)$, then $\gamma_+\sim\gamma_-$.
    \item If $\mu_{\CZ}(\gamma_+)<\mu_{\CZ}(\gamma_-)$, then $\mathcal{A}(\gamma_+)<\mathcal{A}(\gamma_-)$.
\end{enumerate}
\end{proposition}

\begin{proof}

We first prove (a). Recall the Conley Zehnder index formulas \eqref{equation: CZ indices polyhedral} from Section \ref{subsection: polyhedral}: \begin{equation} \label{equation: CZ indices polyhedral repeat}
    \mu_{\CZ}(\mathcal{V}^k)=2\Bigl\lceil \frac{k}{\mathscr{I}_{\mathscr{V}}} \Bigr\rceil-1,\,\,\,\mu_{\CZ}(\mathcal{E}^k)=k,\,\,\,\mu_{\CZ}(\mathcal{F}^k)=2\Bigl\lfloor \frac{k}{3} \Bigr\rfloor+1,
\end{equation}Consider the following possible values of the homotopy class $[\gamma_{\pm}]$:
        
        \vspace{.25cm}
        
        \emph{Case 1:} $[\gamma_{\pm}]\cong T_{\text{Id}}$, $O_{\text{Id}}$, or $I_{\text{Id}}$, so that $\{\gamma_{\pm}\}\subset\{\mathcal{V}^{2\mathscr{I}_{\mathscr{V}}k_1}, \mathcal{E}^{4k_2}, \mathcal{F}^{6k_3}\,|\,k_i\in\N\}$. By reasoning identical to the analogous case of $\D_{2n}^*$ (Lemma \ref{lemma: dihedral cobordisms}(a), Case 1), $\gamma_+\sim\gamma_-$.         
        \vspace{.25cm}
        
        \emph{Case 2:} $[\gamma_{\pm}]\cong T_{-\text{Id}}$, $O_{-\text{Id}}$, or $I_{-\text{Id}}$, so that $\{\gamma_{\pm}\}\subset\{\mathcal{V}^{\mathscr{I}_{\mathscr{V}}+2\mathscr{I}_{\mathscr{V}}k_1}, \mathcal{E}^{2+4k_2}, \mathcal{F}^{3+6k_3}\,|\,k_i\in\Z_{\geq0}\}$. Again, as in the the analogous case of $\D_{2n}^*$ (Lemma \ref{lemma: dihedral cobordisms}(a), Case 2), $\gamma_+\sim\gamma_-$.        
        \vspace{.25cm}

        \emph{Case 3:} $[\gamma_{\pm}]\cong T_{6,A}, T_{6,B}, T_{3,A}$, or $T_{3,B}$. If both $\gamma_{\pm}$ are iterates of $\mathcal{V}$, then by Lemma \ref{remark: same underlying embedded orbits}(i), they must share the same multiplicity, i.e. $\gamma_+\sim\gamma_-$. If both $\gamma_{\pm}$ are iterates of $\mathcal{F}$ then again by Lemma \ref{remark: same underlying embedded orbits}(i), they must share the same multiplicity, i.e. $\gamma_+\sim\gamma_-$. So we must argue, using \eqref{equation: CZ indices polyhedral repeat}, that in each of these four free homotopy classes that it is impossible that $\gamma_{\pm}$ project to different orbifold points whenever $\mu_{\CZ}(\gamma_+)=\mu_{\CZ}(\gamma_-)$.
        \begin{itemize}
        \itemsep-.35em
            \item If $[\gamma_{\pm}]\cong T_{6,A}$ and $\gamma_{\pm}$ project to different orbifold points then, up to relabeling,\\ $\gamma_+=\mathcal{V}^{1+6k_+}$ and $\gamma_-=\mathcal{F}^{5+6k_-}$. Now, $\mu_{\CZ}(\gamma_+)=4k_++1\neq 4k_-+3=\mu_{\CZ}(\gamma_-)$.
            \item If $[\gamma_{\pm}]\cong T_{6,B}$ and $\gamma_{\pm}$ project to different orbifold points, write $\gamma_+=\mathcal{V}^{5+6k_+}$\\ and $\gamma_-=\mathcal{F}^{1+6k_-}$. Now, $\mu_{\CZ}(\gamma_+)=4k_++3\neq 4k_-+1=\mu_{\CZ}(\gamma_-)$.
            \item If $[\gamma_{\pm}]\cong T_{3,A}$ and $\gamma_{\pm}$ project to different orbifold points, write $\gamma_+=\mathcal{V}^{2+6k_+}$\\ and $\gamma_-=\mathcal{F}^{4+6k_-}$. Now, $\mu_{\CZ}(\gamma_+)=4k_++1\neq 4k_-+3=\mu_{\CZ}(\gamma_-)$.
            \item If $[\gamma_{\pm}]\cong T_{3,B}$ and $\gamma_{\pm}$ project to different orbifold points, write $\gamma_+=\mathcal{V}^{4+6k_+}$\\ and $\gamma_-=\mathcal{F}^{2+6k_-}$. Now, $\mu_{\CZ}(\gamma_+)=4k_++3\neq 4k_-+1=\mu_{\CZ}(\gamma_-)$.
        \end{itemize}
                
        \vspace{.25cm}
        
        \emph{Case 4:} $[\gamma_{\pm}]$ is a homotopy class not covered in Cases 1 - 3. Because every such homotopy class is represented by Reeb orbits either of type $\mathcal{V}$, of type $\mathcal{E}$, or of type $\mathcal{F}$, then we see that $\gamma_{\pm}$ project to the same orbifold point of $S^2/\P$. By Lemma \ref{remark: same underlying embedded orbits}(i), we have that $\gamma_{+}\sim\gamma_-$.
        
        \vspace{.25cm}
        
        We now prove (b). Recall the action formulas \eqref{equation: action polyhedral} from Section \ref{subsection: polyhedral}: \begin{equation}\label{equation: action polyhedral repeat}
        \mathcal{A}(\mathcal{V}^k)=\frac{k\pi(1-\varepsilon)}{\mathscr{I}_{\mathscr{V}}},\,\,\,\mathcal{A}(\mathcal{E}^k)=\frac{k\pi}{2},\,\,\,\mathcal{A}(\mathcal{F}^k)=\frac{k\pi(1+\varepsilon)}{3}.
\end{equation} Consider the following possible values of the homotopy class $[\gamma_{\pm}]$:
        \vspace{.25cm}
        
        \emph{Case 1:} $[\gamma_{\pm}]\cong T_{\text{Id}}$, $O_{\text{Id}}$, or $I_{\text{Id}}$, so that $\{\gamma_{\pm}\}\subset\{\mathcal{V}^{2\mathscr{I}_{\mathscr{V}}k_1}, \mathcal{E}^{4k_2}, \mathcal{F}^{6k_3}\,|\,k_i\in\N\}$. By reasoning identical to the analogous case of $\D_{2n}^*$ (Lemma \ref{lemma: dihedral cobordisms}(b), Case 1), $\mathcal{A}(\gamma_+)<\mathcal{A}(\gamma_-)$.         
        \vspace{.25cm}
        
        \emph{Case 2:} $[\gamma_{\pm}]\cong T_{-\text{Id}}$, $O_{-\text{Id}}$, or $I_{-\text{Id}}$, so that $\{\gamma_{\pm}\}\subset\{\mathcal{V}^{\mathscr{I}_{\mathscr{V}}+2\mathscr{I}_{\mathscr{V}}k_1}, \mathcal{E}^{2+4k_2}, \mathcal{F}^{3+6k_3}\,|\,k_i\in\Z_{\geq0}\}$. Again, as in the the analogous case of $\D_{2n}^*$ (Lemma \ref{lemma: dihedral cobordisms}(b), Case 2), $\mathcal{A}(\gamma_+)<\mathcal{A}(\gamma_-)$.   
        
        \vspace{.25cm}

        \emph{Case 3:} $[\gamma_{\pm}]\cong T_{6,A}, T_{6,B}, T_{3,A}$, or $T_{3,B}$. If both $\gamma_{\pm}$ are of type $\mathcal{V}$, then by Lemma \ref{remark: same underlying embedded orbits}(ii), $\mathcal{A}(\gamma_+)<\mathcal{A}(\gamma_-)$. If both $\gamma_{\pm}$ are of type $\mathcal{F}$, then again by Lemma \ref{remark: same underlying embedded orbits}(ii), $\mathcal{A}(\gamma_+)<\mathcal{A}(\gamma_-)$. So we must argue, using \eqref{equation: CZ indices polyhedral repeat} and \eqref{equation: action polyhedral repeat}, that for each of these four free homotopy classes that if $\gamma_+$ and $\gamma_-$ are not of the same type, and if $\mu_{\CZ}(\gamma_+)<\mu_{\CZ}(\gamma_-)$, then $\mathcal{A}(\gamma_+)<\mathcal{A}(\gamma_-)$.

        \emph{3.A} If $[\gamma_{\pm}]\cong T_{6,A}$, then the pair $(\gamma_+,\gamma_-)$ is either $(\mathcal{V}^{1+6k_+}, \mathcal{F}^{5+6k_-})$ or $(\mathcal{F}^{5+6k_+}, \mathcal{V}^{1+6k_-})$. If the former holds, then $\mu_{\CZ}(\gamma_+)<\mu_{\CZ}(\gamma_-)$ implies $k_+\leq k_-$, and so \[\mathcal{A}(\gamma_+)=\frac{(1+6k_+)\pi}{3}(1-\varepsilon_N)<\frac{(5+6k_-)\pi}{3}(1+\varepsilon_M)=\mathcal{A}(\gamma_-).\] If the latter holds, then $\mu_{\CZ}(\gamma_+)<\mu_{\CZ}(\gamma_-)$ implies $k_+<k_-$, and again the action satisfies \[\mathcal{A}(\gamma_+)=\frac{(5+6k_+)\pi}{3}(1+\varepsilon_N)<\frac{(1+6k_-)\pi}{3}(1-\varepsilon_M)=\mathcal{A}(\gamma_-).\] 
        
         \emph{3.B} If $[\gamma_{\pm}]\cong T_{6,B}$, then the pair $(\gamma_+,\gamma_-)$ is either $(\mathcal{V}^{5+6k_+}, \mathcal{F}^{1+6k_-})$ or $(\mathcal{F}^{1+6k_+}, \mathcal{V}^{5+6k_-})$. If the former holds, then  $\mu_{\CZ}(\gamma_+)<\mu_{\CZ}(\gamma_-)$ implies $k_+<k_-$, and so \[\mathcal{A}(\gamma_+)=\frac{(5+6k_+)\pi}{3}(1-\varepsilon_N)<\frac{(1+6k_-)\pi}{3}(1+\varepsilon_M)=\mathcal{A}(\gamma_-).\] If the latter holds, then $\mu_{\CZ}(\gamma_+)<\mu_{\CZ}(\gamma_-)$ implies $k_+\leq k_-$, and so \[\mathcal{A}(\gamma_+)=\frac{(1+6k_+)\pi}{3}(1+\varepsilon_N)<\frac{(5+6k_-)\pi}{3}(1-\varepsilon_M)=\mathcal{A}(\gamma_-).\] 
        
          \emph{3.C}  If $[\gamma_{\pm}]\cong T_{3,A}$, then the pair $(\gamma_+,\gamma_-)$ is either $(\mathcal{V}^{2+6k_+}, \mathcal{F}^{4+6k_-})$ or $(\mathcal{F}^{4+6k_+}, \mathcal{V}^{2+6k_-})$. If the former holds, then $\mu_{\CZ}(\gamma_+)<\mu_{\CZ}(\gamma_-)$ implies $k_+\leq k_-$, and so \[\mathcal{A}(\gamma_+)=\frac{(2+6k_+)\pi}{3}(1-\varepsilon_N)<\frac{(4+6k_-)\pi}{3}(1+\varepsilon_M)=\mathcal{A}(\gamma_-).\] If the latter holds, then $\mu_{\CZ}(\gamma_+)<\mu_{\CZ}(\gamma_-)$ implies $k_+<k_-$, and so \[\mathcal{A}(\gamma_+)=\frac{(4+6k_+)\pi}{3}(1+\varepsilon_N)<\frac{(2+6k_-)\pi}{3}(1-\varepsilon_M)=\mathcal{A}(\gamma_-).\] 
        
        \emph{3.D} If $[\gamma_{\pm}]\cong T_{3,B}$, then the pair $(\gamma_+,\gamma_-)$ is either $(\mathcal{V}^{4+6k_+}, \mathcal{F}^{2+6k_-})$ or $(\mathcal{F}^{2+6k_+}, \mathcal{V}^{4+6k_-})$. If the former holds, then $\mu_{\CZ}(\gamma_+)<\mu_{\CZ}(\gamma_-)$ implies $k_+<k_-$, and so \[\mathcal{A}(\gamma_+)=\frac{(4+6k_+)\pi}{3}(1-\varepsilon_N)<\frac{(2+6k_-)\pi}{3}(1+\varepsilon_M)=\mathcal{A}(\gamma_-).\] If the latter holds, then $\mu_{\CZ}(\gamma_+)<\mu_{\CZ}(\gamma_-)$ implies $k_+\leq k_-$, and so \[\mathcal{A}(\gamma_+)=\frac{(2+6k_+)\pi}{3}(1+\varepsilon_N)<\frac{(4+6k_-)\pi}{3}(1-\varepsilon_M)=\mathcal{A}(\gamma_-).\] 
                
        \vspace{.25cm}
        
        \emph{Case 4:} $[\gamma_{\pm}]$ is a homotopy class not covered in Cases 1 - 3. Because every such homotopy class is represented by Reeb orbits either of type $\mathcal{V}$, of type $\mathcal{E}$, or of type $\mathcal{F}$, then we see that $\gamma_{\pm}$ project to the same orbifold point of $S^2/\P$. By Lemma \ref{remark: same underlying embedded orbits}(ii),  $\mathcal{A}(\gamma_{+})<\mathcal{A}(\gamma_{-})$.
      
\end{proof}




\noindent \textsc{Leo Digiosia \\  Rice University, PhD 2022 \\ Quantitative Model Development Analyst at US Bank}\\
{\em email: }\texttt{digiosialeo@gmail.com}\\

\noindent \textsc{Jo Nelson \\  Rice University}\\
{\em email: }\texttt{jo.nelson@rice.edu}\\


\addcontentsline{toc}{section}{References}


\begin{thebibliography}{CFHWI}

    \bibitem[AHNS17]{AHNS} L. Abbrescia, I. Huq-Kuruvilla, J. Nelson, and N. Sultani, \emph{Reeb dynamics of the link of the $A_n$ singularity}, Involve, Vol. 10, no. 3, Mathematical Sciences Publishers, 2017.
    
    \bibitem[BEHWZ03]{BEHWZ}  F. Bourgeois, Y. Eliashberg, H. Hofer, K. Wysocki, and E. Zehnder, \emph{Compactness results in symplectic field theory}, Geom. Topol. 7 (2003), 799-888.
    
    \bibitem[BO17]{BO} F. Bourgeois and A. Oancea, \emph{$S^1$-equivariant symplectic homology and linearized contact homology}, Int. Math. Res. Not. Volume 2017, Issue 13, July 2017, Pages 3849–3937.
    
    \bibitem[CH14]{CH} C. Cho and H. Hong, \emph{Orbifold Morse–Smale–Witten complexes}, Internat. J. Math. 25 (2014), no. 5, 1450040, 35 pp.
    
\bibitem[EGH00]{EGH} Y. Eliashberg, A. Givental and H. Hofer, {\em Introduction to symplectic field theory\/}, Geom. Funct. Anal. 2000, Special Volume, Part II, 560--673.    
    
    \bibitem[HM22]{HM} S. Haney and T. Mark, \emph{Cylindrical contact homology of 3-dimensional Brieskorn manifolds}, Algebr. Geom. Topol. 22 (2022), no. 1, 153--187.
    
    \bibitem[HWZ96]{hwz1} H. Hofer, K. Wysocki and E. Zehnder, {\em Properties of pseudoholomorphic curves in symplectisations. I. Asymptotics\/}, Ann. Inst. H. Poincare Anal. Non Lineaire {\bf 13} (1996), 337--379.

\bibitem[HWZ95]{hwz2} H. Hofer, K. Wysocki and E. Zehnder, {\em Properties of pseudo-holomorphic curves in symplectisations. II. Embedding controls and algebraic invariants\/}, Geom. Funct. Anal. {\bf 5} (1995), 270--328.

\bibitem[Hu02b]{index} M. Hutchings, {\em An index inequality for embedded pseudoholomorphic curves in symplectizations\/}, J. Eur. Math. Soc. {\bf 4} (2002), 313--361.


\bibitem[Hu09]{revisit} M. Hutchings, {\em The embedded contact homology index revisited\/}, New perspectives and challenges in symplectic field theory, 263--297, CRM Proc. Lecture Notes 49, Amer. Math. Soc., 2009.
    
    \bibitem[HN16]{HN} M. Hutchings and J. Nelson, \emph{Cylindrical contact homology for dynamically convex contact forms in three dimensions}, J. Symplectic Geom. Volume 14 (2016), no. 4, 983-1012.
    
    \bibitem[HN22]{HN2} M. Hutchings and J. Nelson, \emph{$S^1$-equivariant contact homology for hypertight contact forms}, J. Topology 2022, (15) 1455-1539. 

\bibitem[HT09]{ht2} M. Hutchings and C. H. Taubes, {\em Gluing pseudoholomorphic curves along branched covered cylinders II\/}, J. Symplectic Geom. {\bf 7} (2009), 29--133.

\bibitem[McK80]{mckay} J. McKay, \emph{Graphs, singularities, and finite groups},  The Santa Cruz Conference on Finite Groups, pp. 183-186, Proc. Symp. Pure Math. Vol. 37. No. 183. 1980, AMS.
    
    \bibitem[MR]{MR} M. McLean and A. Ritter \emph{The McKay correspondence via Floer theory}, J. Differential Geom. 124 (2023), no. 1, 113--168. 
        
    \bibitem[MS15]{MS} D. McDuff and D. Salamon, \emph{Introduction to symplectic topology}, 3rd ed., Oxford University Press, 2015.
    
    \bibitem[Ne15]{N1} J. Nelson, \emph{Automatic transversality in contact homology I: regularity}, Abh. Math. Semin. Univ. Hambg. 85 (2015), no. 2, 125-179.
    
    \bibitem[Ne20]{N2} J. Nelson, \emph{Automatic transversality in contact homology II: filtrations and computations}, Proc. Lond. Math. Soc. (3) 120 (2020), 853-917.

    \bibitem[Ne1]{N3} J. Nelson, \emph{A symplectic perspective of the simple singularities}, \url{https://math.rice.edu/~jkn3/nelson_simple.pdf}, Accessed 7/8/21.
    
    \bibitem[NW23]{NW} J. Nelson and M. Weiler, \emph{Embedded contact homology of prequantization bundles}, J. Symp. Geom. vol 21, issue 6, 2023, 1077-1189.


    \bibitem[NW24]{kech} J. Nelson and M. Weiler, \emph{Torus knot filtered embedded contact homology of the tight contact 3-sphere},   J. Topology 2024, (7) Issue 2, 74 pages.

    \bibitem[NW2]{NW2} J. Nelson and M. Weiler, \emph{Torus knotted Reeb dynamics and the Calabi invariant}, \href{https://arxiv.org/abs/2310.18307}{arXiv:2310.18307}
        
    \bibitem[Sal99]{S} D. Salamon, \emph{Lectures on Floer homology}, Symplectic geometry and topology (Park City, UT, 1997), volume 7 of IAS/Park City Math. Ser., 143–229. Amer. Math. Soc., Providence, RI,
1999.


\bibitem[SZ92]{SZ}D. Salamon and E. Zehnder,  \emph{Morse theory for periodic solutions of Hamiltonian systems and the Maslov index}, {Comm. Pure Appl. Math.} 45 (1992), no 10, 1303-1360. 

    
    \bibitem[Si08]{s1} R. Siefring, {\em Relative asymptotic behavior of pseudoholomorphic half-cylinders\/}, Pure Appl. Math. {\bf 61} (2008).

\bibitem[Si11]{s2} R. Siefring, {\em Intersection theory of punctured pseudoholomorphic curves\/}, Geom. Topol. {\bf 15} (2011), 2351--2457.


    \bibitem[Sl80]{Sl}  P. Slodowy, \emph{Simple Singularities and Simple Algebraic Groups}, Lecture Notes in Math. 815, Springer-Verlag, New York (1980).
    
    \bibitem[St85]{St}  R. Steinberg, \emph{Finite subgroups of} $SU_2$, \emph{dynkin diagrams, and affine coxeter elements}, Pacific Journal of Math. Vol. 118, No. 2, 1985.


\bibitem[Wen10]{W} C. Wendl, \emph{Automatic transversality and orbifolds of punctured holomorphic curves in dimension four}, Comment. Math. Helv. 85 (2010), no. 2, 347-407.

 \bibitem[We20]{wint} C. Wendl, {\em Lectures on contact 3-manifolds, holomorphic curves and intersection theory.} Cambridge Tracts in Mathematics, 220, 2020.

\bibitem[Wen-SFT]{wendl-sft} C. Wendl, {\em Lectures on Symplectic Field Theory}, to appear in EMS Lectures in Mathematics series,  \href{https://www2.mathematik.hu-berlin.de/~wendl/Sommer2020/SFT/lecturenotes.pdf}{link to version 3, 2021}; arXiv:1612.01009v2

\bibitem[Za58]{Z} H. Zassenhaus,  \emph{The Theory of Groups}, 2nd ed., Chelsea Publishing Company, New York, 1958.
\end{thebibliography}
\end{document}